\documentclass{amsart}
\usepackage{amssymb}

\newcommand{\map}[1]{\xrightarrow{#1}}
\newcommand{\iso}{\cong}

\newcommand{\Hom}{\mathrm{Hom}}
\newcommand{\Aut}{\mathrm{Aut}}
\newcommand{\End}{\mathrm{End}}
\newcommand{\Spec}{\mathrm{Spec}}
\newcommand{\Q}{\mathbb Q}
\newcommand{\Z}{\mathbb Z}
\newcommand{\R}{\mathbb R}
\newcommand{\C}{\mathbb C}
\newcommand{\F}{\mathbb F}
\newcommand{\A}{\mathbb A}
\newcommand{\co}{\mathcal O}
\newcommand{\alg}{\mathrm{alg}}
\newcommand{\ord}{\mathrm{ord}}
\newcommand{\length}{\mathrm{length}}
\newcommand{\Lie}{\mathrm{Lie}}
\newcommand{\Frob}{\mathrm{Fr}}

\newcommand{\action}{\bullet}

\newcommand{\CM}{\mathrm{CM}}
\newcommand{\can}{\mathrm{can}}
\newcommand{\green}{\mathtt{Gr}}

\input xy
\xyoption{all}

\author{Benjamin Howard}
\thanks{This research was supported in part by NSF grant DMS-0901753.}
\title{Complex multiplication cycles and  Kudla-Rapoport divisors}

\begin{document}

\begin{abstract}
We study the intersections of special cycles on  a unitary Shimura variety of signature $(n-1,1)$,
and show that the intersection multiplicities of these cycles agree with Fourier coefficients of Eisenstein series.
The results are new cases of conjectures of Kudla, and suggest a Gross-Zagier  theorem for
unitary Shimura varieties. 
\end{abstract}

\maketitle

\theoremstyle{plain}
\newtheorem{Thm}{Theorem}[subsection]
\newtheorem{Prop}[Thm]{Proposition}
\newtheorem{Lem}[Thm]{Lemma}
\newtheorem{Cor}[Thm]{Corollary}
\newtheorem{Conj}[Thm]{Conjecture}
\newtheorem{BigThm}{Theorem}
\newtheorem{BigCor}[BigThm]{Corollary}

\theoremstyle{definition}
\newtheorem{Def}[Thm]{Definition}
\newtheorem{BigHyp}[BigThm]{Hypothesis}

\theoremstyle{remark}
\newtheorem{Rem}[Thm]{Remark}
\newtheorem{Ques}[Thm]{Question}
\newtheorem{Hyp}[Thm]{Hypothesis}

\numberwithin{equation}{subsection}
\renewcommand{\theBigThm}{\Alph{BigThm}}
\renewcommand{\theBigCor}{\Alph{BigCor}}
\renewcommand{\theBigHyp}{\Alph{BigHyp}}



\section{Introduction}



\subsection{Overview}


In \cite{KRunitaryII}, Kudla and Rapoport define a family of divisors $Z(m)$
on a  unitary Shimura variety $M$ of dimension $n-1$, all defined
over a quadratic imaginary field $K_0$.  The variety and the divisors have integral models $\mathcal{M}$
and $\mathcal{Z}(m)$ over $\co_{K_0}$.  The program begun in  \cite{KRunitaryI, KRunitaryII, terstiege}
seeks to compute the $n$-fold intersection multiplicity of a tuple $\mathcal{Z}(m_1),\ldots,\mathcal{Z}(m_n)$, and
to relate the intersection multiplicity to Fourier coefficients of Eisenstein series.   
In this article we intersect the Kudla-Rapoport divisors with a different cycle on $\mathcal{M}$,
formed by points with complex multiplication.     By fixing  a CM field $K$
 of degree $n$ over $K_0$ and  a CM type $\Phi$
 satisfying a suitable signature condition, we obtain a $0$-cycle $X_{\Phi}$ on $M$ defined
over the reflex field of $\Phi$, representing points with complex multiplication by $\co_K$
and  CM type $\Phi$.   Passing to integral models yields a cycle $\mathcal{X}_\Phi$ on $\mathcal{M}$
of absolute dimension $1$, and our main results relate the intersection  multiplicity of $\mathcal{Z}(m)$ 
and $\mathcal{X}_\Phi$ with Fourier coefficients of an Eisenstein series.

 The  intersection $\mathcal{Z}(m) \cap \mathcal{X}_\Phi$ naturally  decomposes as a disjoint union of 
 $0$-dimensional stacks   $ \mathcal{Z}_{\Phi}(\alpha)$, where the index $\alpha$ ranges over those totally 
 positive elements of the maximal totally real subfield 
 $F\subset K$ which satisfy $\mathrm{Tr}_{F/\Q}(\alpha)=m$.  In the body of the paper we allow $K$ to be a product of CM fields, 
 in which case some $\mathcal{Z}_{\Phi}(\alpha)$ may have dimension one, \emph{i.e.}~the cycles 
 $\mathcal{X}_\Phi$ and $\mathcal{Z}(m)$ may intersect improperly.  This does not happen when $K$ is 
 a field.  The \emph{Arakelov degree} of $\mathcal{Z}_{\Phi}(\alpha)$ is (essentially) defined
 to be the sum of the lengths of the local rings of all geometric points, 
 and our main result shows that, as $\alpha$ varies, 
 these degrees are the Fourier coefficients of the derivative of a weight one Hilbert modular Eisenstein 
 series  $\mathcal{E}_{\Phi}(\tau,s)$ at the  center $s=0$ of its functional equation.
 
 Returning to the original problem,
 the intersection multiplicity of $\mathcal{X}_\Phi$ with $\mathcal{Z}(m)$ is obtained by adding 
 together the  degrees of those $\mathcal{Z}_{\Phi}(\alpha)$ with $\mathrm{Tr}_{F/\Q}(\alpha)=m$.  This 
 intersection multiplicity is equal to the $m^\mathrm{th}$ Fourier coefficient of the central derivative
 of   the pullback of $\mathcal{E}_{\Phi}(\tau,s)$ via the 
 diagonal embedding of the upper half plane into a product of upper half planes.


\subsection{Statement of the results}


Fix  a quadratic imaginary field $K_0\subset \C$,  denote by $\iota$ the inclusion $K_0 \to \C$, 
and let $\overline{\iota}$ be the conjugate embedding.   For  nonnegative integers $r,s$ 
let $\mathcal{M}_{(r,s)}$ be the algebraic  stack over $\Spec(\co_{K_0})$ whose functor of points assigns to 
every $\co_{K_0}$-scheme $S$  the groupoid of triples $(A,\kappa,\lambda)$ in which 
 \begin{itemize}
 \item
 $A\to S$ is an abelian scheme of relative dimension $r+s$,
\item
$\kappa:\co_{K_0} \to \End(A)$ is an action of $\co_{K_0}$ on $A$, 
\item
$\lambda:A\to A^\vee$ is a principal polarization
\end{itemize}
(throughout this paper \emph{scheme}  means \emph{locally Noetherian scheme}, and \emph{algebraic stack}
means \emph{Deligne-Mumford stack}). We require  that  the polarization $\lambda$ be  $\co_{K_0}$-linear, 
in the sense that
\[
\lambda\circ \kappa(\overline{x}) = \kappa(x)^\vee \circ \lambda
\]
for all $x\in \co_{K_0}$. We further require  that the action of $\co_{K_0}$ satisfy the 
$(r,s)$-\emph{signature condition}:  for any $x \in \co_{K_0}$, locally on $S$ the determinant of 
$T-x $ acting on $\mathrm{Lie}(A)$ is equal to the image of 
\[
(T -  \iota(x) )^r  (T - \overline{\iota}(x)   )^s \in \co_{K_0}[T]
\]
in $\co_S[T]$.  Our stack $\mathcal{M}_{(r,s)}$ is the stack denoted $\mathcal{M}(r,s)^\mathrm{naive}$
in \cite{KRunitaryII}; it is smooth of relative dimension $rs$ over $\co_{K_0}[\mathrm{disc}(K_0)^{-1}]$.
The generic fiber of $\mathcal{M}_{(r,s)}$ is a union of Shimura varieties associated to the unitary similitude
groups of finitely many Hermitian spaces over $K_0$, but for us the interpretation as a moduli space
is paramount. 

Note that $\mathcal{M}_{(1,0)}$ is simply the moduli stack of elliptic curves 
$A_0\to S$ over $\co_{K_0}$-schemes with
complex multiplication by $\co_{K_0}$, normalized so that the action of $\co_{K_0}$ on $\Lie(A_0)$ 
is through the structure morphism $\co_{K_0}\to \co_S$.

For the remainder of the introduction we fix a positive integer $n$ and focus on
the case of signature $(n-1,1)$. We will construct two types of cycles on the stack
\[
\mathcal{M} = \mathcal{M}_{(1,0)} \times_{\co_{K_0} }  \mathcal{M}_{(n-1,1)}.
\]
For an $\co_{K_0}$-scheme $S$, an $S$-valued point of $ \mathcal{M}$
 is a sextuple  $(A_0,\kappa_0,\lambda_0,A,\kappa,\lambda)$ with
\[
(A_0,\kappa_0,\lambda_0)\in \mathcal{M}_{(1,0)}(S) \qquad 
(A,\kappa,\lambda)\in \mathcal{M}_{(n-1,1)}(S)
\] 
but we will usually  abbreviate this sextuple to $(A_0,A)$.

The first  family of cycles on $\mathcal{M}$   are the  \emph{Kudla-Rapoport divisors} of   \cite{KRunitaryII}.
If $S$ is connected and $(A_0,A)\in \mathcal{M}(S)$, the projective $\co_{K_0}$-module of finite rank
\[
L(A_0,A)=\Hom_{\co_{K_0}}(A_{0} , A)
\]
comes equipped with a positive definite $\co_{K_0}$-valued Hermitian form
\begin{equation}\label{hermitian def}
\langle f_1, f_2\rangle = \lambda_0^{-1} \circ f_2^\vee \circ \lambda \circ f_1 
\end{equation}
(the right hand side lies in $\co_{K_0}=\End_{\co_{K_0}} (A_0)$).
For an integer  $m\not=0$ let  $\mathcal{Z}(m)$  be the moduli stack  over $\co_{K_0}$ whose 
$S$-valued points are   triples $(A_0,A,f)$ with  $(A_0,A)\in \mathcal{M}(S)$ and  $f\in L(A_0,A)$ satisfying 
$\langle f,f\rangle=m$.    There is an obvious forgetful morphism 
 \[
 \mathcal{Z}(m) \to  \mathcal{M}.
 \]
 In terms of Shimura varieties, these divisors correspond roughly to inclusions of algebraic groups 
 of the form $\mathrm{H} \to  \mathrm{GU}(V)$, where  $V$ is a  Hermitian space
 over $K_0$ of signature $(n-1,1)$, and $\mathrm{H}$ is the stabilizer of a vector of positive length.  
 But, once again, to us it is the moduli 
 interpretation that matters most. 

The second type of cycle is constructed from abelian varieties with complex multiplication.
Let $F$ be a totally real \'etale $\Q$-algebra (in other words, a product of totally real number fields)
with $[F:\Q]=n$, and    fix a CM type $\Phi$ of 
\[
K=F\otimes_\Q K_0
\]
 of signature $(n-1,1)$.  This means that
there are  $n-1$ elements of $\Phi$ whose restriction to $K_0$ is  $\iota$, and a unique element
  whose restriction to  $K_0$ is $\overline{\iota}$.  Let 
  \[
  K_\Phi\subset \C
  \] 
  be a number field containing both  $K_0$ and the reflex field of $\Phi$, and set $\co_\Phi=\co_{K_\Phi}$.
 Let $\mathcal{CM}_{\Phi}$  be the algebraic stack  over $\co_\Phi$ classifying principally 
 polarized abelian schemes  with complex multiplication by 
 $\co_K$ and CM type $\Phi$.  See Section \ref{ss:moduli} for the precise definition.
For an $\co_\Phi$-scheme $S$, an $S$-valued point of
\[
\mathcal{X}_{\Phi}
 = \mathcal{M}_{(1,0)/\co_\Phi} \times_{ \co_{\Phi} } \mathcal{CM}_{\Phi}.
\] 
 is a pair $(A_0,A)\in \mathcal{M}(S)$ 
together with an extension of the  $\co_{K_0}$-action on $A$  to complex multiplication by $\co_K$,
and as such there is an evident forgetful morphism
\[
\mathcal{X}_{\Phi}  \to  \mathcal{M}_{/\co_\Phi}.
\]
The stack  $\mathcal{X}_\Phi$ is  \'etale and proper over $\co_\Phi$, 
and  in particular is  regular  of dimension $1$.  
In terms of Shimura varieties, the map $\mathcal{X}_\Phi \to \mathcal{M}_{/\co_\Phi}$ 
corresponds roughly to $T\to \mathrm{GU}(V)$, where
$V$ is a Hermitian space over $K_0$ of signature $(n-1,1)$, and $T$ is the torus with
$\Q$-points 
\[
T(\Q) = \{ x\in K^\times :  \mathrm{Nm}_{K/F}(x) \in \Q^\times \}.
\]
 
Now we come to the central problem of this paper: to compute the intersection multiplicity on 
$\mathcal{M}_{/  \co_\Phi}$ of the  Kudla-Rapoport divisor $\mathcal{Z}(m)_{/\co_\Phi}$ 
with the complex multiplication cycle $\mathcal{X}_{\Phi}.$
Consider the cartesian diagram (this is the definition of the upper left corner)
\[
\xymatrix{
 {  \mathcal{X}_{\Phi}  \cap \mathcal{Z}(m) } \ar[d]  \ar[r]  & {\mathcal{Z}(m)_{/\co_\Phi} } \ar[d] \\
{ \mathcal{X}_{\Phi} } \ar[r] &   { \mathcal{M}_{/\co_\Phi}}.
}
\]
Let $S$ be a connected $\co_\Phi$-scheme. Given a point
\[
(A_0,A)\in  \mathcal{X}_{\Phi}  (S),
\]
we may consider the $\co_{K_0}$-module $L(A_0,A)$ attached to the image $(A_0,A)\in \mathcal{M}(S)$.
The fact that the pair $(A_0,A)$ comes from $ \mathcal{X}_{\Phi} (S)$ endows $L(A_0,A)$
with obvious extra structure:  the action of $\co_K$ on $A$ makes $L(A_0,A)$ into an $\co_K$-module.
Slightly less obviously,  there is a unique $K$-valued totally positive definite
$\co_K$-Hermitian form $\langle f_1,f_2\rangle_\CM$ on $L(A_0,A)$, which refines $\langle f_1,f_2\rangle$, in the 
sense that 
\[
\langle f_1,f_2\rangle = \mathrm{Tr}_{K/K_0} \langle f_1,f_2\rangle_\CM.
\]

By contemplation of the the moduli problems, there is  a  decomposition 
\begin{equation}\label{moduli decomp}
 \mathcal{X}_{\Phi}  \cap \mathcal{Z}(m)
=\bigsqcup_{ \substack{ \alpha\in F \\ \mathrm{Tr}_{F/\Q} (\alpha) = m} } \mathcal{Z}_{\Phi}(\alpha),
\end{equation}
where  $\mathcal{Z}_{\Phi}(\alpha)$ is the moduli space of triples  $(A_0,A,f)$ over $\co_\Phi$-schemes $S$,
in which 
\[
(A_0,A)\in   \mathcal{X}_{\Phi}(S)
\]
and    $f\in L(A_0,A)$ satisfies $\langle f,f\rangle_\CM =\alpha$.   
  If $\alpha\in F^\times$ the stack $\mathcal{Z}_{\Phi}(\alpha)$ has dimension $0$, and is nonempty only
 if $\alpha$ is totally positive ($\alpha\gg 0$).   If $\alpha\not\in F^\times$ then  $\mathcal{Z}_{\Phi}(\alpha)$ may 
have irreducible components of  dimension $1$, in which case the intersection (\ref{moduli decomp}) is improper.

For a prime $\mathfrak{p}$ of $K_\Phi$, let  $k_{\Phi,\mathfrak{p}}$ denote the residue field of $\mathfrak{p}$.
When $\mathcal{Z}_{\Phi}(\alpha)$ has dimension $0$, its  \emph{Arakelov degree}
\[
\widehat{\deg}\, \mathcal{Z}_{\Phi}(\alpha) = \sum_{ \mathfrak{p}\subset\co_\Phi }\frac{ \log( \mathrm{N}(\mathfrak{p}))  }{[K_\Phi : \Q]}
\sum_{z\in \mathcal{Z}_{\Phi}( \alpha ) (k^\alg_{\Phi,\mathfrak{p}})  }
\frac{\mathrm{length}(\co^\mathrm{sh}_{\mathcal{Z}_{\Phi}(\alpha),z})}  {\#\Aut(z)}
\]
is finite, and is independent of the choice of $K_\Phi$ (here $\co^\mathrm{sh}_{\mathcal{Z}_{\Phi}(\alpha),z}$
is the strictly Henselian local ring of $\mathcal{Z}_{\Phi}(\alpha)$ at $z$, \emph{i.e}.~the local ring 
for the \'etale topology).
Our first main result is a formula for the Arakelov degree.  To state it we need some notation.
Let $\varphi^\mathrm{sp}\in \Phi$ be the \emph{special element}, determined by 
$\varphi^\mathrm{sp}|_{K_0} = \overline{\iota}$. Recalling that $K$ is a product of CM fields, there is a unique factor 
$K^\mathrm{sp} \subset K$ such that  $\varphi^\mathrm{sp}:K\to \C$ factors through the projection 
$K\to K^\mathrm{sp}$.  Denote by $F^\mathrm{sp}$ the  maximal totally real subfield of 
$K^\mathrm{sp}$, so that $F^\mathrm{sp}$ is a  direct summand of $F$.
   If $\mathfrak{p}$ is a prime of $F^\mathrm{sp}$ then we denote again by $\mathfrak{p}$ the
prime of $F$ determined by pullback through the projection $F\to F^\mathrm{sp}$.
If $\mathfrak{b}$ is a fractional $\co_F$-ideal define
\begin{equation}\label{rho}
\rho(\mathfrak{b}) = \# \{ \mathfrak{B}\subset\co_K : \mathfrak{B}\overline{\mathfrak{B}} = \mathfrak{b} \co_K\}.
\end{equation}
In particular $\rho(\mathfrak{b}) =0 $ if $\mathfrak{b} \not\subset\co_F$.
For any prime $p$ set
\begin{equation}\label{epsilon}
\epsilon_p=\begin{cases}
1 & \hbox{if $K_0/\Q$ is unramified at $p$}\\
0&  \hbox{if $K_0/\Q$ is ramified at $p$.}
\end{cases}
\end{equation}
The following theorem appears in the text as Theorem \ref{Thm:zero cycle degree}.

\begin{BigThm}\label{Main Theorem I}
Assume  the discriminants of $K_0/\Q$ and $F/\Q$ are odd and relatively prime. 
 If $\alpha\in F^{\gg 0}$  then  $\mathcal{Z}_{\Phi}(\alpha)$ has dimension zero, and 
\[
\widehat{\deg}\, \mathcal{Z}_{\Phi}(\alpha)  =   \frac{ h(K_0)}{   w(K_0)  }
\sum_{\mathfrak{p}}   \frac{ \log(\mathrm{N}(\mathfrak{p})) }{  [K^\mathrm{sp} :\Q ] } 
 \cdot \ord_{\mathfrak{p}} (\alpha \mathfrak{p}\mathfrak{d}_F  )
\cdot   \rho ( \alpha\mathfrak{p}^{-\epsilon_p}\mathfrak{d}_F )  
\]
where the sum is over all  primes $\mathfrak{p}$ of $F^\mathrm{sp}$ nonsplit in $K^\mathrm{sp}$, $p$
is the rational prime below $\mathfrak{p}$,  $\mathfrak{d}_F$ is the different of $F/\Q$,  $h(K_0)$ is the 
class number of $K_0$,  $w(K_0)$ is number of roots of unity in $K_0$, and $\mathrm{N}(\mathfrak{p})$
is the cardinality of the residue field of $\mathfrak{p}$.
\end{BigThm}

Fix a prime $\mathfrak{p}\subset \co_\Phi$, and let $W_{\Phi,\mathfrak{p}}$ be the 
completion of the ring of integers in the maximal unramified extension of $\co_{\Phi,\mathfrak{p}}$.
The most difficult part of the proof of Theorem \ref{Main Theorem I} is the calculation of the length of the local 
ring at a geometric point $z\in \mathcal{Z}_\Phi(\alpha)(k^\alg_{\Phi,\mathfrak{p}}  )$, 
corresponding to a triple $(A_0,A,f)$. This calculation proceeds in 
two steps.  First we show that the formal deformation space of the pair $(A_0,A)$ 
is  isomorphic to the formal spectrum of $W_{\Phi,\mathfrak{p}}$.  In more concrete 
terms, this means that $(A_0,A)$ lifts uniquely to any complete local Noetherian 
$W_{\Phi,\mathfrak{p}}$-algebra with residue field $k^\alg_{\Phi,\mathfrak{p}}$, and 
in particular has a unique lift to $W_{\Phi,\mathfrak{p}}$ called the \emph{canonical lift}.
Let $(A_0^{(k)},A^{(k)})$ be the reduction of  the canonical lift to the
quotient $W_{\Phi,\mathfrak{p}} /\mathfrak{m}^k$,
where $\mathfrak{m}$ is the maximal ideal of $W_{\Phi,\mathfrak{p}}$.  The length of the 
local ring of $\mathcal{Z}_\Phi(\alpha)$ at $z$ is then equal to the largest $k$ such that 
$f$ lifts to a map $ A_0^{(k)} \to A^{(k)}$.  In other words, the $\co_K$-module $L(A_0,A)$ has
a filtration 
\[
\dots\subset L^{(3)}\subset L^{(2)} \subset L^{(1)} =L(A_0,A)
\]
in which 
\[
L^{(k)} = \Hom_{\co_{K_0}} (A_0^{(k)}, A^{(k)}),
\]
and the problem is to compute the largest $k$ such that $f\in L^{(k)}$.  We show that this $k$
is 
\[
k=\frac{1}{2} \cdot e_\mathfrak{p} \cdot \ord_{\mathfrak{p}_F} (\alpha\mathfrak{p}_F\mathfrak{d}_F),
\]
where $\mathfrak{p}_F$ is the pullback of $\mathfrak{p}$ under the map $\varphi^\mathrm{sp}:F\map{}K_\Phi$,
and $e_\mathfrak{p}$ is the ramification degree of $\mathfrak{p}_F$ in $K$.
All of this is done in Section \ref{ss:LHI}, using the Grothendieck-Messing deformation theory of $p$-divisible groups,
with the final application to lengths of local rings appearing as Theorem \ref{Thm:local length}.
In the case $n=1$, so that $A_0$ and $A$ are supersingular elliptic curves, all of these calculations reduce to calculations of
Gross, as explained at the end of Section \ref{ss:LHI}.

  In Section \ref{s:eisenstein} we construct a Hilbert modular Eisenstein series 
$\mathcal{E}_{\Phi}(\tau,s)$ of parallel weight one.  
The Eisenstein series $\mathcal{E}_{\Phi}(\tau,s)$ satisfies a function equation 
in $s\mapsto -s$ which forces $\mathcal{E}_{\Phi}(\tau,0)=0$, and the central derivative
has a Fourier expansion 
\[
 \mathcal{E}_{\Phi}'(\tau,0)   =\sum_{\alpha\in F} b_{\Phi}(\alpha,y)\cdot  q^\alpha,
\]
where $\tau=x+iy\in \mathcal{H}^n$ lies in the product of $n$ upper half planes.
The Fourier coefficients $b_{\Phi}(\alpha,y)$ can easily be computed using explicit formulas of Yang
\cite{yang05}, and the result is stated as Corollary  \ref{Cor:final fourier}.   Comparison with 
Theorem \ref{Main Theorem I} shows that, for $\alpha\in F^{\gg 0}$
\begin{equation}\label{naive degree formula}
\widehat{\mathrm{deg}}\, \mathcal{Z}_{\Phi}(\alpha)
 =       -   \frac{ h(K_0) } {w(K_0) } \cdot   \frac{  \sqrt{\mathrm{N}( d_{K/F}) }  } { 2^{r -1 }    [K^\mathrm{sp}:\Q]}
 \cdot  b_{\Phi}(\alpha,y)
\end{equation}
where $d_{K/F}$ is the relative discriminant of $K/F$, and $r$ is the number of places of $F$ ramified in 
$K$, including the archimedean places.  In particular the right hand side is independent of $y$.

Of course the right hand side of (\ref{naive degree formula}) makes sense for all $\alpha\in F$, while at the moment the 
left hand side is only defined for $\alpha\gg 0$.    To remedy this asymmetry  we introduce in 
Section \ref{ss:divisors} the Gillet-Soul\'e \emph{arithmetic Chow group} $\widehat{\mathrm{CH}}^1(\mathcal{X}_\Phi)$
of the $1$-dimensional stack $\mathcal{X}_\Phi$.  Elements of the arithmetic Chow group are rational
equivalence classes of pairs $(\mathtt{Z},\green)$ where $\mathtt{Z}$ is a $0$-cycle on $\mathcal{X}_\Phi$
with rational coefficients, and $\green$ is a Green function for $\mathtt{Z}$. As $\mathtt{Z}$ has no points
in characteristic $0$, $\green$ is just a function on the finite set of complex points of $\mathcal{X}_\Phi$.  
In Section \ref{ss:divisors} we construct a divisor class
\[
\widehat{\mathtt{Z}}_\Phi(\alpha,y) \in \widehat{\mathrm{CH}}^1(\mathcal{X}_\Phi)
\]
for every $\alpha\in F^\times$ and every $y\in F_\R^{\gg 0}$.  If $\alpha\gg 0$ then 
this class is $(\mathtt{Z}_\Phi(\alpha),0)$, where the $0$-cycle $\mathtt{Z}_\Phi(\alpha)$ is
the image of the map $\mathcal{Z}_\Phi(\alpha) \to  \mathcal{X}_\Phi$, with points counted with appropriate
multiplicities.  
 If $\alpha\not\gg 0$ then our divisor class has the form $(0,\green_\Phi(\alpha,y,\cdot))$
for a particular function $\green_\Phi(\alpha,y,\cdot)$ on the complex points of $\mathcal{X}_\Phi$.
There is a canonical \emph{arithmetic degree}
\begin{equation}\label{intro deg}
\widehat{\deg}  : \widehat{\mathrm{CH}}^1(\mathcal{X}_\Phi) \to \R,
\end{equation}
and (\ref{naive degree formula}) has the following generalization, which appears in the text as 
Theorem \ref{Thm:degree-fourier}.

\begin{BigThm}\label{Main Theorem II} 
Assume  the discriminants of $K_0/\Q$ and $F/\Q$ are odd and relatively prime. 
If $\alpha\in F^\times$ and $y\in F_\R^{\gg 0}$ then
\[
\widehat{\mathrm{deg}}\, \widehat{\mathtt{Z}}_{\Phi}(\alpha,y)
 =       -   \frac{ h(K_0) } {w(K_0) } \cdot   \frac{  \sqrt{\mathrm{N}( d_{K/F}) }  } { 2^{r -1 }    [K^\mathrm{sp}:\Q]}
 \cdot  b_{\Phi}(\alpha,y) .
\]
\end{BigThm}

We now return to our original motivating problem: the calculation of the intersection multiplicity of 
$\mathcal{X}_\Phi$ and $\mathcal{Z}(m)$ on $\mathcal{M}$.  Assume that $m\not=0$, and that $F$ is 
a field.  This guarantees that $\mathcal{X}_\Phi \cap \mathcal{Z}(m)$ is $0$ dimensional.  The 
intersection multiplicity $I( \mathcal{X}_\Phi : \mathcal{Z}(m))$ is defined as the 
Arakelov degree of the $0$-dimensional stack $\mathcal{X}_\Phi \cap \mathcal{Z}(m)$.  It is a more or less
formal consequence of (\ref{moduli decomp}), see Theorem \ref{Thm:fundamental decomp I}, that
\begin{equation}\label{intro finite}
I( \mathcal{X}_\Phi : \mathcal{Z}(m)) = \sum_{  \substack{  \alpha\in F^\times, \alpha\gg 0 \\ \mathrm{Tr}_{F/\Q}(\alpha)=m  }  }
\widehat{\mathrm{deg}}\, \widehat{\mathtt{Z}}_{\Phi}(\alpha,y)
\end{equation}
for all $y\in \R^{>0}$.   In Section \ref{ss:divisors II} we define a Green function $\green(m,y,\cdot)$
for the Kudla-Rapoport divisor $\mathcal{Z}(m)$.  
It is a smooth function on $\mathcal{M}(\C)$, except for  
a logarithmic singularity along the divisor $\mathcal{Z}(m)(\C)$, and depends on a 
parameter $y\in \R^{>0}$.  If $m<0$ then $\mathcal{Z}(m) =\emptyset$, and $\green (m,y,\cdot)$
is simply a smooth function on $\mathcal{M}(\C)$.
This function may be evaluated at the finite set of complex points of $\mathcal{X}_\Phi$, and the result,
Theorem \ref{Thm:fundamental decomp II}, is 
\begin{equation}\label{intro arch}
\green(m,y, \mathcal{X}_\Phi) =
 \sum_{  \substack{  \alpha\in F^\times, \alpha\not \gg 0 \\ \mathrm{Tr}_{F/\Q}(\alpha)=m  }  }
\widehat{\mathrm{deg}}\, \widehat{\mathtt{Z}}_{\Phi}(\alpha,y).
\end{equation}

Let  $i_F:\mathcal{H}\to \mathcal{H}^n$ be the diagonal embedding of the  upper half plane.
The restriction $\mathcal{E}_{\Phi}(i_F(\tau),s)$  of $\mathcal{E}_{\Phi} (\tau,s)$
to $\mathcal{H}$ vanishes at $s=0$, and the derivative has a Fourier expansion
\[
\mathcal{E}'_{\Phi}(i_F(\tau),0) = \sum_{m\in \Z} c_{\Phi}(m,y)  \cdot  q^m
\]
where now $\tau=x+iy\in \mathcal{H}$, and 
\[
 c_{\Phi}(m,y)  =  \sum_{  \substack{  \alpha\in F^\times \\ \mathrm{Tr}_{F/\Q}(\alpha)=m  }  } b_{\Phi}(m,y).
\]
Combining this with Theorem \ref{Main Theorem II} and the decompositions
(\ref{intro finite}) and (\ref{intro arch})  gives an arithmetic interpretation of the Fourier coefficients $c_{\Phi}(m,y)$.

\begin{BigThm}\label{Main Theorem III}
Assume  the discriminants of $K_0/\Q$ and $F/\Q$ are odd and relatively prime. 
If $F$ is a field and $m$ is  nonzero  then
\[
I(\mathcal{X}_\Phi : \mathcal{Z}(m) )  +  \green(m,y, \mathcal{X}_\Phi ) =
 -   \frac{ h(K_0) }{w(K_0) } \cdot 
 \frac{  \sqrt{\mathrm{N}( d_{K/F}) }  } { 2^{r -1 } [K:\Q]} \cdot  c_{\Phi}(m,y) 
\]
for all $y\in \R^{>0}$.
\end{BigThm}

When $n=1$ or $2$ our results have precedents in the literature, albeit in very different language.
The case $n=1$ is essentially treated by Kudla-Rapoport-Yang in \cite{kudla99b}.
In this case   $F=\Q$,  $\mathcal{Z}(m)$ is a divisor on the $1$-dimensional stack 
\[
\mathcal{M} = \mathcal{M}_{(1,0)} \times_{\co_{K_0}}  \mathcal{M}_{(0,1)},
\]
 and $\mathcal{X}_\Phi = \mathcal{M}$.  Is this degenerate case the intersection $I(\mathcal{X}_\Phi : \mathcal{Z}(m))$
 is simply the Arakelov degree of the $0$-cycle $\mathcal{Z}(m)$, which is not quite what 
 is computed in \cite{kudla99b}.  For every triple $(A_0,\kappa_0,\lambda_0)$ in $\mathcal{M}_{(1,0)}$
 there is a conjugate triple $(A_0,\overline{\kappa}_0, \lambda_0)$, where 
 $\overline{\kappa}_0(x) = \kappa_0(\overline{x})$.
 The functor taking a triple to its conjugate defines an isomorphism $\mathcal{M}_{(1,0)} \to \mathcal{M}_{(0,1)}$,
 which allows us to define the  substack $\mathcal{M}^\Delta\to \mathcal{M}$ as the image of 
 $\mathcal{M}_{(1,0)}$ under the diagonal embedding.  The intersection 
 $\mathcal{Z}^\Delta(m)= \mathcal{Z}(m) \cap \mathcal{M}^\Delta$ is then the moduli space of triples $(E,\kappa, f)$
 where $E$ is an elliptic curve, $\kappa : \co_{K_0} \to \End(E)$
 is an action of $\co_{K_0}$ (suitable normalized), and $f\in\End(E)$ is a degree $m$ isogeny satisfying
 $\kappa(x)\circ f= f\circ \kappa(\overline{x})$ for all $x\in \co_{K_0}$.  It is the Arakelov degree
 of $\mathcal{Z}^\Delta(m)$ which is computed in  \cite{kudla99b}, and is shown to agree
 with the Fourier coefficients of the central derivative of a weight one Eisenstein series.

When $n=2$, so $F$ is either $\Q\times\Q$ or a real quadratic field,  our results are closely related to the work of 
Gross-Zagier on  prime factorizations of singular moduli \cite{gross-zagier85}, and heights of 
Heegner points \cite{gross-zagier}.   In this case the moduli space $\mathcal{M}$
is a union of Shimura varieties attached to groups of type $\mathrm{GU}(1,1)$.  Such Shimura varieties
are, roughly, unions of Shimura curves parametrizing abelian surfaces with quaternionic multiplication, including
 the classical modular curves.  This is worked out in detail in \cite[Section 14]{KRunitaryII}.
 The author has not worked out carefully the translation of the results of this paper into the language
 of moduli of elliptic curves, but the picture should look roughly like this.
 Both cycles $\mathcal{X}_\Phi$ and $\mathcal{Z}(m)$ are divisors
 on  $\mathcal{M}$ representing points with complex multiplication, \emph{i.e.~}Heegner points.
  In the case where $F$ is a real quadratic field, the 
 compositum $K=K_0\cdot F$ contains another quadratic imaginary field $K_1$, and 
 $\mathcal{X}_\Phi$ is the divisor formed by elliptic curves with complex multiplication by $\co_{K_1}$.
 The divisor $\mathcal{Z}(1)$ is formed by elliptic curves with complex multiplication by $\co_{K_0}$,
 and $\mathcal{Z}(m)$ is the translate of $\mathcal{Z}(1)$ by the   $m^\mathrm{th}$ Hecke
 correspondence.  The calculation of the intersection multiplicity $I(\mathcal{X}_\Phi,\mathcal{Z}(m))$, which amounts
 to computing congruences between values of the $j$-function at CM points, and the observation that these
 intersection multiplicities appear as the 
 Fourier coefficients of the diagonal restriction of a Hilbert modular Eisenstein series, is the content of
  \cite{gross-zagier85}.  See also  \cite{howard-yangA}.  
 
If $n=2$ and $F=\Q\times\Q$ then our calculations should be closely related to the more famous result
of Gross-Zagier \cite{gross-zagier}.  In this case the calculation of $I(\mathcal{X}_\Phi,\mathcal{Z}(m))$
amounts to the calculation of the intersection multiplicity, on the modular curve, of the divisor 
of elliptic curves with complex multiplication by $\co_{K_0}$ with the same divisor translated by the $m^\mathrm{th}$
Hecke correspondence.  This is the key calculation performed in \cite{gross-zagier}, although
those authors deal with instances of improper intersection (and other serious complications), 
which we have avoided.  Our results 
assert that these intersections agree with the Fourier coefficients of the central derivative 
of the diagonal pullback of an Eisenstein series on $\mathrm{GL}_2\times\mathrm{GL}_2$; that is to say, the 
derivative of a \emph{product} of two weight one Eisenstein series on $\mathcal{H}$, say $E_1(\tau,s)E_2(\tau,s)$.  
On the other hand, the results of Gross-Zagier assert that these same  intersection multiplicities are the Fourier coefficients
of the product of the central derivative of an Eisenstein series with a weight one theta series.  
One of our Eisenstein series, say $E_1(\tau,s)$, vanishes at $s=0$, while the other does not,
and so the central derivative of the product is $E'_1(\tau,0)\cdot E_2(\tau,0)$.  But the 
Siegel-Weil formula then asserts that the central value $E_2(\tau,0)$ is actually a weight one
theta series, so our results are compatible with those of \cite{gross-zagier}.


\subsection{Speculation}
\label{ss:speculation}


We would like to interpret Theorem \ref{Main Theorem III} in terms of the arithmetic intersection 
theory of Gillet-Soul\'e \cite{gillet09,gillet-soule90,soule92}.    Let $\widehat{\mathrm{CH}}^1(\mathcal{M})$
be the codimension one arithmetic Chow group, so that $\mathcal{Z}(m)$ (now viewed as 
a divisor on $\mathcal{M}$), with its Green function $\green(m,y,\cdot)$, defines a class
\begin{equation}\label{arith cycle}
\widehat{\mathtt{Z}}(m,y) \in \widehat{\mathrm{CH}}^1(\mathcal{M})
\end{equation}
for every $m\not=0$ and $y\in \R^+$.
The composition of pull-back by $\mathcal{X}_\Phi \to  \mathcal{M}_{/\co_\Phi}$ with 
(\ref{intro deg}) defines a linear functional
\[
 \widehat{\mathrm{CH}}^1(\mathcal{M}_{/\co_\Phi}) \to   \widehat{\mathrm{CH}}^1(\mathcal{X}_\Phi)
  \to \R.
\]
Composing with  base change from $\co_{K_0}$ to $\co_\Phi$, we obtain a linear functional
\[
\widehat{\deg}_{\mathcal{X}_\Phi} : \widehat{\mathrm{CH}}^1(\mathcal{M}) \to \R,
\]
called the \emph{arithmetic degree along $\mathcal{X}_\Phi$}.  
What our Theorem \ref{Main Theorem III} essentially shows is that
(ignoring the uninteresting constants appearing in the theorem)
\begin{equation}\label{arithmetic pullback}
\widehat{\deg}_{\mathcal{X}_\Phi}  \widehat{\mathtt{Z}}(m,y)  =  c_{\Phi}(m,y).
\end{equation}

There are several  gaps in the above interpretation of Theorem \ref{Main Theorem III}: 
 to have a good theory of arithmetic Chow groups one needs to work on a stack that is flat, regular, and proper.
 The stack $\mathcal{M}$ has none of these properties.   The stack $\mathcal{M}$
 is only flat and regular after inverting $\mathrm{disc}(K_0)$, but Pappas \cite{pappas00} and Kr\"amer 
 \cite{kramer} have
  modified the moduli problem defining $\mathcal{M}$ in order to obtain a flat and regular moduli stack
  which agrees with $\mathcal{M}$ over $\co_{K_0}[\mathrm{disc}(K_0)^{-1}]$.  As for properness,
 the theory of toroidal compactifications of the complex fiber of $\mathcal{M}$ is well 
understood \cite{AMRT,cogdell85,miller09}, and Lan's thesis \cite{kw-lan} gives a complete theory of the arithmetic
toroidal compactifications of $\mathcal{M}$  over $\co_{K_0}[\mathrm{disc}(K_0)^{-1}]$.  
See also \cite{larsen} for the signature $(2,1)$ case.  It seems  likely that the results of Lan's thesis 
can be extended to give a compactification of the integral model of Pappas and Kr\"amer, but the 
details have not been written down.  In any case, let us suppose that we have replaced $\mathcal{M}$
by a stack that is  flat, regular, and proper.  

In order to define the class $\widehat{\mathtt{Z}}(m,y)$ for $m\not=0$,
 one needs to understand the 
behavior of the Green function $\green(m,y,\cdot)$ near the boundary components of the 
newly compactified $\mathcal{M}$.
Some preliminary calculations suggest that if one adds a particular linear combination of boundary 
components to the divisor $\mathcal{Z}(m)$, the function $\green(m,y,\cdot)$ 
becomes a Green function for the modified
divisor, but with $\log$-$\log$ singularities along the boundary.  Thus one expects to obtain a class 
 in the generalized arithmetic Chow group of Burgos-Gil--Kramer--K\"uhn \cite{bruinier-burgos-kuhn,BKK}.
Assume this is the case, so that (\ref{arith cycle}) is defined for all $m\not=0$.

The next step is to define the class $\widehat{\mathtt{Z}}(0,y)$.  
The definition of the stack $\mathcal{Z}(m)$ makes sense when $m=0$, 
but the map $\mathcal{Z}(0)\map{}\mathcal{M}$ is surjective, and so this is clearly 
not the way to proceed.
To find the correct definition of $\widehat{\mathtt{Z}}(0,y)$ one should interpret 
$\widehat{\mathrm{CH}}^1(\mathcal{M})$
as the group  of isomorphism classes of metrized line bundles on $\mathcal{M}$.  Based on 
work of Kudla-Rapoport-Yang \cite{kudla04a,KRY} and conjectures of Kudla \cite{kudla04b}, the class
$\widehat{\mathtt{Z}}(0,y)$ should  be defined as the Hodge bundle on $\mathcal{M}$, endowed with a 
particular choice of metric (which will depend on the parameter $y\in \R^{ > 0}$).

One should also seek a natural definition of 
\[
\widehat{\mathtt{Z}}_\Phi(\alpha,y) \in \widehat{\mathrm{CH}}^1(\mathcal{X}_\Phi)
\]
 for all $\alpha\in F$, not just for $\alpha\not\in F^\times$,  for which  Theorem \ref{Main Theorem II} continues to hold.
There are two cases, depending on whether or not $\varphi^\mathrm{sp}(\alpha)=0$.  If $\varphi^\mathrm{sp}(\alpha)\not=0$
then Theorem \ref{Thm:local length} shows that the stack $\mathcal{Z}_\Phi(\alpha)$ has dimension zero;
if $\varphi^\mathrm{sp}(\alpha)=0$
then the proof of that same theorem  shows that every irreducible component 
of $\mathcal{Z}_\Phi(\alpha)$ has dimension $1$.   The upshot is that if $\varphi^\mathrm{sp}(\alpha)\not=0$,
the definition of $\widehat{\mathtt{Z}}_\Phi(\alpha,y)$ should be close to the definition we have given 
for $\alpha\in F^\times$.  If $\varphi^\mathrm{sp}(\alpha)=0$ then the correct
definition should be in terms of the metrized Hodge bundle in $\widehat{\mathrm{CH}}^1(\mathcal{X}_\Phi)$.
These classes should satisfy two properties.  First, the pull-back map
 \[
 \widehat{\mathrm{CH}}^1(\mathcal{M}_{/\co_\Phi} ) \to   \widehat{\mathrm{CH}}^1(\mathcal{X}_\Phi)
 \]
should satisfy
 \[
 \widehat{\mathtt{Z}}(m,y) \mapsto \sum_{ \substack{  \alpha\in F \\ \mathrm{Tr}_{F/\Q}(\alpha)=m }   }
 \widehat{\mathtt{Z}}_\Phi(\alpha,y)
 \]
 for all $m$ and all $y\in \R^{>0}$. Second,  the relation 
 \[
 \widehat{\deg}\, \widehat{\mathtt{Z}}_\Phi(\alpha,y) = b_\Phi(\alpha,y)
 \]
should hold for all $\alpha\in F$ and $y\in F_\R^{\gg 0}$.    Given these two properties one can deduce 
 (\ref{arithmetic pullback})  for all $m\in \Z$ and all $y\in \R^{>0}$.

The final step in the program laid out by Kudla \cite{kudla04b} is 
to form the  vector-valued generating series 
\[
\widehat{\theta}(\tau) = \sum_{m\in \Z} \widehat{\mathtt{Z}}(m,y) \cdot q^m
\in \widehat{\mathrm{CH}}^1(\mathcal{M})[[q]]
\]
for $\tau=x+iy \in \mathcal{H}$.  The equality (\ref{arithmetic pullback}) amounts to
\[
\widehat{\deg}_{\mathcal{X}_\Phi}  \widehat{\theta}(\tau)  =   \mathcal{E}'_{\Phi}( i_F(\tau),0).
\]
Given the results of Kudla-Rapoport-Yang \cite{KRY} on CM cycles on Shimura curves, and the results of  
Bruinier--Burgos-Gil--K\"uhn \cite{bruinier-burgos-kuhn} on Hirzebruch-Zagier divisors on Hilbert modular surfaces, 
it is reasonable to expect that the above generating series is a vector-valued
nonholomorphic modular form of weight $n$.  If $f$ is a weight $n$ cuspform on $\mathcal{H}$, we may therefore
form the Petersson inner product of $f(\tau)$ with $\widehat{\theta}(\tau)$, and so define the \emph{arithmetic theta lift}
\[
\widehat{\theta}_f  = \langle f, \widehat{\theta} \rangle^\mathrm{Pet} \in \widehat{\mathrm{CH}}^1(\mathcal{M}).
\]
Moving the linear functional $\widehat{\deg} _{\mathcal{X}_\Phi}$ inside the integral defining the 
Petersson inner product, one finds the  Gross-Zagier style formula
\begin{align}
\widehat{\deg}_{\mathcal{X}_\Phi}\widehat{\theta}_f  
&=  \langle f,  \widehat{\deg}_{\mathcal{X}_\Phi}  \widehat{\theta} \rangle^\mathrm{Pet}  \nonumber\\
&=   \langle f (\tau),  \mathcal{E}'_{\Phi}( i_F(\tau),0) \rangle^\mathrm{Pet}  \nonumber\\
&= \mathcal{L}'_{\Phi}(f,0)  \label{looks GZ}
\end{align}
where
\[
\mathcal{L}_{\Phi}(f,s) = \langle  f (\tau),  \mathcal{E}_{\Phi}( i_F(\tau),s) \rangle^\mathrm{Pet}.
\]

 In the case where $F$ is a real quadratic field (so $n=2$),
a function very much like $\mathcal{L}_{\Phi}(f,s)$ appears in the work of Gross-Kohnen-Zagier 
\cite{GKZ}, and is shown to be closely 
related to the usual $L$-function of $f$.  When $F$ is a field of degree $>2$ there seems to be no literature at all
on the function $\mathcal{L}_{\Phi}(f,s)$, and the author is at a loss as to its properties and significance.

However, there are interesting cases where $n>2$ 
and one has some hope of  better understanding $\mathcal{L}_\Phi(f,s)$.
For example, consider the totally degenerate case of $F=  \Q \times  \cdots\times  \Q$ and
$K=K_0\times\cdots\times K_0$.   Modulo details, one should expect the following.
Our Hilbert modular Eisenstein series on $\mathcal{H}^n$  is just a product
of classical weight one Eisenstein series 
\[
\mathcal{E}_\Phi(\tau,s) = \mathcal{F}_1(\tau_1,s) \cdots \mathcal{F}_n(\tau_n,s).
\]
Each factor will satisfy a functional equation in $s\to -s$, and the sign of the functional equation
will be $1$ for all factors but one.  Say the last factor has sign $-1$.
The first $n-1$ factors are  \emph{coherent}, the last one is the \emph{incoherent} factor.
The Siegel-Weil formula implies that the value at $s=0$ of each coherent factor is a theta function
$\Theta$ attached to the extension $K_0/\Q$, and so 
\[
\mathcal{E}'_{\Phi}( i_F(\tau),0)  = \Theta^{n-1}(\tau)  \mathcal{F}_n'(\tau,0).
\]
But the Petersson inner product
\[
L( f \times \Theta^{n-1} ,s)=\langle f(\tau) , \Theta^{n-1}(\tau) \mathcal{F}_n(\tau,s) \rangle^\mathrm{Pet}
\]
is, up to rescaling and shifting in the variable $s$, just the Rankin-Selberg convolution $L$-function
of $f$ with $\Theta^{n-1}$, and hence the mysterious function $\mathcal{L}_\Phi(f,s)$  has the 
less mysterious central derivative
\[
 \mathcal{L}'_\Phi(f,0)  = L'( f \times \Theta^{n-1} ,0).
\]
When $n=2$, so $F=\Q\times\Q$, 
the Rankin-Selberg $L$-function on the right is the one appearing in the work of Gross-Zagier
\cite{gross-zagier}, as we have noted earlier.

Finally, it may be helpful to put the above results and conjectures into the context of 
seesaw dual pairs and the Siegel-Weil formula, which are among 
the guiding principles of Kudla's conjectures \cite{kudla04b}. 
 Suppose we start with free $K$-module $W$ of rank $1$, equipped with 
a totally positive definite Hermitian form $\langle\cdot,\cdot\rangle_\CM$. 
Let $V$ denote the underlying $K_0$-vector space with the $K_0$-Hermitian form 
$\langle v_1,v_2\rangle = \mathrm{Tr}_{K/K_0} \langle w_1,w_2\rangle_\CM$.  
Define a torus $T=\mathrm{Res}_{F/\Q} U(W)$ so that  $T \subset U(V)$.
The  dual reductive pairs $( \mathrm{SL}_2, U(V) )$ and 
$(\mathrm{Res}_{F/\Q}\mathrm{SL}_2,T  ) $ can be arranged into the 
 \emph{seesaw diagram}
\[
\xymatrix{
{ \mathrm{Res}_{F/\Q}\mathrm{SL}_2 } \ar@{-}[d]     &   {U(V)} \ar@{-}[d]  \\
{\mathrm{SL}_2 } \ar@{-}[ur]  & { T . } \ar@{-}[ul]
}
\]
Starting with a cusp form $f$ on $\mathrm{SL}_2$, one can theta lift to an automorphic form $\theta_f$ on $U(V)$,
then restrict to $T$ and integrate against the constant function $1$.  By tipping the 
seesaw, this is the same as theta lifting the constant function $1$ on  $T$
to a Hilbert modular theta series on $\mathrm{Res}_{F/\Q}\mathrm{SL}_2$, restricting that theta series
to the diagonally embedded $\mathrm{SL_2}$, and integrating against $f$.  The Siegel-Weil formula
implies that the Hilbert modular theta series appearing in the this process is in fact the central value 
of a Hilbert modular Eisenstein series, say $E(g,s)$, at $s=0$.  Thus 
\begin{equation}\label{looks SW}
\int_{T(\A)} \theta_f(t) \, dt= \int_{\mathrm{SL}_2(\A)} f(g) E(g,0) \, dg.
\end{equation}

The conjectural picture described (and largely proved) above, is formally similar.
Automorphic forms on $U(V)$ are replaced by elements of $\widehat{\mathrm{CH}}^1(\mathcal{M})$,
the theta lift $f\mapsto \theta_f$ is replaced by the arithmetic theta lift $f \mapsto \widehat{\theta}_f$, 
and the linear functional ``integrate over the torus $T$" 
is replaced by  the linear functional ``arithmetic degree along $\mathcal{X}_\Phi$."
On the other side of the seesaw, the Hilbert modular theta series (which is the central \emph{value}
of an Eisenstein series) is replaced by the central \emph{derivative} of 
an Eisenstein series, and the integral over $\mathrm{SL}_2$ is replaced by the 
Petersson inner product.  In this way,  (\ref{looks GZ}) can be seen as 
an arithmetic version of (\ref{looks SW}).

\subsection{Acknowledgements}

The author thanks both Steve Kudla and Tonghai Yang for helpful conversations,
and the anonymous referee for helpful comments on an earlier draft of this paper.


\section{Barsotti-Tate groups with complex multiplication}
\label{s:deformations}


This section contains the technical deformation theory calculations that will eventually be used in the 
proof of Theorem \ref{Thm:local length} to compute the lengths of the local rings of the 
$0$-dimensional stack $\mathcal{Z}_{\Phi}(\alpha)$.
The reader might prefer to begin with the global theory of Section \ref{S:global moduli}, and refer back to this
section as needed.

Fix a prime $p$ and let $\F$ be an algebraic closure of the field of $p$ elements.  Let 
$W$ be the ring of Witt vectors of $\F$, let $\mathrm{Frac}(W)$ be the fraction field of $W$, and let
$\C_p$ be any algebraically closed field containing   $\mathrm{Frac}(W)$.    
If $L$ is a product of finite extensions of $\Q_p$, denote by $L^u$ the maximal unramified extension of $\Q_p$ in $L$,
and by $\co^u_L$ the ring of integers of $L^u$. A $p$-divisible group over $\F$ is \emph{supersingular} if 
all slopes of its Dieudonn\'e module are equal to $1/2$.  Here and throughout, \emph{Dieudonn\'e module} means 
\emph{covariant Dieudonn\'e module}.


\subsection{Deformations of Barsotti-Tate groups with complex multiplication}
\label{ss:canonical lifts}


Let $F$ be a field extension of $\Q_p$ of degree $n$, and let
$K$ be a quadratic \'etale $F$-algebra (so $K$ is a either a quadratic field extension of $F$, or $K\iso F\times F$).  
Denote by $x\mapsto \overline{x}$ the nontrivial  automorphism  of $K/F$, and for any $\Q_p$-algebra map 
$\varphi:K\to \C_p$ define the conjugate map $\overline{\varphi}(x)=\varphi(\overline{x})$.  
A \emph{$p$-adic CM type}  of   $K$ is a set $\Phi$ of $\Q_p$-algebra maps $K\to \C_p$ 
such that $\Hom(K, \C_p)$ is the disjoint   union of $\Phi$ and $\overline{\Phi}$.   Fix such a $\Phi$, and  let  
 $K_\Phi\subset \C_p$ be any finite extension   of $\Q_p$ large enough that  
 \begin{equation}\label{reflex condition}
 \sigma\in \Aut(\C_p/K_\Phi) \implies \Phi^\sigma=\Phi.
 \end{equation}
  Denote by \[W_\Phi\subset\C_p\] the ring of integers in the completion of the maximal unramified 
  extension of $K_\Phi$, so that $\mathrm{Frac}(W_\Phi)$ is the compositum of 
   $K_\Phi$ and  $\mathrm{Frac}(W)$.  

Let $\mathbf{ART}$ be the category of local Artinian    $W_\Phi$-algebras with residue field $\F$.
If $R$ is an object of $\mathbf{ART}$ and $A$ is a $p$-divisible group over $R$,  
 an action $\kappa: \co_K\to \End(A)$ satisfies the \emph{$\Phi$-determinant condition} if for
 every $x_1,\ldots,x_r\in \co_K$ the determinant of $T_1 x_1+\cdots+ T_r x_r$ acting on
  $\Lie(A)$  is equal to the image of
  \begin{equation}\label{det poly}
  \prod_{\varphi\in \Phi} ( T_1 \varphi(x_1)+\cdots+ T_r \varphi(x_r) ) \in W_\Phi[T_1,\ldots,T_r]
  \end{equation}
  in $R[T_1,\ldots, T_r]$. In particular,  this implies 
  \[
   \mathrm{dim}(A)  =  [ F : \Q_p ].
  \]
  For the remainder of this subsection, fix a  triple $(A,\kappa,\lambda)$ in which 
\begin{itemize}
\item
$A$ is a $p$-divisible group over $\F$,
\item
$\kappa:\co_K\to \End(A)$ satisfies the $\Phi$-determinant condition, 
\item
$\lambda:A\to A^\vee$ is an $\co_K$-linear polarization of $A$ 
(which is not  assumed to be principal).  The condition of $\co_K$-linearity means that
\[
\lambda\circ \kappa(\overline{x}) = \kappa(x)^\vee\circ \lambda
\]
  for every $x\in \co_K$.
\end{itemize} 
 Let  $\mathrm{Def}_\Phi(A,\kappa,\lambda)$
be the functor that assigns to every object $R$ of $\mathbf{ART}$ the set of isomorphism classes of
deformations of $(A,\kappa,\lambda)$ to $R$, where the deformations are again required to satisfy the 
$\Phi$-determinant condition.  The goal of this subsection is to prove that $\mathrm{Def}_\Phi(A,\kappa,\lambda)$
is pro-represented by $W_\Phi$.

 \begin{Prop}\label{Prop:BT basics}
 Let $(A,\kappa,\lambda)$ be the triple fixed above.
\begin{enumerate}
\item
The Dieudonn\'e module $D(A)$ is free of rank one over $\co_K\otimes_{\Z_p} W$.
\item
The image of $\co_K$ in $\End(A)$ is equal to its own centralizer.
\item
If $K$ is a field then $A$ is supersingular.
\end{enumerate}
\end{Prop}

\begin{proof}
The first claim follows from the argument of  \cite[Lemma 1.3]{rapoport78}, and the second claim follows easily from the first.
The category of Dieudonn\'e modules over $\F$ up to isogeny is semisimple.  If 
\[
D(A) \sim D_1^{m_1}\times \cdots \times D_r^{m_r}
\] 
with $D_1,\ldots, D_r$ simple and pairwise non-isogenous, then
\[
\End(D(A))\otimes_{\Z_p}\Q_p \iso M_{m_1}(H_1) \times\cdots\times M_{m_r}(H_r)
\] 
with each $H_i$ a  division algebra over $\Q_p$.  The only way  this product can contain a field  
equal to its own centralizer is if $r=1$.  Therefore $D(A)$ is isoclinic: it is isogenous to a power of a simple 
Dieudonn\'e module, and hence its slope sequence is constant.   By hypothesis $D(A)$ admits a polarization, so its
slope sequence is symmetric in $s\mapsto 1-s$.  Therefore $1/2$ is the unique slope of $D(A)$.
\end{proof}

 Given a $\Q_p$-algebra map $\varphi:K\to \C_p$, let $\C_p(\varphi)$ denote $\C_p$,
with $K$ acting through $\varphi$.    There is a unique $W_\Phi$-algebra map
 \[
 \eta_\Phi:\co_K\otimes_{\Z_p} W_\Phi \to  \prod_{\varphi\in \Phi} \C_p(\varphi)
 \]
 sending  $x\otimes 1$ to the $\Phi$-tuple $(\varphi(x))_{\varphi\in \Phi}$.  
The kernel and image of $\eta_\Phi$ are denoted  $J_\Phi$ and $\Lie_\Phi$, respectively, 
so that there is an exact sequence of 
 $\co_K\otimes_{\Z_p} W_\Phi$-modules
\begin{equation}\label{generic hodge}
0\to J_\Phi \to  \co_K\otimes_{\Z_p} W_\Phi \to  \Lie_\Phi\to 0.
\end{equation}
We will make repeated use of the isomorphism
\[
\co_K^u \otimes_{\Z_p} W\iso \prod_{ \psi : \co_K^u \to W } W
\]
sending  $x\otimes 1\mapsto (\psi(x))_\psi$.  For each factor on the right hand side there is a corresponding 
idempotent $e_\psi \in \co_K^u \otimes_{\Z_p} W$ characterized by 
\[
(x\otimes 1)  e_\psi = (1\otimes \psi(x)) e_\psi
\]  
for all $x\in \co_K^u$.

\begin{Lem}\label{Lem:key}
The ideal $J_\Phi\subset \co_K\otimes_{\Z_p} W_\Phi$ enjoys the following properties.
\begin{enumerate}
\item[(a)]
As $W_\Phi$-modules, $J_\Phi$ and $\mathrm{Lie}_\Phi$ are each free of rank $n$.   Furthermore, for any 
tuple $x_1,\ldots, x_r\in \co_K$ the determinant of $T_1x_1+\cdots+ T_r x_r$ acting on $\Lie_\Phi$ 
is equal to (\ref{det poly}).
\item[(b)]
The ideal $J_\Phi$ is generated by  the set of all elements of the form
\[
j_\Phi(x,\psi) = e_\psi  \prod_{ \substack{  \varphi\in \Phi  \\ \varphi|_{\co_K^u} =\psi  }} 
 ( x\otimes 1- 1\otimes \varphi(x)) \in \co_K\otimes_{\Z_p}W_\Phi.
\]
 with $x\in \co_K$ and $\psi:\co_K^u \to  W.$
\item[(c)]
Suppose $R$ is an object of $\mathbf{ART}$, $M$ is a free $\co_K\otimes_{\Z_p} R$-module of rank one, 
and $M_1\subset M$ is an $\co_K$-stable $R$-direct summand such that for any  $x_1,\ldots, x_r \in \co_K$ 
the determinant of $T_1x_1+\cdots+T_r x_r$ acting on  $M/M_1$ is (\ref{det poly}).  Then $M_1=J_\Phi M$.
\end{enumerate}
\end{Lem}

 \begin{proof}
The first claim is elementary linear algebra, and  the proof is left to the reader.
For the second claim, $j_\Phi(x,\psi)\in J_\Phi$ is obvious from the definitions.  To prove the
other inclusion, fix a $\varpi\in \co_K$ such that $\co_K=\co_K^u[\varpi]$, and  
let $\mu(z)\in \co_K^u[z]$ be the minimal polynomial of $\varpi$.
Using $\varpi\mapsto z$ to identify $\co_K\iso \co_K^u[z]/(\mu)$, we obtain an isomorphism
\[
\co_K\otimes_{\Z_p} W_\Phi \iso (\co_K^u\otimes_{\Z_p} W_\Phi)[z]/(\mu)\iso 
\prod_{\psi: \co_K^u\to W }W_\Phi[z]/(\mu_\psi)
\]
where $\mu_\psi$ is the image of $\mu$ under $\psi:\co_K^u[z]\to W[z]$.  Under these isomorphisms,
the element $j_\Phi(\varpi,\psi)$ on the left is identified with the tuple on the right whose $\psi$-coordinate
is the polynomial 
\begin{equation}\label{roots}
\prod_{ \substack{\varphi \in \Phi \\ \varphi|_{\co_K^u=\psi } }} (z-\varphi(\varpi))
\end{equation}
and all other coordinates are $0$.  Now suppose $j\in J_\Phi$.  Under the above isomorphism 
$j$ corresponds to a tuple of polynomials $j_\psi(z)\in W_\Phi[z]/(\mu_\psi)$, and the assumption 
that $j\in J_\Phi$ means precisely that the polynomial $j_\psi(z)$ vanishes at $z=\varphi(\varpi)$
for each $\varphi\in \Phi$ whose restriction to $\co_K^u$ is $\psi$.  Such a $j_\psi(z)$ is obviously divisible,
in $W_\Phi[z]$, by (\ref{roots}).  It follows that $e_\psi j$ is a multiple of 
$j_\Phi(\varpi,\psi)$, and hence that $J_\Phi$
is contained in the ideal generated by the elements $j_\Phi(\varpi,\psi)$ as $\psi$ varies. 

For the final claim, extend each $\varphi\in \Phi$ to a $W$-linear map $\varphi:\co_K\otimes_{\Z_p}W\to \C_p$.
The determinant condition imposed on $M/M_1$ implies that for every $x\in \co_K$, 
\[
e_\psi (x\otimes 1) \in \co_K\otimes_{\Z_p} W
\] 
acts on $M/M_1$ with characteristic polynomial
\[
 \prod_{  \varphi\in \Phi } 
 (T-  \varphi(e_\psi)\varphi(x) ) 
 = T^r
 \prod_{ \substack{  \varphi\in \Phi  \\ \varphi|_{\co_K^u} =\psi  }} 
 (T-  \varphi(x)) \in W_\Phi[T]
\]
where $r=\#\{ \varphi\in \Phi : \varphi|_{\co_K^u} \not=\psi\}$, and hence acts on $e_\psi(M/M_1)$
with  characteristic polynomial 
\[
 \prod_{ \substack{  \varphi\in \Phi  \\ \varphi|_{\co_K^u} =\psi  }} 
 (T-  \varphi(x)) \in W_\Phi[T].
\]
Therefore $x\otimes 1$ acts on $e_\psi(M/M_1)$ with this same characteristic polynomial, and
the Cayley-Hamilton theorem implies that $e_\psi(M/M_1)$  is annihilated by
\[
 \prod_{ \substack{  \varphi\in \Phi  \\ \varphi|_{\co_K^u} =\psi  }} 
 ( x\otimes 1- 1\otimes \varphi(x)) \in \co_K\otimes_{\Z_p}W_\Phi.
\]
Hence $M/M_1$ is annihilated by $j_\Phi(x,\psi)$.  By the second claim of the lemma, $J_\Phi M\subset M_1$,
and as  $J_\Phi M$ and $M_1$ are  $R$-module direct summands of $M$ of the same rank, they must be equal.
\end{proof}

We now make use of the theory of Grothendieck-Messing crystals. Standard references include 
\cite{BBM,grothendieck74,messing72}; for Zink's reconstruction of the theory by different means, 
see \cite{messing07,zink02}.    Associated to $A$ is a short exact sequence
\begin{equation}\label{hodge base}
0\to \mathrm{Fil}^1 \mathcal{D}_A(\F)\to \mathcal{D}_A(\F)\to \mathrm{Lie}(A)\to 0
\end{equation}
of  $\F$-modules, in which $\mathcal{D}_A(\F)$ is the covariant  Grothendieck-Messing crystal
of $A$ evaluated at $\F$, and the submodule $\mathrm{Fil}^1\mathcal{D}_A(\F)$ is its Hodge filtration.  
Proposition \ref{Prop:BT basics} and the isomorphisms 
\[
\mathcal{D}_A(\F)\iso \mathcal{D}_A(W)\otimes_W\F \iso D(A)\otimes_W\F,
\]
the second by \cite[Theorem 4.2.14]{BBM}, imply that 
\[
\mathcal{D}_A(\F) \iso \co_K\otimes_{\Z_p}\F,
\] 
and the final claim of Lemma \ref{Lem:key} implies that 
\[
\mathrm{Fil}^1\mathcal{D}_A(\F) = J_\Phi \mathcal{D}_A(\F).
\] 
 In particular, (\ref{hodge base}) is obtained from 
(\ref{generic hodge}) by applying $\otimes_{W_\Phi}\F$, which explains our choice of notation $\mathrm{Lie}_\Phi$.
Similarly, if  $R$ is an object of $\mathbf{ART}$ and 
\[
(A',\kappa',\lambda')\in \mathrm{Def}_\Phi(A,\kappa,\lambda)(R),
\] 
 there is an associated short exact sequence of free $R$-modules
 \begin{equation}\label{hodge lift}
0\to  \mathrm{Fil}^1 \mathcal{D}_{A'}(R)\to \mathcal{D}_{A'}(R)\to \mathrm{Lie}(A')\to 0.
\end{equation}
Applying $\otimes_R\F$ to (\ref{hodge lift}) recovers (\ref{hodge base}), and  an easy Nakayama's lemma argument 
then shows that 
\[
\mathcal{D}_{A'}(R)\iso \co_K\otimes_{\Z_p} R.
\]  
Another application of 
Lemma \ref{Lem:key} shows that  \[ \mathrm{Fil}^1\mathcal{D}_{A'}(R) = J_\Phi \mathcal{D}_{A'}(R), \] and so
  (\ref{hodge lift}) is obtained from (\ref{generic hodge}) by applying $\otimes_{W_\Phi} R$.
 In this sense, (\ref{generic hodge}) is the universal  Hodge short exact sequence of  deformations 
 of $(A,\kappa,\lambda)$.
 As the following theorem demonstrates, this  information is enough to deduce the existence and 
 uniqueness of deformations of $(A,\kappa,\lambda)$.

\begin{Thm}\label{Thm:BT canonical}
The functor $\mathrm{Def}_\Phi(A,\kappa,\lambda)$ is pro-represented by $W_\Phi$.  Equivalently,
the triple $(A,\kappa,\lambda)$ admits a unique deformation  to every object  of $\mathbf{ART}$.
\end{Thm}

\begin{proof}
Let $S\to R$ be a surjective morphism in $\mathbf{ART}$ whose kernel $\mathcal{I}$
satisfies $\mathcal{I}^2=0$.  In particular $\mathcal{I}$ comes equipped with its trivial divided power 
structure.  Suppose we have already lifted the triple $(A,\kappa,\lambda)$ over $\F$ to a triple
$(A',\kappa',\lambda')$ over $R$.  Let  $\mathcal{D}_{A'}$ be the 
Grothendieck-Messing crystal of $A'$, so that 
$\mathcal{D}_{A'}(R) \iso \co_K\otimes_{\Z_p} R$, and 
\[
\mathrm{Fil}^1\mathcal{D}_{A'}(R) = J_\Phi \mathcal{D}_{A'}(R).
\]
Now evaluate the crystal $\mathcal{D}_{A'}$ on $S$.  As $\mathcal{D}_{A'}(S)\otimes_{S}\F\iso \mathcal{D}_A(\F)$,
a Nakayama's lemma argument shows that $\mathcal{D}_{A'}(S)$ is free of rank one over 
$\co_K\otimes_{\Z_p} S$.   By the main result of Grothendieck-Messing theory,
deformations of the pair $(A',\kappa')$  to $S$ (satisfying the $\Phi$-determinant condition, as always) are in bijection 
with $\co_K$-stable $S$-direct summands $M_1\subset \mathcal{D}_{A'}(S)$ for which the action of $\co_K$
on  $\mathcal{D}_{A'}(S)/M_1$ satisfies   the  $\Phi$-determinant condition.
The final claim of Lemma \ref{Lem:key} shows that $M_1=J_\Phi\mathcal{D}_{A'}(S)$  
is the unique such summand, and so  $(A',\kappa')$ admits a unique deformation $(A'',\kappa'')$  to $S$.

By the results of \cite[Chapter 5.3]{BBM}, the polarization $\lambda'$ of $A'$ 
induces an alternating $S$-bilinear form
\[
\lambda':\mathcal{D}_{A'}(S)\times\mathcal{D}_{A'}(S) \to S
\]
satisfying $\lambda( x v,w) = \lambda(v, \overline{x} w)$ for every $x\in \co_K\otimes_{\Z_p}W_\Phi$.  
But every  $x\in J_\Phi$ satisfies 
\[
x\overline{x}\in \mathrm{ker}(\eta_\Phi) \cap \mathrm{ker}(\eta_{\overline{\Phi}}) = 0,
\]
and hence $J_\Phi\mathcal{D}_{A'}(S)$ is isotropic for the pairing $\lambda'$.  
This  implies that the polarization $\lambda'$ lifts (uniquely) to a polarization $\lambda''$
of $(A'',\kappa'')$, and so  $(A',\kappa',\lambda')$ admits a unique deformation to $S$.  
Induction  on the length   now shows that  $(A,\kappa,\lambda)$ lifts uniquely to every object 
of $\mathbf{ART}$.
\end{proof}

\begin{Rem}\label{Rem:strong lift}
The proof of Theorem \ref{Thm:BT canonical} actually shows that something slightly 
stronger is true: the pair $(A,\kappa)$ deforms uniquely to every object of $\mathbf{ART}$, 
and $\lambda$ automatically  lifts to that deformation.
\end{Rem}

\begin{Rem}
Instead of the $\Phi$-determinant condition imposed on the action $\co_K\to \End(A)$
at the beginning of this subsection, we might have imposed the (seemingly) weaker condition that 
every $x\in \co_K$ acts on $\Lie(A)$ with characteristic polynomial equal to the image of
\[
\prod_{\varphi\in \Phi}(T-\varphi(x)) \in W_\Phi[T]
\]
in $R[T]$.   The advantage of the stronger $\Phi$-determinant condition  
is that it determines not only the characteristic 
polynomial of every element of $\co_K$,  but of every element
of $\co_K\otimes_{\Z_p} R$.  This was needed in the proof of part (c) of Lemma \ref{Lem:key}.
It would be interesting to know whether the two conditions are equivalent.
\end{Rem}


\subsection{Deformations of CM abelian varieties}


Theorem \ref{Thm:BT canonical}, when combined with the Serre-Tate theorem,  gives 
information about the formal deformation spaces of CM abelian varieties over $\F$, and 
in much greater generality than is needed in this paper.  Because the  result,
Theorem \ref{Thm:cm etale}, is of independent interest, we state it in full generality.

Let $K=\prod K_i$ be any product of CM fields, and let $F=\prod F_i$ be its maximal totally real subalgebra.
Let  $\Phi$ be any CM type of $K$, and  let $K_\Phi\subset \C$ be a number field 
containing the reflex field of $\Phi$.  Set $\co_\Phi=\co_{K_\Phi}$.
Suppose $\mathfrak{o}\subset K$ is an order and $S$ is a locally Noetherian $\co_\Phi$-scheme.  A 
\emph{polarized $(\mathfrak{o},\Phi)$-CM abelian scheme over $S$} is a triple  $(A,\kappa,\lambda)$ 
 in which 
\begin{itemize}
\item
$A\to S$ is  an abelian scheme  over $S$ of relative dimension $n$,
\item
$\kappa:\mathfrak{o}\to \End(A)$ is an action of $\mathfrak{o}$ on $A$ such that, locally on $S$, for any
tuple $x_1,\ldots,x_r \in \mathfrak{o}$ the determinant of $T_1x_1+\cdots+T_r x_r$ on $\mathrm{Lie}(A)$ 
is equal to the image of 
\[
\prod_{\varphi\in \Phi}( T_1\varphi(x_1) + \cdots+ T_r \varphi(x_r) ) \in \co_\Phi[T_1,\ldots, T_r]
\]
in $\co_S [T_1,\ldots,T_r]$,
\item 
$\lambda:A\to A^\vee$ is a polarization of $A$ satisfying 
$\lambda\circ \kappa(\overline{x}) = \kappa(x)^\vee\circ \lambda$ for  every $x\in \mathfrak{o}$.
\end{itemize}

Given a prime $\mathfrak{p}$ of $K_\Phi$,  let $W_{\Phi,\mathfrak{p}}$ be the completion of the ring
of integers in the maximal unramified extension of $K_{\Phi,\mathfrak{p}}$, 
and let $k_{\Phi,\mathfrak{p}}^\alg$ be its residue field.  Let $p$ be the rational prime below $\mathfrak{p}$.

\begin{Thm}\label{Thm:cm etale}
Let $R$ be a complete local Noetherian $W_{\Phi,\mathfrak{p}}$-algebra  with residue field  
$k_{\Phi,\mathfrak{p}}^\alg$.  If $\mathfrak{o}$ is maximal at $p$ then every polarized 
$(\mathfrak{o},\Phi)$-CM abelian scheme  $(A,\kappa,\lambda)$ over $k_{\Phi,\mathfrak{p}}^\alg$ admits a unique 
deformation to a polarized $(\mathfrak{o},\Phi)$-CM abelian scheme over $R$.
\end{Thm}

\begin{proof}
  Let $\C_\mathfrak{p}$ be an algebraically closed field
containing $W_{\Phi,\mathfrak{p}}$, and fix an isomorphism between the algebraic closures of $K_\Phi$ in 
$\C$ and $\C_\mathfrak{p}$.  This allows us to view elements of $\Phi$ as maps $\varphi:K\to \C_\mathfrak{p}$.
For any prime $\mathfrak{P}$ of $F$ above $p$ let  $\Phi(\mathfrak{P})\subset \Phi$ be the subset consisting of those
$\varphi$ whose restriction to $F$ induces the prime $\mathfrak{P}$.
There is a decomposition of  $p$-divisible groups 
\[
A[p^\infty]\iso \prod_{\mathfrak{P}} A[\mathfrak{P}^\infty]
\]
where the product is over the primes of $F$ lying above $p$.  A similar decomposition holds for any 
deformation of the triple $(A,\kappa,\lambda)$.  Each factor  $A[\mathfrak{P}^\infty]$ has an action
\[
\kappa[\mathfrak{P}^\infty] : \co_{K,\mathfrak{P}} \to \End( A[\mathfrak{P}^\infty] )
\]
satisfying the $\Phi(\mathfrak{P})$-determinant condition, and an $\co_{K,\mathfrak{P}}$-linear
polarization $\lambda[\mathfrak{P}^\infty]$. If $R$ is Artinian, we may apply
Theorem \ref{Thm:BT canonical} to see that each triple 
$(A[\mathfrak{P}^\infty],\kappa[\mathfrak{P}^\infty],\lambda[\mathfrak{P}^\infty])$ admits
 a unique deformation to $R$.  By the Serre-Tate theorem \cite{messing72}, 
 the same is true of the triple $(A,\kappa,\lambda)$.
 This proves the claim if $R$ is Artinian, and the general case follows from Grothendieck's formal existence theorem 
 as in   \cite[Section 3]{conrad04}.
 \end{proof}

Theorem \ref{Thm:cm etale} is false  if one  omits the hypothesis that $\mathfrak{o}$ is maximal at $p$, 
even for elliptic curves.   This is clear from the theory of \emph{quasi-canonical lifts} of elliptic curves,
due to Serre-Tate in the ordinary case, and Gross in the supersingular case 
\cite{gross86, ARGOS-9, ARGOS-8}.


\subsection{Lifting homomorphisms: the signature $(n-1,1)$ case}
\label{ss:LHI}


In this subsection  $K_0$ is a quadratic field extension of $\Q_p$,  $F/\Q_p$ is a field extension of 
degree  $n$, and 
\[
K=K_0\otimes_{\Q_p} F.
\]  
We assume that $K_0$ does not embed into $F$, so that $K$
is a field.  Let $\mathfrak{D}_0$ and $\mathfrak{D}$ be the differents of $K_0/\Q_p$ and $K/\Q_p$, respectively,
and let $\mathfrak{p}_F$ be the maximal ideal of $\co_F$.

 Fix an embedding $\iota:K_0\to \C_p$, so that  
$\Phi_0=\{\iota\}$ is a $p$-adic CM type of $K_0$.  A $p$-adic CM type $\Phi$ of $K$ is said to be
of \emph{signature} $(n-1,1)$ if there is a unique  $\varphi^{\mathrm{sp}}\in \Phi$ 
 satisfying $\varphi^\mathrm{sp}|_{K_0}=\overline{\iota}$.   The distinguished element $\varphi^\mathrm{sp}$
 is the  \emph{special element} of $\Phi$, and this element determines $\Phi$ uniquely.  Fix a $\Phi$ of signature
 $(n-1,1)$ and  define  
 \[
 K_\Phi = \varphi^\mathrm{sp}(K).
 \]  
 For any $\sigma\in \Aut(\C_p/K_\Phi)$ the $p$-adic 
CM type $\Phi^\sigma$ is again of signature $(n-1,1)$, and still contains $\varphi^\mathrm{sp}$.
Thus  $\Phi=\Phi^\sigma$, and  condition (\ref{reflex condition}) is satisfied.  Let $W_\Phi$ be 
as in Section \ref{ss:canonical lifts}.

 Fix a triple $(A,\kappa,\lambda)$ in which 
 \begin{itemize}
 \item
 $A$ is a $p$-divisible group over  $\F$ of dimension $n$,
 \item
 $\kappa:\co_K\to \End(A)$ satisfies the $\Phi$-determinant condition, 
 \item  
$\lambda:A\to A^\vee$ is an $\co_K$-linear polarization with kernel $A[\mathfrak{a}]$ for 
some ideal $\mathfrak{a}\subset\co_F$.
\end{itemize}  
Fix a second triple 
 $(A_0,\kappa_0,\lambda_0)$ in which 
 \begin{itemize}
 \item
  $A_0$ is a  $p$-divisible group over  $\F$ of dimension $1$,
  \item
  $\kappa_0:\co_{K_0}\to \End(A_0)$ satisfies the $\Phi_0$-determinant condition,
  \item 
$\lambda_0:A_0\to A_0^\vee$ is an $\co_{K_0}$-linear principal polarization.  
\end{itemize}
By  Proposition \ref{Prop:BT basics} both   $A_0$ and $A$ are supersingular.

The $\co_K$-module 
\[
L(A_0,A)=\Hom_{\co_{K_0}}(A_0,A)
\]
 has a natural   $\co_{K_0}$-Hermitian form $\langle f_1,f_2\rangle$ defined by
\[
\langle f_1, f_2\rangle = \lambda_0^{-1} \circ f_2^\vee \circ \lambda \circ f_1.
\]
This Hermitian form is compatible with the action of $\co_K$ on $L(A_0,A)$, in the sense that 
\[
\langle x\cdot f_1,f_2\rangle = \langle f_1, \overline{x}\cdot f_2\rangle
\]
for every $x\in \co_K$, and it follows that there is a unique $K$-valued
$\co_K$-Hermitian form $\langle f_1,f_2\rangle_\CM$ on $L(A_0,A)$ satisfying
\[
\langle f_1,f_2\rangle = \mathrm{Tr}_{K/K_0} \langle f_1,f_2\rangle_\CM.
\]
It is not easy to give a description of $\langle\cdot ,\cdot\rangle_\CM$, other than ``the Hermitian 
form whose trace is $\langle\cdot ,\cdot\rangle$."  Nevertheless, 
the structure of the Hermitian space $\big( L(A_0,A), \langle\cdot,\cdot\rangle_\CM \big)$
will be described quite explicitly  in Proposition \ref{Prop:hermitian local I} below.

Set 
\[
\mathcal{S}= \co_K \otimes_{\Z_p} W.
\]  
If   $\Frob\in \Aut(W)$ is the Frobenius automorphism, there is an induced automorphism of $\mathcal{S}$ 
defined by 
\[
(x\otimes w)^\Frob = x\otimes w^\Frob.
\]
As in Section \ref{ss:canonical lifts}, for each $\psi:\co^u_K\to W$ there is an idempotent $e_\psi\in \mathcal{S}$ satisfying  
\[
(x\otimes 1) e_\psi = (1\otimes \psi(x))e_\psi
\] 
for every $x\in \co_K^u$.  These idempotents satisfy 
$(e_\psi)^\Frob = e_{\Frob\circ \psi}$, and 
\[
\mathcal{S}=
\prod_{ \psi:\co^u_K \to W} e_\psi \mathcal{S}
\]
where each factor on the right is isomorphic to the ring of integers in the completion of the 
maximal unramified extension of $K$. In particular each factor is a discrete valuation ring, 
whose valuation determines a surjection   
\[
\ord_\psi:\mathcal{S} \to \Z^{\ge 0}\cup\{\infty\}.
\]
Denote by 
\[
m(\psi,\Phi) = \# \{\varphi\in \Phi : \varphi|_{\co_K^u} = \psi \}
\]
 the \emph{multiplicity of $\psi$} in $\Phi$. 
Similarly, if we set 
\[
\mathcal{S}_0=\co_{K_0}\otimes_{\Z_p} W
\] 
 there is a decomposition of $W$-algebras
\[
\mathcal{S}_0 = \prod_{\psi_0:\co^u_{K_0} \to W} e_{\psi_0} \mathcal{S}_0
\]
in which each factor is isomorphic to the integers in the completion of the maximal unramified extension 
of $K_0$.  For each $\psi_0$ there is an associated valuation 
\[
\ord_{\psi_0}:\mathcal{S}_0\to \Z^{\ge 0} \cup \{\infty\},
\]
and the \emph{multiplicity of $\psi_0$ in $\Phi_0$} is 
\[
m(\psi_0,\Phi_0) = \# \{\varphi\in \Phi_0 : \varphi|_{\co_{K_0}^u} = \psi_0 \}.
\]

Let $D(A_0)$ and $D(A)$ be the  Dieudonn\'e modules of $A_0$ and $A$, respectively.
The following lemma makes the structure of these   Dieudonn\'e modules more explicit.

\begin{Lem}\label{Lem:dieu coords}
There is an isomorphism of $\mathcal{S}$-modules $D(A) \iso \mathcal{S}$.  Under any such isomorphism 
the operators $F$ and $V$ on $D(A)$ take the form $F= a \circ \Frob$ and $V = b \circ \Frob^{-1}$
for some $a, b\in \mathcal{S}$ satisfying $a b^\Frob =p$, and satisfying \[\ord_\psi (b) = m(\psi,\Phi)\]
for every $\psi:\co_K^u\to W$.

Similarly, there is an isomorphism of $\mathcal{S}_0$-modules $D(A_0) \iso \mathcal{S}_0$.  Under any such isomorphism 
the operators $F$ and $V$ on $D(A_0)$ take the form $F= a_0 \circ \Frob$ and $V = b_0 \circ \Frob^{-1}$
for some $a_0, b_0\in \mathcal{S}_0$ satisfying $a_0 b_0^\Frob =p$, and 
satisfying \[ \ord_{\psi_0} (b_0) = m(\psi_0,\Phi_0) \] for every $\psi_0:\co^u_{K_0} \to W$.
\end{Lem}

\begin{proof}
We give the proof only for $D(A)$, as the proof for $D(A_0)$ is identical. 
 The only assertion that isn't obvious from Proposition \ref{Prop:BT basics}
is the formula for $\ord_\psi(b)$.  The Lie algebra of $A$ is canonically identified with
\[
D(A)/V D(A) \iso \mathcal{S}/b\mathcal{S}, 
\]
and the $e_\psi$ component of $\mathcal{S}/b\mathcal{S}$ is an $\F$-vector space
of dimension $\ord_\psi(b)$. It follows that the characteristic polynomial of any $x\in \co^u_K$ 
acting on $\Lie(A)$ is equal to
\[
\prod_{ \psi:\co^u_K \to W} (T- \psi(x))^{\ord_\psi (b)} \in \F[T].
\]
On the other hand the $\Phi$-determinant condition imposed on $(A,\kappa,\lambda)$ implies that this 
characteristic polynomial is equal to 
\[
\prod_{\varphi\in \Phi} (T- \varphi(x)) =  \prod_{ \psi:\co^u_K \to W}    (T- \psi(x))^{m(\psi,\Phi)}.
\]
 It follows that  $\ord_\psi (b) = m(\psi,\Phi)$ for every  $\psi$. 
\end{proof}

Fix  isomorphisms $D(A_0)\iso \mathcal{S}_0$ and $D(A) \iso \mathcal{S}$  as in the lemma, 
so that  $L(A_0,A)$ is identified with  an $\co_K$-submodule of 
$\Hom_{\mathcal{S}_0}(\mathcal{S}_0,\mathcal{S}) \iso \mathcal{S}$. 
Of course an element of $\Hom_{\mathcal{S}_0}(\mathcal{S}_0,\mathcal{S})$ 
lies in the submodule $L(A_0,A)$ if and only if it respects the $V$ (equivalently, $F$) operators on $D(A_0)$
and $D(A)$.  In the notation of Lemma \ref{Lem:dieu coords} this amounts to
\begin{equation}\label{hermitian model}
L(A_0,A) \iso \{ s\in \mathcal{S} : (b_0  s)^\Frob=b^\Frob s  \}.
\end{equation}
Let  $s\mapsto \overline{s}$ is the automorphism of $\mathcal{S}$ induced by the 
nontrivial automorphism of $K/F$

\begin{Lem}\label{Lem:dieu coords II}
Under the isomorphism (\ref{hermitian model}) the Hermitian form $\langle\cdot,\cdot\rangle_\CM$
on $L(A_0,A)$ is identified with the Hermitian form 
\[ \langle s_1,s_2\rangle_\CM= \xi s_1\overline{s}_2 \]
on the right hand side of (\ref{hermitian model}), for some $\xi\in \mathcal{S}\otimes_\Z\Q$ satisfying 
\begin{enumerate}
\item
$\overline{\xi}=\xi$,
\item
$\xi\mathcal{S} = \mathfrak{a}\mathfrak{D}_0\mathfrak{D}^{-1} \mathcal{S}$, and
\item
$(b_0\overline{b}_0)^\Frob \xi= \xi^\Frob (b\overline{b})^\Frob$
\end{enumerate}
(the last condition guarantees that $\xi s_1\overline{s}_2$ lies in 
$K=(\mathcal{S}\otimes_\Z\Q)^{\Frob=1}$, as it must).
\end{Lem}

\begin{proof}
This is an exercise in linear algebra.
The polarization $\lambda$  induces a $W$-symplectic form 
\[
\lambda:  D(A) \times D(A) \to  W
\]
satisfying $\lambda(sx,y) = \lambda(x,\overline{s} y)$  for all $s\in \mathcal{S}$, and 
$\lambda(Fx,y) = \lambda(x,Vy)^\Frob$.  The first property implies that the induced pairing 
\[
\lambda: \mathcal{S}\times \mathcal{S} \to W
\] 
has the form
\[
\lambda(s_1,s_2) = \mathrm{Tr}_{K/\Q_p}( \zeta s_1\overline{s}_2)
\]
for some $\zeta\in \mathcal{S}\otimes_\Z\Q$ satisfying $\overline{\zeta}=-\zeta$.  The 
second property implies that  $p\zeta=(\zeta b\overline{b})^\Frob$.
The assumption that $\lambda:A\to A^\vee$ has kernel $A[\mathfrak{a}]$ implies that
$\zeta\mathcal{S} =   \mathfrak{a}\mathfrak{D}^{-1}\mathcal{S}$.
Similarly, the principal polarization $\lambda_0$  induces a perfect pairing 
\[
\lambda_0: \mathcal{S}_0\times \mathcal{S}_0 \to W
\]
of the form 
\[
\lambda_0(s_1,s_2) = \mathrm{Tr}_{K_0/\Q_p}( \zeta_0 s_1\overline{s}_2)
\]
for some $\zeta_0\in \mathcal{S}_0\otimes_\Z\Q$ satisfying $\overline{\zeta}_0=-\zeta_0$,
$p\zeta_0=(\zeta_0 b_0\overline{b}_0)^\Frob$, and $\zeta_0\mathcal{S}_0=\mathcal{S}_0$.

The $\co_{K_0}$-Hermitian form 
$\langle f_1, f_2\rangle$ on (\ref{hermitian model}) is then given by the explicit  formula 
\[
\langle s_1, s_2\rangle=\mathrm{Tr}_{K/K_0}(\zeta_0^{-1}\zeta s_1\overline{s}_2).
\]
It follows that  
$
\langle s_1, s_2\rangle_\CM=\xi s_1\overline{s}_2
$
where  $\xi= \zeta_0^{-1}\zeta$.
\end{proof}

Armed with the above explicit coordinates, we  may describe  the Hermitian space  
$L(A_0,A)$ attached to our fixed triples $(A_0,\kappa_0,\lambda_0)$
and $(A,\kappa,\lambda)$.

\begin{Prop}\label{Prop:hermitian local I}
For some $\beta\in F^\times$ satisfying 
\[
\beta \co_K=\begin{cases}
\mathfrak{a}\mathfrak{p}_F\mathfrak{D}_0\mathfrak{D}^{-1}\co_K  & \hbox{if $K_0/\Q$ is unramified}\\
\mathfrak{a}\mathfrak{D}_0\mathfrak{D}^{-1}\co_K &  \hbox{if $K_0/\Q$ is ramified}
\end{cases}
\] 
there is an isomorphism of $\co_K$-modules  $L(A_0,A) \iso \co_K$  identifying $\langle \cdot,\cdot\rangle_\CM$ with  
the Hermitian form $\langle x,y\rangle_\CM=\beta x\overline{y}$ on $\co_K$.
\end{Prop}

\begin{proof}
First we show that $L(A_0,A)$ is free of rank one over $\co_K$.  
Let 
\[
H=\End(A_0)\otimes_{\Z_p}\Q_p,
\] 
so that $H$ is a quaternion division algebra over $\Q_p$.
As $A_0$ and $A$ are supersingular  there is an isogeny
$A\to  A_0\times\cdots \times A_0$ ($n$ factors).  The Noether-Skolem theorem implies that any 
two maps $K_0\to M_n(H)$ are conjugate, and it follows that the above isogeny may be chosen to 
be $\co_{K_0}$-linear.  A choice of such isogeny allows us to identify 
\[
L(A_0,A)\otimes_{\Z_p}\Q_p \iso \Hom_{\co_{K_0}}(A_0, A_0\times\cdots \times A_0)\otimes_{\Z_p}\Q_p
\iso K_0\times \cdots K_0
\]
as $K_0$-vector spaces. Thus $L(A_0,A)$ is free of rank $n$ over $\co_{K_0}$, and hence
$L(A_0,A)$ is free of rank one over $\co_K$.

Fix an $\co_K$-module  generator $s$ of (\ref{hermitian model}),
so that  $x\cdot s \mapsto x$ defines an isomorphism $L(A_0,A)\to  \co_K$ identifying 
$\langle \cdot,\cdot\rangle_\CM$ with $\beta x\overline{y}$, where, in the notation of 
Lemma \ref{Lem:dieu coords II},
\[
\beta=\xi s\overline{s}.
\]
We know that $\xi\mathcal{S} = \mathfrak{a} \mathfrak{D}_0\mathfrak{D}^{-1}\mathcal{S}$,
and so it only remains to determine the ideal $s\overline{s}\mathcal{S}$.  

Let $e(K/K_0)$ be the ramification degree of $K/K_0$.   
Set $d=[K^u:\Q_p]$, and enumerate the maps $\co^u_K\to W$ as  
$\{ \psi^i  : i\in \Z/d\Z\}$ in such a way that $\psi^{i+1}=\Frob\circ \psi^i$. Let $\psi^i_0$ be the restriction of 
$\psi^i$ to  $\co^u_{K_0}$.  The relation $(b_0s)^\Frob = b^\Frob s$ implies 
\begin{align*}
\ord_{\psi^{i+1}}(s) &= \ord_{\psi^i}(s) - \ord_{\psi^i}(b) + \ord_{\psi^i}(b_0) \\
 &= \ord_{\psi^i}(s) - \ord_{\psi^i}(b) + e(K/K_0) \cdot \ord_{\psi^i_0}(b_0) \\
&= \ord_{\psi^i}(s) - m(\psi^i,\Phi)  + e(K/K_0) \cdot  m(\psi^i_0,\Phi_0),
\end{align*}
where the final equality is by Lemma \ref{Lem:dieu coords}.
Note that there is at least one $\psi:\co^u_K\to W$ for which  $\ord_\psi(s)=0$; otherwise
$s$ would be divisible in $L(A_0,A)$ by a uniformizing parameter of $\co_K$.
This observation and the above relation between $\ord_{\psi^{i+1}}(s)$ and $\ord_{\psi^i}(s)$
will allow us to compute $\ord_\psi(s)$ for all $\psi:\co_K^u \to W$.

If  $K_0/\Q_p$ is ramified then $m(\psi^i_0,\Phi_0)=1$. Each $\psi^i:\co^u_K\to W$ admits 
\[
[K:K^u]= 2\cdot e(K/K_0)
\] 
extensions to a map $K\to \C_p$.  Exactly half of these extensions lie in $\Phi$, and so
$m(\psi^i,\Phi)=e(K/K_0)$.  It follows that 
$\ord_{\psi^{i+1}}(s) =\ord_{\psi^i}(s)$ for every $i\in \Z/d\Z$, hence    $s\in \mathcal{S}^\times$ and 
\[
\beta \mathcal{S}=\xi s\overline{s}\mathcal{S}
=\mathfrak{a}\mathfrak{D}_0\mathfrak{D}^{-1}\mathcal{S}
\] 
as desired.

Now suppose that $K_0/\Q_p$ is unramified.    As $K_0$ does not embed into $F$,
this implies $[F^u:\Q_p]=2f+1$ for some $f\in \Z^{\ge 0}$, and so $d=4f+2$.  Assume the $\psi^i$ have 
been enumerated in such a way that $\psi^0$ is the restriction of $\varphi^\mathrm{sp}$ to $\co^u_K$.  
This implies   $\psi^0_0=\overline{\iota}$, and so
\[
m(\psi^i_0 , \Phi_0)  =
\begin{cases}
1 & \hbox{if $i$ is odd}\\
0 & \hbox{if $i$ is even.}
\end{cases}
\]
If $\varphi$ is any extension of $\psi^0$ to a map $K\to \C_p$ then the restriction of $\varphi$ to $\co_{K_0}$ is 
$\overline{\iota}$.  Therefore $\varphi\in \Phi$ if and only if $\varphi=\varphi^\mathrm{sp}$, and so
$m(\psi^0,\Phi)=1$.  This shows  
\[
\ord_{\psi^1}(s)=\ord_{\psi^0}(s)-1.
\]   
The automorphism $x\mapsto \overline{x}$
of $K$ restricts to $\Frob^{2f+1}$ on $K^u$ and so the restriction of $\overline{\varphi^\mathrm{sp}}$ to 
$\co^u_K$ is $\psi^{2f+1}$.  The map $\psi^{2f+1}:\co^u_K\to W$ then admits $[K:K^u]=e(K/K_0)$
distinct extensions to a map $K\to \C_p$, every one of which except $\overline{\varphi^\mathrm{sp}}$
is contained in $\Phi$.  Therefore $m(\psi^{2f+1},\Phi) = e(K/K_0)-1$, from which we deduce
\[
\ord_{\psi^{2f+2}}(s)=\ord_{\psi^{2f+1}}(s) +1.
\]  
Now suppose $i\in \Z/d\Z$ is not equal to $0$ or $2f+1$,
so that $\psi^i$ is not the restriction to $\co^u_K$ of either $\varphi^\mathrm{sp}$ or $\overline{\varphi^\mathrm{sp}}$.
Similar reasoning to the above shows that if $i$ is even then $m(\psi^i,\Phi)$ and $m(\psi^i_0,\Phi_0)$ are both $0$,
while if $i$ is odd then $m(\psi^i,\Phi)=e(K/K_0)$ and $m(\psi^i_0,\Phi_0)=1$.  In either case
$\ord_{\psi^{i+1}}(s)=\ord_{\psi^i}(s)$.  Recalling that $\ord_\psi(s)=0$ for at least one $\psi:\co^u_K\to W$,
we deduce first
\[
\ord_{\psi^i}(s) = 
\begin{cases}
0 & \hbox{if } 1\le i\le  2f+1 \\
1 & \hbox{if } 2f+2\le i \le d
\end{cases}
\]
and then
$
\ord_{\psi^i}(s\overline{s}) = \ord_{\psi^i}(s) + \ord_{\psi^{i+2f+1}}(s) = 1.
$
Thus $s\overline{s}\mathcal{S}=\mathfrak{p}_F\mathcal{S}$ and  
\[
\beta \mathcal{S}=\mathfrak{a} \mathfrak{p}_F\mathfrak{D}_0\mathfrak{D}^{-1}\mathcal{S}.
\]
\end{proof}

\begin{Rem}
Proposition \ref{Prop:hermitian local I} specifies  $\beta$ up to multiplication by $\co_F^\times$,
but to determine  the isomorphism class of  $(\co_K,\beta x\overline{y})$ one needs to know 
 $\beta$ up to multiplication by   $\mathrm{Nm}_{K/F}(\co_K^\times)$.  If $K/F$ is unramified 
there is no difference, and so Proposition \ref{Prop:hermitian local I} 
completely determines the isomorphism class of the pair $\big( L(A_0,A) , \langle\cdot,\cdot\rangle_\CM \big)$.
If $K/F$ is ramified there is some remaining ambiguity, as Proposition \ref{Prop:hermitian local I} 
only narrows down the isomorphism class of the pair $\big( L(A_0,A) , \langle\cdot,\cdot\rangle_\CM \big)$
to two possibilities.
\end{Rem}

Let $\mathfrak{m}$ be the maximal ideal of $W_\Phi$, and for every $k\in \Z^{>0}$ set
\[
R^{(k)}=W_\Phi/\mathfrak{m}^k.
\]  
By Theorem \ref{Thm:BT canonical} there is a unique
deformation $(A^{(k)}, \kappa^{(k)}, \lambda^{(k)})$ of $(A,\kappa,\lambda)$ to $R^{(k)}$, and a unique 
deformation $(A_0^{(k)}, \kappa_0^{(k)}, \lambda_0^{(k)})$ of $(A_0,\kappa_0,\lambda_0)$ to $R^{(k)}$.  
The image of the reduction map
\[
\Hom_{\co_{K_0}} (A_0^{(k)}, A^{(k)}) \to  L(A_0,A)
\]
is an $\co_K$-submodule $L^{(k)}$, and 
\[
\dots\subset L^{(3)}\subset L^{(2)} \subset L^{(1)} =L(A_0,A)
\]
is a decreasing filtration of $L(A_0,A)$.

The following theorem, which shows that the filtration on $L(A_0,A)$
is completely determined by the Hermitian form $\langle\cdot,\cdot\rangle_\CM$,
 generalizes a result of Gross \cite{gross86}, as explained in the remarks at the 
end of this subsection.  Gross's original proof, which can be found in an expanded form in the ARGOS volume  \cite{ARGOS-8}, 
is based on Lubin-Tate groups and the theory of formal group laws.  Our proof will be based on 
crystalline deformation theory, and is closer in spirit to Zink's proof of Gross's result, found in 
\cite[Proposition 77]{zink02}.

\begin{Thm}\label{Thm:crystal deform}
Assume that at least one of the following hypotheses is satisfied:
\begin{enumerate}
\item
$K/\Q_p$ is unramified,
\item
$p\not=2$ and one of $K_0/\Q_p$ or $F/\Q_p$ is unramified.
\end{enumerate}
For any nonzero $f\in L(A_0,A)$, $f$ is in $L^{(k)}$ but not  $L^{(k+1)}$ where $\alpha=\langle f,f\rangle_\CM$ and
\[
 k= \frac{1}{2} \cdot \ord_{\co_K} (\alpha\mathfrak{p}_F\mathfrak{D}\mathfrak{D}_0^{-1}\mathfrak{a}^{-1}).
\]
\end{Thm}

The proof, which occupies the remainder of this subsection, 
is by induction on the divisibility of $f$ by a uniformizing parameter of $\co_K$. 
Proposition \ref{Prop:lift I} serves as the base case,  and  Proposition \ref{Prop:lift II} forms  the inductive step.

Fix an injective ring homomorphism $\co_F\to M_n(\Z_p)$.
If $B_0$ is a $p$-divisible group  defined
over some base scheme $S$, denote by $B_0\otimes\co_F$ the $p$-divisible 
group $B_0^n$, and let $\co_F$ act through the embedding  $\co_F\to M_n(\Z_p)$ just chosen.  
This construction has a more intrinsic characterization: the functor of points of $B_0\otimes\co_F$ is 
\[
(B_0\otimes\co_F) (T) = B_0(T)\otimes_{\Z_p} \co_F
\]
for any $S$-scheme $T$.  If $B_0$ has an action of $\co_{K_0}$ then $B_0\otimes\co_F$ inherits
an action of the subring  $\co_{K_0}\otimes_{\Z_p} \co_F \subset\co_K$.  
If $B$ is a $p$-divisible group over $S$ with an action of $\co_K$, then every 
$\co_{K_0}$-linear homomorphism $f: B_0\to B$ induces an $\co_{K_0}\otimes_{\Z_p}\co_F$-linear homomorphism
$f:B_0\otimes\co_F\to B$.

\begin{Prop}\label{Prop:lift I}
Suppose $f$ is an $\co_K$-module generator of $L(A_0,A)$.  
\begin{enumerate}
\item
If $K_0/\Q_p$ is unramified then $f$ is in $L^{(1)}$ but not $L^{(2)}$.
\item
If $K_0/\Q_p$ is ramified  and $F/\Q_p$ is unramified 
then $f$ is in $L^{(d)}$ but not $L^{(d+1)}$, where $d=\ord_{\co_{K_0}}(\mathfrak{D}_0)$.
\end{enumerate}
\end{Prop}

\begin{proof}
Let $\mathcal{D}_0$ and $\mathcal{D}$ be the Grothendieck-Messing crystals of $A_0=A_0^{(1)}$ and 
$A=A^{(1)}$, respectively.  First assume $K_0/\Q_p$ is unramified.
The kernel $\mathcal{I}$ of $R^{(2)}\to R^{(1)}$ can be equipped with a 
divided power structure 
compatible with the canonical divided powers on $pR^{(2)}$ (take the trivial divided powers on $\mathcal{I}$ 
if $W_\Phi/W$ is  ramified, and the canonical divided powers on $\mathcal{I}=pR^{(2)}$ otherwise),
and once such divided powers are chosen we may identify, using \cite[Corollary 56]{zink02},
\[
\mathcal{D}_0(R^{(2)})  \iso  D(A_0) \otimes_W R^{(2)}  \iso  \mathcal{S}_0\otimes_W R^{(2)}  
\]
and
\[
\mathcal{D}(R^{(2)}) \iso D(A) \otimes_W R^{(2)} \iso  \mathcal{S}\otimes_W R^{(2)}.
\]
As in the proof of Theorem \ref{Thm:BT canonical} the lifts of the Hodge filtrations
of $\mathcal{D}_0(R^{(1)})$ and $\mathcal{D}(R^{(1)})$ 
corresponding to the  deformations $A_0^{(2)}$ and $A^{(2)}$ are 
$J_{\Phi_0} \mathcal{D}_0(R^{(2)})$ and $J_{\Phi}\mathcal{D}(R^{(2)})$,
and $f$ lifts to a map $A^{(2)}_0\to A^{(2)}$ if and only if 
\[
J_{\Phi_0}  \mathcal{D}_0(R^{(2)}) \map{f}  \mathcal{D}(R^{(2)})/J_{\Phi}  \mathcal{D}(R^{(2)})
\]
is trivial.  If $f$ corresponds to $s\in \mathcal{S}$ under the isomorphism (\ref{hermitian model}), 
we must therefore prove that the map
\[
J_{\Phi_0}(\mathcal{S}_0\otimes_W W_\Phi) \map{s\cdot} 
 (\mathcal{S}\otimes_W W_\Phi)/J_\Phi (\mathcal{S}\otimes_W W_\Phi)
\]
is nonzero modulo $\mathfrak{m}^2$.  But this is clear from the proof of 
Proposition \ref{Prop:hermitian local I}: if $\psi$ denotes the restriction of $\varphi^\mathrm{sp}$
to $\co^u_K\to W$ then we have already seen that $\ord_\psi(s)=1$, and so the
image of $s$ under the surjection $\varphi^\mathrm{sp}:\mathcal{S}\to W_\Phi$ is a uniformizing parameter.  
The assumption that $K_0/\Q_p$ is unramified implies that 
$\mathcal{S}_0\otimes_W W_\Phi \iso W_\Phi \times W_\Phi$, and that the composition
\[
W_\Phi\iso J_{\Phi_0}(\mathcal{S}_0\otimes_W W_\Phi) \map{s \cdot} 
 (\mathcal{S}\otimes_W W_\Phi)/J_\Phi (\mathcal{S}\otimes_W W_\Phi) \map{\varphi^{\mathrm{sp}}}W_\Phi
\]
is multiplication by $\varphi^\mathrm{sp}(s)$.

Now assume $K_0/\Q_p$ is ramified and $F/\Q_p$ is unramified, so that 
\[
\co_K=\co_{K_0}\otimes_{\Z_p}\co_F.
\]
Let $s\in \mathcal{S}$ correspond to $f$ under the isomorphism
(\ref{hermitian model}), and recall from the proof of Proposition \ref{Prop:hermitian local I} that 
$s\in \mathcal{S}^\times$.
This implies that the induced map 
\[
f:D(A_0)\otimes_{\Z_p}\co_F \to D(A)
\]
is an isomorphism of Dieudonn\'e
modules, and in particular $f$ induces an isomorphism of Lie algebras
\[
\Lie(A_0)\otimes_{\Z_p}\co_F\iso \Lie(A).
\]
If $f$ lifts to a map $f^{(k)}:A_0^{(k)}\to A^{(k)}$ then Nakayama's lemma implies that the induced map
\[
\Lie(A_0^{(k)})\otimes_{\Z_p}\co_F\iso \Lie(A^{(k)})
\]
is again an isomorphism, and comparing the $\co_K$-action on each side we find the equality in $R^{(k)}[T]$
\begin{equation}\label{cong polys}
\prod_{\varphi\in \Phi^*}(T-\varphi(x)) = \prod_{\varphi\in \Phi}(T-\varphi(x))
\end{equation}
for every $x\in \co_K$, where 
\[
\Phi^*= \{ \varphi \in \Hom(K,\C_p)  :  \varphi|_{K_0} = \iota\}
\]
is the $p$-adic CM type of $K$ obtained by replacing $\varphi^\mathrm{sp}$ by $\overline{\varphi}^\mathrm{sp}$.
Comparing the coefficients of $T^{n-1}$ shows that
$
\varphi^\mathrm{sp},\overline{\varphi}^\mathrm{sp} : \co_K\to W_\Phi
$
are congruent modulo $\mathfrak{m}^k$, which implies $k\le  \ord_{\co_{K_0}}(\mathfrak{D}_0)$.

Conversely, if $k\le \ord_{\co_{K_0}}(\mathfrak{D}_0)$ then the polynomials (\ref{cong polys}) in $R^{(k)}[T]$
are equal, and so the natural $\co_K$-action on the $p$-divisible group $A_0^{(k)}\otimes_{\Z_p}\co_F$ 
over $R^{(k)}$ satisfies the $\Phi$-determinant condition.  The map  $f:A_0\otimes_{\Z_p}\co_F\to A$
is an isomorphism of $p$-divisible groups (because it induces an isomorphism of Dieudonn\'e modules), 
and this allows us to view $A_0^{(k)}\otimes_{\Z_p}\co_F$ as a deformation of $A$ with its $\co_K$-action.
By the uniqueness of such deformations (see Remark \ref{Rem:strong lift})  there is an $\co_K$-linear 
isomorphism 
\[
A_0^{(k)}\otimes_{\Z_p}\co_F \to  A^{(k)}
\] 
 lifting  $f:A_0\otimes_{\Z_p}\co_F\to A$, and precomposing with the inclusion 
\[
A_0^{(k)}\to A^{(k)}_0\otimes_{\Z_p}\co_F
\] 
gives the  desired lift of $f:A_0\to A$.  This shows that $f$ lifts to $A_0^{(k)}\to A^{(k)}$ if and only if 
$k\le  \ord_{\co_{K_0}}(\mathfrak{D}_0)$.
\end{proof}

\begin{Prop}\label{Prop:lift II}
Let $\pi_K$ be a uniformizer of $\co_K$.  If $f\in L^{(k)}$ then $\pi_K f\in L^{(k+1)}$. 
Furthermore, if any one of the conditions
\begin{enumerate}
\item
$k> 1$
\item
$p\not=2$
\item
$K/\Q_p$ is unramified
\end{enumerate}
is satisfied then the map $\pi_K:L^{(k)}/L^{(k+1)} \to  L^{(k+1)}/L^{(k+2)}$ is injective.
\end{Prop}

\begin{proof}
The essential observation is that if 
\[
j_0\in J_{\Phi_0}=\mathrm{ker}\big( \co_{K_0}\otimes_{\Z_p}W_\Phi \to \C_p(\iota)\big)
\]
then 
\[
(x\otimes 1-1\otimes\varphi^\mathrm{sp}(x))\cdot  j_0 \in J_\Phi 
=\mathrm{ker}\big( \co_K\otimes_{\Z_p}W_\Phi  \to \prod_{\varphi\in \Phi}\C_p(\varphi)\big)
\]
for every $x\in \co_K$.  So, given  an $\co_{K_0}\otimes_{\Z_p}W_\Phi$-module $M_0$, 
an $\co_K\otimes_{\Z_p} W_\Phi$-module $M$, and an $\co_{K_0}$-linear map $f:M_0\to M$,
the induced map
\[
f :J_{\Phi_0}M_0 \to  M/J_\Phi M
\]
satisfies 
\[
(x\otimes 1 - 1\otimes \varphi^\mathrm{sp}(x)) \circ f = 0 
\] 
for all  $x\in \co_K$.

Now suppose $f:A_0\to A$ lifts to a map 
\[
f^{(k)}:A_0^{(k)}\to A^{(k)}.
\]
  Let $\mathcal{D}_0$ and
$\mathcal{D}$ be the Grothendieck-Messing crystals of $A_0^{(k)}$ and $A^{(k)}$, respectively.
By equipping the kernel of $R^{(k+1)}\to R^{(k)}$ with its trivial divided power structure, $f^{(k)}$
induces a commutative diagram
\[
\xymatrix@=16pt{
 {  J_{\Phi_0} \mathcal{D}_0(R^{(k+1)}) } \ar[r]\ar[d]  &  {\mathcal{D}_0(R^{(k+1)})} \ar[r]^{f^{(k)}}\ar[d]  & 
 {  \mathcal{D}(R^{(k+1)}) } \ar[r] \ar[d] & {  \mathcal{D}(R^{(k+1)})  / J_\Phi\mathcal{D}(R^{(k+1)}) } \ar[d]  \\
  {  J_{\Phi_0} \mathcal{D}_0(R^{(k)}) } \ar[r]  &  {\mathcal{D}_0(R^{(k)})} \ar[r]^{f^{(k)}} & 
 {  \mathcal{D}(R^{(k)}) } \ar[r] & {  \mathcal{D}(R^{(k)})  / J_\Phi\mathcal{D}(R^{(k)}) }.
 }
\]
As the middle arrow of the bottom row must preserve the Hodge filtrations of the crystals,  the  composition 
along the bottom row is trivial (see the proof of Theorem \ref{Thm:BT canonical}).
Therefore the composition along the top row
\[
J_{\Phi_0} \mathcal{D}_0(R^{(k+1)}) \to   \mathcal{D}(R^{(k+1)})  / J_\Phi\mathcal{D}(R^{(k+1)}) 
\]
becomes trivial after applying $\otimes_{R^{(k+1)}} R^{(k)}$, and so has image
annihilated  by $\mathfrak{m}$.  By the comments of the previous paragraph 
\[
\pi_Kf^{(k)} = \varphi^\mathrm{sp}(\pi_K) f^{(k)}
\] 
when viewed as maps 
\[
J_{\Phi_0} \mathcal{D}_0(R^{(k+1)}) \to   \mathcal{D}(R^{(k+1)})  / J_\Phi\mathcal{D}(R^{(k+1)}), 
\]
and as $\varphi^\mathrm{sp}(\pi_K) \in \mathfrak{m}$ we deduce that these maps are trivial.
Therefore 
\[
\pi_Kf^{(k)}: \mathcal{D}_0(R^{(k+1)}) \to \mathcal{D}(R^{(k+1)})
\]
takes the submodule $J_{\Phi_0} \mathcal{D}_0(R^{(k+1)})$ into $J_{\Phi} \mathcal{D}(R^{(k+1)})$.
By the proof of Theorem \ref{Thm:BT canonical} these submodules are the lifts of the Hodge filtrations
defining $A_0^{(k+1)}$ and $A^{(k+1)}$, and so $\pi_K f^{(k)}$ lifts to a map $A_0^{(k+1)}\to A^{(k+1)}$.

Now suppose $f$ is in  $L^{(k)}$ but not  $L^{(k+1)}$, and let $\mathcal{I}$ be the kernel of $R^{(k+2)} \to R^{(k)}$.
If $k>1$ then $\mathcal{I}^2=0$, and we may  equip $\mathcal{I}$ with its trivial divided powers.   If  $p\not=2$ 
then  $\mathcal{I}^3=0$ allows us  to  equip $\mathcal{I}$ with its trivial divided powers. If $K/\Q_p$
is unramified then $W_\Phi=W$, and we may equip $\mathcal{I}=p^kW_\Phi$ with its canonical divided powers.
In any case there is some divided power structure on $\mathcal{I}$, and so we may add a third row 
\[
\xymatrix@=16pt{
 {  J_{\Phi_0} \mathcal{D}_0(R^{(k+2)}) } \ar[r]\ar[d]  &  {\mathcal{D}_0(R^{(k+2)})} \ar[r]^{f^{(k)}}\ar[d]  & 
 {  \mathcal{D}(R^{(k+2)}) } \ar[r] \ar[d] & {  \mathcal{D}(R^{(k+2)})  / J_\Phi\mathcal{D}(R^{(k+2)}) } \ar[d]  \\
 {  J_{\Phi_0} \mathcal{D}_0(R^{(k+1)}) } \ar[r]\ar[d]  &  {\mathcal{D}_0(R^{(k+1)})} \ar[r]^{f^{(k)}}\ar[d]  & 
 {  \mathcal{D}(R^{(k+1)}) } \ar[r] \ar[d] & {  \mathcal{D}(R^{(k+1)})  / J_\Phi\mathcal{D}(R^{(k+1)}) } \ar[d]  \\
  {  J_{\Phi_0} \mathcal{D}_0(R^{(k)}) } \ar[r]  &  {\mathcal{D}_0(R^{(k)})} \ar[r]^{f^{(k)}} & 
 {  \mathcal{D}(R^{(k)}) } \ar[r] & {  \mathcal{D}(R^{(k)})  / J_\Phi\mathcal{D}(R^{(k)}) }.
 }
\]
to the diagram above.  If $\pi_K f^{(k)}$ lifts to a map $A_0^{(k+2)} \to A^{(k+2)}$ then 
\[
\pi_Kf^{(k)} : J_{\Phi_0} \mathcal{D}_0(R^{(k+2)}) \to  \mathcal{D}(R^{(k+2)})  / J_\Phi\mathcal{D}(R^{(k+2)})
\]
is trivial.  By the comments of the first paragraph this implies that
\[
\varphi^\mathrm{sp}(\pi_K)f^{(k)}:J_{\Phi_0} \mathcal{D}_0(R^{(k+2)}) \to  
\mathcal{D}(R^{(k+2)})  / J_\Phi\mathcal{D}(R^{(k+2)})
\]
is also trivial, and so 
\[
f^{(k)} : J_{\Phi_0} \mathcal{D}_0(R^{(k+2)}) \to  \mathcal{D}(R^{(k+2)})  / J_\Phi\mathcal{D}(R^{(k+2)})
\]
takes values in $\mathfrak{m}^{k+1} \cdot \mathcal{D}(R^{(k+2)})  / J_\Phi\mathcal{D}(R^{(k+2)})$.
But this implies that
\[
f^{(k)} : J_{\Phi_0} \mathcal{D}_0(R^{(k+1)}) \to  \mathcal{D}(R^{(k+1)})  / J_\Phi\mathcal{D}(R^{(k+1)})
\]
is trivial, contradicting our hypothesis that $f^{(k)}$ does not lift to a map $A_0^{(k+1)}\to A^{(k+1)}$.
Therefore $\pi_K f^{(k)}$ lifts to  $A_0^{(k+1)}\to A^{(k+1)}$ but not to  
$A_0^{(k+2)} \to A^{(k+2)}$.
\end{proof}

\begin{proof}[Proof of Theorem \ref{Thm:crystal deform}]
Let $\beta\in F^\times$ be as in Proposition \ref{Prop:hermitian local I}.  Fix a uniformizer $\pi_K\in \co_K$
and write $f=\pi_K^m f_0$ with $f_0$ an $\co_K$-module generator of $L(A_0,A)$, so that 
\[
\langle f_0, f_0\rangle_\CM \co_F =  \beta\co_F.
\]  
If $K_0/\Q_p$ is unramified then
\[
\alpha \co_K=\mathfrak{p}_F^{2m}\beta\co_K=\mathfrak{D}_0\mathfrak{D}^{-1}\mathfrak{ap}_F^{2m+1}.
\]  
Using induction on $m$, Propositions \ref{Prop:lift I} and \ref{Prop:lift II} imply that $f$ is in
$L^{(m+1)}$ but not $L^{(m+2)}$, and the claim follows.  If $K_0/\Q_p$ is ramified then
\[
\alpha \co_K=\mathfrak{p}_F^{m}\beta\co_K=\mathfrak{D}_0\mathfrak{D}^{-1}\mathfrak{ap}_F^{m}.
\]
Using induction on $m$, Proposition \ref{Prop:lift I}  (with $d=1$, as $p$ is odd) and 
Proposition \ref{Prop:lift II}  imply that $f$ is in $L^{(m+1)}$ but not $L^{(m+2)}$, and again the claim follows.
\end{proof}

Consider the special  case of $F=\Q_p$, so that $K=K_0$ and  $W_\Phi$ 
is the completion of the ring of integers in the maximal unramified extension of $K$.   We end this subsection
by explaining how, in this special case, Theorem \ref{Thm:crystal deform} reduces 
to a well-known formula of Gross \cite{gross86}, 
which  plays a crucial role in the proof of the famous Gross-Zagier formula \cite{gross-zagier}.
Assume for simplicity that $p\not=2$.

Keep $(A_0,\kappa_0,\lambda_0)$
as above, but now take $(A,\kappa,\lambda)=(A_0,\overline{\kappa}_0,\lambda_0)$, where
$\overline{\kappa}_0(x)=\kappa_0(\overline{x})$.  Suppressing $\kappa_0$ from the notation, the 
$\co_K$-module $L(A_0,A)$ now sits inside of $\End(A_0)$ as the set of  $j\in \End(A_0)$ satisfying
$j\circ x= \overline{x}\circ j$ for all $x\in \co_K$, and
\[
\End(A_0) = \co_K \oplus L(A_0,A).
\]
 Furthermore,
\[
\End(A_0^{(k)}) = \co_K\oplus L^{(k)},
\]
and so in this special case Theorem \ref{Thm:crystal deform} amounts to an explicit description of 
how the ring $\End(A_0^{(k)})$ shrinks as $k$ grows.    Fix an $\co_K$-module generator $f\in L(A_0,A)$. 
If $\pi_K$ is a uniformizing parameter of $\co_K$, then our results prove
\[
\End(A_0^{(k)}) = \co_K \oplus \co_K \pi_K^{k-1} f,
\]
which is exactly Gross's formula.


\subsection{Lifting homomorphisms: the signature $(n,0)$ case}
\label{ss:LHII}


As in the previous subsection   $K_0$ is a quadratic field extension of $\Q_p$,  $F/\Q_p$ is a field extension of 
degree  $n$, and 
\[
K=K_0\otimes_{\Q_p} F.
\]  
We now allow the possibility  $K\iso F\times F$.
 Fix an embedding $\iota:K_0\to \C_p$, so that   $\Phi_0=\{\iota\}$ is a $p$-adic CM type of $K_0$. 
A $p$-adic CM type $\Phi$ of $K$ is  of \emph{signature} $(n,0)$ if  $\varphi|_{K_0}=\iota$
for every $\varphi\in \Phi$.  Fix such a $\Phi$ (in fact, it's unique). 
 If we let  $K_\Phi \subset \C_p$ be any subfield containing $\iota(K_0)$,
then condition (\ref{reflex condition})  is satisfied.  
Let $W_\Phi$ and $\mathbf{ART}$ be   as in Section \ref{ss:canonical lifts}.

 Fix a triple $(A,\kappa,\lambda)$ in which 
 \begin{itemize}
 \item
 $A$ is a $p$-divisible group over  $\F$ of dimension $n$,
 \item
 $\kappa:\co_K\to \End(A)$ satisfies the $\Phi$-determinant condition,
 \item
$\lambda:A\to A^\vee$ is an $\co_K$-linear polarization.     
\end{itemize}
Fix a second triple 
 $(A_0,\kappa_0,\lambda_0)$ in which  
 \begin{itemize}
\item
 $A_0$ is a $p$-divisible group over  $\F$ of dimension $1$,
\item
 $\kappa_0:\co_{K_0}\to \End(A_0)$ satisfies the $\Phi_0$-determinant condition, 
 \item
$\lambda_0:A_0\to A_0^\vee$ is an $\co_{K_0}$-linear polarization.  
\end{itemize} 
By Theorem  \ref{Thm:BT canonical} each of   $(A_0,\kappa_0,\lambda_0)$ and $(A,\kappa,\lambda)$
admits a unique deformation to any object of $\mathbf{ART}$.
The  following is the signature $(n,0)$ version of  Theorem \ref{Thm:crystal deform}.  Now the 
situation is drastically simplified.

\begin{Prop}\label{Prop:simple deformation}
Let $R$ be an object of $\mathbf{ART}$, and $(A_0',\kappa_0',\lambda_0')$ and $(A',\kappa',\lambda')$
 the unique deformations of  $(A_0,\kappa_0,\lambda_0)$ and $(A,\kappa,\lambda)$ to $R$.
The reduction map
\[
\Hom_{\co_{K_0}}(A_0',A') \to  \Hom_{\co_{K_0}}(A_0,A)
\]
is a bijection.
\end{Prop}

\begin{proof}
Let $R^{(2)}\to R^{(1)}$ be any surjection in $\mathbf{ART}$ whose kernel $\mathcal{I}$
satisfies $\mathcal{I}^2=0$, and equip $\mathcal{I}$ with its trivial divided power structure.  Let
$(A_0^{(i)},\kappa_0^{(i)},\lambda_0^{(i)})$ and $(A^{(i)},\kappa^{(i)},\lambda^{(i)})$
be the unique deformations of $(A_0,\kappa_0,\lambda_0)$ and $(A,\kappa,\lambda)$ to $R^{(i)}$,
and suppose we are given an $\co_{K_0}$-linear map  $f : A_0^{(1)} \to A^{(1)}$.
Let $\mathcal{D}_0$ and $\mathcal{D}$ be the Grothendieck-Messing crystals of $A_0^{(1)}$
and $A^{(1)}$.  The map $f$ induces an $\co_{K_0}\otimes_{\Z_p}W_\Phi$-linear map on crystals
$
f:\mathcal{D}_0(R^{(2)})\to  \mathcal{D}(R^{(2)}).
$
The hypothesis that $\Phi$ has signature $(n,0)$ implies that
\[
J_{\Phi_0}(\co_K\otimes_{\Z_p}W_\Phi) \subset J_\Phi,
\]  
and therefore   $f$ satisfies
\[
f(J_{\Phi_0} \mathcal{D}_0(R^{(2)}))  =J_{\Phi_0} f(\mathcal{D}_0(R^{(2)}))  \subset J_\Phi\mathcal{D}(R^{(2)}).
\]
By the proof of Theorem \ref{Thm:BT canonical}, the deformations $A_0^{(2)}$ and 
$A^{(2)}$ of $A_0^{(1)}$ and $A^{(1)}$ correspond to the lifts 
\[
J_{\Phi_0} \mathcal{D}_0(R^{(2)}) \subset \mathcal{D}_0(R^{(2)})
\qquad 
J_{\Phi} \mathcal{D}(R^{(2)}) \subset \mathcal{D}(R^{(2)})
\]
of the Hodge filtrations of $\mathcal{D}_0(R^{(1)})$ and $\mathcal{D}(R^{(1)})$.  As $f$
preserves these filtrations,  it follows  that  $f$ lifts  (uniquely) to a map $A_0^{(2)} \to A^{(2)}$. 
The claim  follows  by induction on the length of  $R$.
\end{proof}


\section{Arithmetic intersection theory}
\label{S:global moduli}


Throughout Section \ref{S:global moduli} we fix the following data, as in the introduction: 
\begin{itemize}
\item
$K_0\subset \C$ is a quadratic imaginary field,  and $\iota:K_0\to \C$ is the inclusion,
\item
$F$ is a totally real \'etale $\Q$-algebra of degree $n$,
\item
$K=K_0\otimes_{\Q} F$,
\item
$\Phi$ is a CM type of $K$ of signature $(n-1,1)$; this means  there is a unique
 $\varphi^\mathrm{sp}\in \Phi$ whose restriction to $K_0$ is $\overline{\iota}:K_0\to \C$,
\item
 $K_\Phi\subset \C$ is a finite extension of $K_0$ containing the reflex field of $\Phi$,
 \item
$\co_\Phi$ is   the ring of integers of $K_\Phi$,
\item
$\mathfrak{a}\subset \co_F$ is an ideal (eventually we will take $\mathfrak{a}=\co_F$).
\end{itemize}
The CM type $\Phi$  is  uniquely determined by its \emph{special element} $\varphi^\mathrm{sp}$, and 
thus  $\sigma \in \Aut(\C/K_0)$  fixes $\Phi$ if and only if it fixes $\varphi^\mathrm{sp}$.
It follows  that $\varphi^\mathrm{sp}(K) \subset K_\Phi$, and  that we 
may take $K_\Phi=\varphi^\mathrm{sp}(K)$ if we choose.
In any case, every prime $\mathfrak{p}$ of $K_\Phi$  restricts, via the map 
\begin{equation}\label{special}
\varphi^\mathrm{sp}:K\to K_\Phi,
\end{equation}
to a  prime $\mathfrak{p}_K$ of $K$.    The prime of $F$ below $\mathfrak{p}_K$ 
is denoted $\mathfrak{p}_F$.   
The special element $\varphi^\mathrm{sp}\in \Phi$   determines an  archimedean place of $K$, whose 
restriction to $F$ is denoted $\infty^{\mathrm{sp}}$.

Let $\pi_0(F)$ denote the set of connected 
components of $\Spec(F)$.  The algebra $F$ is a product of totally real number fields indexed by $\pi_0(F)$, 
and each connected component has the form $\Spec(F')$ for a 
 subfield $F'\subset F$.  There is a quadratic character of $(F'_\A)^\times$ associated to the 
CM field $K'=K_0\otimes_\Q F'$, and by collecting together the quadratic characters of the different components of 
$\Spec(F)$ we obtain a  generalized character 
\begin{equation}\label{general character}
\chi_{K/F}: F_\A^\times\to \{\pm1\}^{\pi_0(F)}
\end{equation}
associated to the extension $K/F$.

For each $\mathfrak{p}\subset\co_\Phi$, fix an algebraic closure $K_{\Phi,\mathfrak{p}}^\alg$ of $K_{\Phi,\mathfrak{p}}$,  
let $\C_\mathfrak{p}$ be its completion, and let 
\[ 
W_{\Phi,\mathfrak{p}} \subset\C_\mathfrak{p}
\] 
be the ring of integers of the completion of the maximal unramified
extension of $K_{\Phi,\mathfrak{p}}$.   Denote by $k^\alg_{\Phi,\mathfrak{p}}$ the  common residue field of 
$K_{\Phi,\mathfrak{p}}^\alg$, $\C_\mathfrak{p}$, and  $W_{\Phi,\mathfrak{p}}$ 
Let $\mathfrak{D}_0$ and $\mathfrak{D}$ be the differents of  $K_0/\Q$ and  $K/\Q$, respectively,
and let $\mathfrak{d}_F$ be the different of $F/\Q$.


\subsection{The stack $\mathcal{CM}_{\Phi}^\mathfrak{a}$}
\label{ss:moduli}


Recall the moduli stack $\mathcal{M}_{(r,s)}$  of the introduction.  We now
define the  cycle of points of $\mathcal{M}_{(n-1,1)}$ with complex multiplication by $\co_K$ and 
CM type $\Phi$.  Taking $\mathfrak{a}=\co_F$ in the following definition 
gives the stack $\mathcal{CM}_{\Phi}$ of the introduction.

\begin{Def}
Let $\mathcal{CM}_{\Phi}^\mathfrak{a}$ 
be the algebraic stack over $\co_\Phi$ whose functor of points assigns to a connected
$\co_\Phi$-scheme $S$ the groupoid of triples $(A,\kappa,\lambda)$  in which 
\begin{itemize}
\item
$A\to S$ is an abelian scheme of relative dimension $n$,
\item
$\kappa:\co_K\to \End(A)$ satisfies the $\Phi$-determinant condition,
\item
$\lambda:A\to A^\vee$ is an $\co_K$-linear polarization with kernel $A[\mathfrak{a}]$.  
\end{itemize}
\end{Def}

The condition of \emph{$\co_K$-linearity} means that 
\[
\lambda\circ \kappa(\overline{x}) = \kappa(x)^\vee \circ \lambda
\]
for every $x\in \co_K$.
The \emph{$\Phi$-determinant condition}, introduced by Kottwitz \cite{kottwitz92}, is the following:
locally on $S$, for any  $x_1,\ldots, x_r \in \co_K$ the determinant of $T_1x_1+\cdots+T_rx_r$ acting on  $\Lie(A)$
is equal to the image of 
\[
\prod_{\varphi\in \Phi}( T_1\varphi(x_1) + \cdots+ T_r \varphi(x_r) ) \in \co_\Phi[T_1,\ldots, T_r]
\]
in $\co_S [T_1,\ldots,T_r]$.    
Note that this condition implies that every $x\in \co_K$ acts on $\Lie(A)$ with characteristic polynomial
\[
\prod_{\varphi\in \Phi} (T-\varphi(x))\in \co_\Phi[T],
\]  
and in particular, the action of $\co_{K_0}$ on $A$ satisfies the signature $(n-1,1)$-condition of the introduction.
If we take $\mathfrak{a}=\co_F$, then  restricting the action from $\co_K$ to $\co_{K_0}$ defines a morphism
\[
\mathcal{CM}_{\Phi}^{\co_F} \to  \mathcal{M}_{(n-1,1)/\co_\Phi}.
\]

For an  $\co_{\Phi}$-scheme $S$, an $S$-valued point of 
\[
\mathcal{M}_{(1,0)} \times \mathcal{CM}_{\Phi}^\mathfrak{a}
= \mathcal{M}_{(1,0)/\co_\Phi} \times_{\co_\Phi} \mathcal{CM}_{\Phi}^\mathfrak{a}.
\]
is a sextuple $(A_0,\kappa_0,\lambda_0, A,\kappa,\lambda)$
with $(A_0,\kappa_0,\lambda_0) \in  \mathcal{M}_{(1,0)}(S)$ and 
$(A,\kappa,\lambda) \in  \mathcal{CM}_{\Phi}^\mathfrak{a}(S)$.  We usually abbreviate this sextuple to
 $(A_0,A)$.

\begin{Prop}\label{Prop:reduction II}
Let $\mathfrak{p}$ be a prime of  $K_\Phi$, and let $\mathcal{Y}_\Phi$ denote one of 
 $\mathcal{M}_{(1,0)/\co_\Phi}$,  $\mathcal{CM}^\mathfrak{a}_{\Phi}$,  or 
 $\mathcal{M}_{(1,0)} \times \mathcal{CM}_{\Phi}^\mathfrak{a}$.
\begin{enumerate}
\item
If  $R$ is  any complete local Noetherian 
$W_{\Phi,\mathfrak{p}}$-algebra with residue field $k_{\Phi,\mathfrak{p}}^\alg$, the 
reduction map
\[
\mathcal{Y}_\Phi(R)  \to \mathcal{Y}_\Phi(k_{\Phi,\mathfrak{p}}^\alg)
\]
(on isomorphism classes) is a bijection.  
\item
The completed strictly Henselian local ring of $\mathcal{Y}_\Phi$ at  any
geometric point $z\in \mathcal{Y}_\Phi(k_{\Phi,\mathfrak{p}}^\alg)$ is isomorphic to 
$W_{\Phi,\mathfrak{p}}$.  
\item
The structure morphism $\mathcal{Y}_\Phi\to  \Spec(\co_\Phi)$ is \'etale and proper.
In particular $\mathcal{Y}_\Phi$ is a regular stack of dimension one.
\end{enumerate}
\end{Prop}

\begin{proof}
The first claim follows easily from  Theorem \ref{Thm:cm etale}.   The second follows from the first,
as the completed strictly Henselian local ring of a geometric point $z$ represents the functor 
of deformations of $z$ to complete local Noetherian $W_{\Phi,\mathfrak{p}}$-algebras.  
The \'etaleness part of the third claim can be checked on the level of completed strictly Henselian local rings,
and so follows from the second claim.   Properness follows from the valuative criterion of
properness for stacks, together with the fact that CM abelian varieties over discrete valuation rings have
potentially good reduction.
\end{proof}

\begin{Rem}\label{Rem:extra canonical}
If $\mathcal{Y}_\Phi$ is as above,
it follows from Proposition \ref{Prop:reduction II} that there is a canonical bijection
\[
\mathcal{Y}_\Phi (\C_\mathfrak{p}) \to  \mathcal{Y}_\Phi(k_{\Phi,\mathfrak{p}}^\alg)
\]
on isomorphism classes.
Indeed, each object of $\mathcal{Y}_\Phi (\C_\mathfrak{p})$ is a polarized  abelian variety
with complex multiplication (or a pair of such things).  By the theory of complex multiplication such an abelian variety admits 
a model with good reduction defined over some  finite extension of $K_{\Phi,\mathfrak{p}}$.  Reducing this 
model modulo $\mathfrak{p}$ and then base changing to $k_{\Phi,\mathfrak{p}}^\alg$ defines the desired reduction map.
For the inverse: each object of $\mathcal{Y}_\Phi(k_{\Phi,\mathfrak{p}}^\alg)$ lifts uniquely to $\mathcal{Y}_\Phi(W_{\Phi,\mathfrak{p}})$,
and we base change this lift  to $\C_\mathfrak{p}$.
\end{Rem}

\begin{Def}\label{Def:canonical}
Let $\mathfrak{p}$ be a prime of $\co_\Phi$.  The unique lift of a triple
\[
(A,\kappa,\lambda) \in  \mathcal{CM}_{\Phi}^\mathfrak{a}( k_{\Phi,\mathfrak{p}}^\alg),
\]
to $W_{\Phi,\mathfrak{p}}$ is its \emph{canonical lift}
\[
(A^\can,\kappa^\can,\lambda^\can) \in  \mathcal{CM}_{\Phi}^\mathfrak{a}( W_{\Phi,\mathfrak{p} }).
\]
Similarly, the  unique lift of a triple
\[
(A_0,\kappa_0,\lambda_0) \in  \mathcal{M}_{(1,0)}( k_{\Phi,\mathfrak{p}}^\alg),
\]
to $W_{\Phi,\mathfrak{p}}$ is its \emph{canonical lift}
\[
(A_0^\can,\kappa_0^\can,\lambda_0^\can) \in  \mathcal{M}_{(1,0)}( W_{\Phi,\mathfrak{p} }).
\]
\end{Def}

An abelian variety $A$ over an algebraically closed field of nonzero characteristic is  \emph{supersingular} 
if $A$  is isogenous  to a product  of supersingular elliptic curves.   Equivalently, by \cite[Theorem 4.2]{oort74}, 
$A$ is supersingular if its $p$-divisible group is supersingular, in the sense of Section \ref{s:deformations}.

\begin{Prop}\label{Prop:ST}
Fix a prime $\mathfrak{p}$ of $K_\Phi$, let $\mathfrak{p}_F$ be the prime of $F$ defined
after (\ref{special}), let $p$ be the rational prime below $\mathfrak{p}$, 
and assume that
 $p$ is nonsplit in $K_0$.
For any $(A,\kappa,\lambda) \in \mathcal{CM}_{\Phi}^\mathfrak{a}(k^\alg_{\Phi,\mathfrak{p}})$ the
following hold.
\begin{enumerate}
\item
If $\mathfrak{q}\subset\co_F$ is a prime above $p$ different from $\mathfrak{p}_F$, 
then  $A[\mathfrak{q}^\infty]$ is supersingular;
\item
$A[\mathfrak{p}_F^\infty]$ is supersingular  if and only if  $\mathfrak{p}_F$ is  nonsplit in $K$;
\item
$A$ is supersingular if and only if $\mathfrak{p}_F$ is nonsplit in $K$.
 \end{enumerate}
\end{Prop}

\begin{proof}
It suffices to prove the first two claims, as the third is a trivial consequence.
Following Remark \ref{Rem:extra canonical}, let 
$(A^*,\kappa^*,\lambda^*)\in \mathcal{CM}^\mathfrak{a}_{\Phi}(\C_\mathfrak{p} )$ be the unique lift of 
$(A,\kappa,\lambda)$. Fix an isomorphism of $K_\Phi$-algebras $\C_\mathfrak{p} \iso \C$, 
and view each $\varphi:K\to \C$ as taking values in $\C_\mathfrak{p}$.  Fix a prime $\mathfrak{q}\subset \co_F$
above $p$.  

For a prime $\mathfrak{Q}$ of $K$ above $\mathfrak{q}$,
let $H_\mathfrak{Q}$ be the set of  all $\Q$-algebra maps $K\to \C_\mathfrak{p}$ inducing
$\mathfrak{Q}$.  For a map $\varphi:K\to \C_\mathfrak{p}$ we define the  conjugate 
by $\overline{\varphi}(x)=\varphi(\overline{x})$, so that $H_{\overline{\mathfrak{Q}}} = \overline{H_\mathfrak{Q}}$.
 The proof of the Shimura-Taniyama formula, for example \cite[Corollary 4.3]{conrad:shimura-taniyama}, shows that
\[
\mathrm{dim}\ A[\mathfrak{Q}^\infty] = \mathrm{dim}\ A^*[\mathfrak{Q}^\infty] = \#(\Phi\cap H_\mathfrak{Q})
\]
and
\[
\mathrm{height}\ A[\mathfrak{Q}^\infty] =\mathrm{height}\ A^*[\mathfrak{Q}^\infty]
=\#H_\mathfrak{Q}.
\]
The argument used in the proof of Proposition \ref{Prop:BT basics}  shows that  the Dieudonn\'e module of
$A[\mathfrak{Q}^\infty]$ is isoclinic.  If the slope sequence  consists of  $s/t$ repeated $k$ times then  
$\mathrm{dim}\ A[\mathfrak{Q}^\infty] = sk$ and 
$\mathrm{height}\ A[\mathfrak{Q}^\infty] =tk$. It follows that  
\[
A[\mathfrak{Q}^\infty]  \hbox{ is supersingular } \iff 
 \frac{1}{2} =\frac{\#(\Phi\cap H_\mathfrak{Q})}{\#H_\mathfrak{Q}}.
\]

First consider the easy case  in which $\mathfrak{q}$ is nonsplit in $K$.  Then $\mathfrak{Q}=\overline{\mathfrak{Q}}$, 
and so $H_\mathfrak{Q}$ is the  disjoint union of $\Phi\cap H_\mathfrak{Q}$ with 
\[
\overline{\Phi}\cap H_\mathfrak{Q} = \overline{\Phi}\cap H_{\overline{\mathfrak{Q}} } = \overline{\Phi\cap H_\mathfrak{Q}},
\]
and it follows from the preceding paragraph that $A[\mathfrak{q}^\infty]$ is supersingular.

Now assume  $\mathfrak{q}$ is split in $K$.  This implies that $K_{0,p}$ embeds into $F_\mathfrak{q}$, and so
$[F_\mathfrak{q} : \Q_p]=[K_\mathfrak{Q} : \Q_p]=\#H_\mathfrak{Q}$ is even, say $\#H_\mathfrak{Q}=2d$.   Each of the
sets
\begin{align*}
H_\mathfrak{Q}(\iota) &= \{ \varphi \in H_\mathfrak{Q} : \varphi|_{K_0}=\iota \} \\
H_\mathfrak{Q}(\overline{\iota}) &= \{ \varphi \in H_\mathfrak{Q} : \varphi|_{K_0}=\overline{\iota} \} \\
H_{\overline{\mathfrak{Q}}}(\iota) &= \{ \varphi \in H_{\overline{\mathfrak{Q}}} : \varphi|_{K_0}=\iota \} \\
H_{\overline{\mathfrak{Q}}}(\overline{\iota}) &= \{ \varphi \in H_{\overline{\mathfrak{Q}}} : \varphi|_{K_0}=\overline{\iota} \}
\end{align*}
has $d$ elements.   If $\mathfrak{q}\not=\mathfrak{p}_F$ then 
$\Phi\cap  H_\mathfrak{Q}(\overline{\iota})$ and $\Phi \cap H_{\overline{\mathfrak{Q}}}(\overline{\iota})$ 
are empty, and so 
\[
H_\mathfrak{Q}(\iota) = [  \Phi \cap H_\mathfrak{Q}(\iota) ] \cup [ \overline{\Phi} \cap  H_\mathfrak{Q}(\iota) ] 
= [ \Phi\cap H_\mathfrak{Q}(\iota)] \cup [\overline{ \Phi\cap H_{\overline{\mathfrak{Q}}}(\overline{\iota}) }] 
= \Phi\cap H_\mathfrak{Q}(\iota) .
\]
This implies
\[
\Phi \cap H_\mathfrak{Q}  = [ \Phi  \cap H_\mathfrak{Q}(\iota)] \cup  [ \Phi\cap H_\mathfrak{Q}(\overline{\iota}) ]
=H_\mathfrak{Q}(\iota)
\]
and so 
\[
\frac{\#(\Phi\cap H_\mathfrak{Q})}{\#H_\mathfrak{Q}} =\frac{d}{2d}.
\]
This proves that $A[\mathfrak{Q}^\infty]$ is supersingular.  
The same argument shows that $A[\overline{\mathfrak{Q}}^\infty]$
is supersingular, and hence so is $A[\mathfrak{q}^\infty] = A[\mathfrak{Q}^\infty] \times A[\overline{\mathfrak{Q}}^\infty]$.
If $\mathfrak{q}=\mathfrak{p}_F$ then one of $\Phi\cap  H_\mathfrak{Q}(\overline{\iota})$ and 
$\Phi \cap H_{\overline{\mathfrak{Q}}}(\overline{\iota})$ is empty, and the other is $\{\varphi^\mathrm{sp}\}$.
After possibly interchanging $\mathfrak{Q}$ and $\overline{\mathfrak{Q}}$ we may assume that
$\Phi\cap  H_\mathfrak{Q}(\overline{\iota}) =\{\varphi^\mathrm{sp} \}$ and 
$\Phi \cap H_{\overline{\mathfrak{Q}}}(\overline{\iota}) = \emptyset$.
The argument above shows first that $H_\mathfrak{Q}(\iota)  =  \Phi\cap H_\mathfrak{Q}(\iota)$, and then that
\[
\Phi \cap H_\mathfrak{Q}  = [ \Phi  \cap H_\mathfrak{Q}(\iota)] \cup  [ \Phi\cap H_\mathfrak{Q}(\overline{\iota}) ]
=H_\mathfrak{Q}(\iota) \cup\{\varphi^\mathrm{sp}\}.
\]
Therefore
\[
\frac{\#(\Phi\cap H_\mathfrak{Q})}{\#H_\mathfrak{Q}} =\frac{d+1}{2d}\not=\frac{1}{2}.
\]
Therefore $A[\mathfrak{Q}^\infty]$ is not supersingular, and so neither is  $A[\mathfrak{q}^\infty]$. 
\end{proof}

The following proposition tells us that $\mathcal{CM}^\mathfrak{a}_\Phi$  is typically nonempty.

\begin{Prop}\label{Prop:s ideal}
There is a unique fractional $\co_F$-ideal  $\mathfrak{s}$ for which   
\[
\mathfrak{s}\co_K=\mathfrak{D}_0\mathfrak{D}^{-1}.
\]
If the discriminants of $K_0/\Q$ and $F/\Q$ are relatively prime then 
the category $\mathcal{CM}_{\Phi}^\mathfrak{a}(\C)$ is nonempty, and furthermore
 $\mathfrak{s}^{-1}= \mathfrak{d}_F$.
\end{Prop}

\begin{proof}
Let $\delta_0\in \widehat{K}_0^\times$ satisfy $\overline{\delta}_0=-\delta_0$ and 
$\delta_0\co_{K_0}=\mathfrak{D}_0$,
and let $\delta\in \widehat{K}^\times$ satisfy $\overline{\delta}=-\delta$ and $\delta\co_{K}=\mathrm{Diff}(K/F)$.
There is $c\in \widehat{F}^\times$ such that $\delta_0=c\delta$, and setting $\mathfrak{c}=c\co_F$ the 
existence of the ideal $\mathfrak{s}$
follows from  $\mathfrak{D}_0 \mathfrak{D}^{-1}=\mathfrak{c} \cdot  \mathrm{Diff}(F/\Q)^{-1}$, 
and the uniqueness is clear. 
Now assume that $K_0/\Q$ and $F/\Q$ have relatively prime discriminants.  This implies that
$\mathfrak{D}_0\co_K = \mathrm{Diff}(K/F)$, and the equality $\mathfrak{s}^{-1}=\mathfrak{d}_F$
follows easily.  

Taking the product over all $\varphi\in \Phi$ yields an isomorphism of 
$\R$-vector spaces $K_\R\iso \C^n$, and allows us to view $K_\R$ as a complex vector space.  
Let $\zeta\in K^\times$ be any element satisfying $\overline{\zeta}=-\zeta$.  Using weak approximation
we may multiply $\zeta$ by an element of $F^\times$ in order to assume that $\varphi(\zeta)\cdot i >0$ for every
$\varphi\in \Phi$.  Then $\lambda(x,y)=\mathrm{Tr}_{K/\Q}(\zeta x\overline{y})$ defines an $\R$-symplectic 
form on $K_\R$, and $\lambda(i\cdot x,x)$ is positive definite.  
Class field theory implies that the norm map from the ideal class group of $K$ to the 
narrow ideal class group of $F$ is surjective  ($K_0/\Q$ and $F/\Q$ have
relatively prime discriminants,  and so $K/F$  is ramified at some finite prime;  therefore
the Hilbert class field of $K$ and the narrow Hilbert class field of $F$ are linearly disjoint over $F$).  
It follows that there is a fractional $\co_K$-ideal $\mathfrak{A}$
and a $u\in F^{\gg 0}$ satisfying $u \mathfrak{A}\overline{\mathfrak{A}} = \zeta^{-1}\mathfrak{D}^{-1}\mathfrak{a}.$
Replacing $\zeta$ by $\zeta u^{-1}$ we may therefore assume 
$\zeta \mathfrak{A}\overline{\mathfrak{A}} = \mathfrak{a}\mathfrak{D}^{-1}$, and so 
\[
\mathfrak{a}^{-1} \mathfrak{A}=\{ x\in K_\R :  \lambda(x,\mathfrak{A})\subset \Z \}.
\]
The Riemann form $\lambda$  defines a polarization of the complex torus $K_\R/\mathfrak{A}$,
and the kernel of this polarization is the subgroup $\mathfrak{a}^{-1} \mathfrak{A}/\mathfrak{A}$
of $\mathfrak{a}$-torsion points.  This proves that $\mathcal{CM}_{\Phi}^\mathfrak{a}(\C)\not=\emptyset$.
\end{proof}


\subsection{The space $L(A_0,A)$: first results}
\label{ss:hermitian spaces}


Suppose we are given a connected $\co_\Phi$-scheme $S$ and a pair 
\[
(A_0, A)\in ( \mathcal{M}_{(1,0)} \times \mathcal{CM}_{\Phi}^\mathfrak{a} )  (S).
\]
   The $\co_{K_0}$-module
 \[
 L(A_0,A) = \Hom_{\co_{K_0}}(A_0,A)
 \]
carries a natural  positive definite $\co_{K_0}$-Hermitian form  \cite[Lemma 2.8]{KRunitaryII} defined by
\[
\langle f_1,f_2\rangle = \lambda_0^{-1} \circ f_2^\vee \circ \lambda \circ f_1,
\]
and the action of $\co_K$ on $A$ determines an action of $\co_K$ on $L(A_0,A)$ satisfying
\[
\langle x\cdot f_1 , f_2\rangle = \langle f_1, \overline{x}\cdot f_2\rangle
\] 
for every $x\in \co_K$.  It follows that there is a unique $K$-valued totally positive definite 
$\co_K$-Hermitian form $\langle f_1, f_2\rangle_\CM$ on  $ L(A_0,A)$  for which 
\[
 \langle f_1, f_2\rangle = \mathrm{Tr}_{K/K_0}\langle f_1, f_2\rangle_\CM .
\]
Set
\[
V(A_0,A) = L(A_0,A) \otimes_\Z\Q.
\]

Recall Serre's  twisting construction, as in    \cite[Section 7]{conrad04}.
 Suppose we are given a scheme $S$, an abelian scheme $B\to S$, an action $\co \to \End(B)$ of an order 
 in a number field,  and a projective $\co$-module $\mathfrak{Z}$.  To this data we may attach a new abelian scheme 
 $\mathfrak{Z} \otimes_\co B$ over $S$.  This abelian scheme is determined by its functor of points
 \[
(\mathfrak{Z}  \otimes_\co B ) (T) = \mathfrak{Z}  \otimes_\co B(T)
 \]
 for any $S$-scheme $T$.

 The following proposition shows that $V(A_0,A)$ is rather small, unless $A_0$ and $A$ are supersingular.

\begin{Prop}\label{Prop:supersingular}
Suppose $k$ is an algebraically closed field, and 
\[
(A_0,A) \in (\mathcal{M}_{(1,0)} \times \mathcal{CM}_\Phi^\mathfrak{a}) (k).
\]
If there is an $f\in V(A_0,A)$ such that $\langle f,f\rangle_\CM \in F^\times$, 
then $k$ has nonzero characteristic, and $A_0$ and $A$ are supersingular.
\end{Prop}

\begin{proof}
 The map $f$ induces an $\co_{K_0}$-linear map  
$f_F:  \co_F \otimes_\Z A_0  \to A.$
  Fix  a prime $\ell\nmid \mathrm{char}(k)$, and for any abelian variety $B$  over $k$  let 
\[
\mathrm{Ta}^0_\ell(B)=\mathrm{Ta}_\ell(B)\otimes_{\Z_\ell}\Q_\ell
\] 
be the rational $\ell$-adic Tate module.
  The  polarization $\lambda_0$  induces a perfect $\Q_\ell$-linear pairing
\[
\lambda_0:\mathrm{Ta}^0_\ell(A_0) \times \mathrm{Ta}^0_\ell(A_0) \to \Q_\ell(1),
\]
and tensoring with $F_\ell$ results in a perfect  $F_\ell$-linear pairing
\[
\Lambda_0:\mathrm{Ta}^0_\ell(\co_F  \otimes_\Z A_0) \times \mathrm{Ta}^0_\ell(\co_F  \otimes_\Z A_0) \to F_\ell(1).
\]
The polarization $\lambda$ induces a perfect pairing 
\[
\lambda :\mathrm{Ta}^0_\ell(A) \times \mathrm{Ta}^0_\ell(A) \to \Q_\ell(1),
\] 
which  has the form $\lambda=\mathrm{Tr}_{F/\Q}\Lambda$ for a unique $F_\ell$-linear 
\[
\Lambda:\mathrm{Ta}^0_\ell(A) \times \mathrm{Ta}^0_\ell(A) \to F_\ell(1).
\]

The \emph{adjoint} of  
\begin{equation}\label{adjoint}
f_F:\mathrm{Ta}^0_\ell(\co_F  \otimes_\Z A_0)\to \mathrm{Ta}^0_\ell(A)
\end{equation} 
is the unique
$f_F^\dagger :\mathrm{Ta}^0_\ell(A) \to \mathrm{Ta}^0_\ell(\co_F  \otimes_\Z A_0)$ for which
$\Lambda_0(x,f_F^\dagger y)=\Lambda(f_Fx,y)$,
 and some  linear algebra shows that $\langle f ,f \rangle_\CM = f_F^\dagger \circ f_F$ as elements of 
 \[
 F_\ell \subset \End_{\Q_\ell}( \mathrm{Ta}^0_\ell(\co_F  \otimes_\Z A_0)   ) .
 \]
The hypothesis $\langle f,f\rangle_\CM\in F^\times$ now implies that  (\ref{adjoint})
is injective, and it follows that  $f_F: \co_F  \otimes_\Z A_0 \to A$ is an  isogeny. 
Thus  we have  $\co_{K_0}$-linear isogenies
\[
A\sim \co_F  \otimes_\Z A_0 \sim \underbrace{A_{0} \times \cdots\times A_{0}}_{n\hbox{ times}}.
\]
As in the proof of   \cite[Lemma 2.22]{KRunitaryII}, the signature conditions imposed on $A_0$ and 
$A$ now  imply that  $\mathrm{char}(k)>0$ and  that $A_0$ and $A$ are supersingular.  
\end{proof}

\begin{Prop}\label{Prop:BT hermite switch}
Suppose $k$ is an algebraically closed field of nonzero characteristic, and
\[
(A_0,A) \in (\mathcal{M}_{(1,0)} \times \mathcal{CM}_\Phi^\mathfrak{a}) (k)
\]
with $A_0$ and $A$ supersingular.  Then $L(A_0,A)$ is a projective $\co_K$-module of rank one.
Furthermore, if $q$ is a rational prime (which may or may not equal the characteristic of $k$), and 
 $\mathfrak{q}$ is a prime of  $F$ above $q$, then the natural map 
\[
L(A_0,A) \otimes_{\co_F} \co_{F,\mathfrak{q}} \to \Hom_{\co_{K_0}}(A_0[q^\infty] , A[\mathfrak{q}^\infty])
\]
is an isomorphism.  Here $A_0[q^\infty]$ and $A[\mathfrak{q}^\infty]$ are the $q$-divisible groups of 
$q$-power and $\mathfrak{q}$-power torsion in $A_0$ and $A$.
\end{Prop}

\begin{proof}
An argument using the 
Noether-Skolem theorem (as in the  beginning of the proof of Proposition  \ref{Prop:hermitian local I}) 
shows there is an $\co_{K_0}$-linear isogeny 
\[
A\to \underbrace{A_0\times \cdots \times A_0}_{n\mbox{ times}}.
\]   
Fixing such an isogeny determines an injection with finite cokernel
\[
\Hom_{\co_{K_0}} (A_0,A) \to  \Hom_{\co_{K_0}}(A_0, A_0\times\cdots \times A_0) 
 \iso \co_{K_0} \times\cdots \times \co_{K_0},
\] 
and we deduce that $\Hom_{\co_{K_0}} (A_0,A)$ is a projective $\co_{K_0}$-module of rank $n$.  
For the same reason $\Hom_{\co_{K_0}} (A_0[q^\infty] ,A[q^\infty])$ is a projective $\co_{K_0,q}$-module of rank 
$n$. The natural map
\[
\Hom_{\co_{K_0}} (A_0,A)  \otimes_\Z \Z_q \to \Hom_{\co_{K_0}} (A_0[q^\infty] , A[q^\infty])
\]
is injective with $\Z_q$-torsion free cokernel, and hence
is an isomorphism, as both sides have the same $\Z_q$-rank.  It follows easily that
\[
L(A_0,A) \otimes_{\co_F} \co_{F,\mathfrak{q}} \to \Hom_{\co_{K_0}} (A_0[q^\infty],A[\mathfrak{q}^\infty])
\]
is an isomorphism.

We now prove that  $L(A_0,A)$ is a projective $\co_K$-module of rank one.  Of course
if $K$ is a field this is obvious, as we know from the previous paragraph 
that $L(A_0,A)$ is a torsion-free $\Z$-module of the same rank as $\co_K$.  
The point is to rule out the possibility that the action of $\co_K$ on $L(A_0,A)$ factors through
projection to a proper direct summand of $\co_K$.
Fix a prime $q\not=\mathrm{char}(k)$.  The argument used in the proof of \cite[Lemma 1.3]{rapoport78} shows that
the $q$-adic Tate modules  $\mathrm{Ta}_q(A_0)$ and 
$\mathrm{Ta}_q(A)$ are free of rank one over $\co_{K_0,q}$ and $\co_{K,q}$, respectively,
and combining this with the paragraph above shows that
\[
L(A_0,A)\otimes_\Z\Z_q \iso \Hom_{\co_{K_0}}(\mathrm{Ta}_q(A_0) , \mathrm{Ta}_q(A)) \iso \co_{K,q}.
\]
As $L(A_0,A)$ is $\Z$-torsion free, this is enough to show that $L(A_0,A)$ is projective of 
rank one.
\end{proof}


\subsection{Twisting Hermitian spaces}
\label{ss:twisting}


In the next subsection we will determine the structure of the Hermitian space 
$L(A_0,A)$ of Proposition \ref{Prop:BT hermite switch} more explicitly.  In this 
subsection we first recall some elementary properties of Hermitian spaces.  Suppose $L$
is a projective $\co_K$-module of rank one, $V=L\otimes_{\co_K} K$, and $H$
is a nondegenerate $K$-Hermitian form on $V$.   By fixing a  $K$-linear isomorphism $V\iso K$ we see that
\[
( L , H ) \iso (\mathfrak{A},\alpha x\overline y)
\]
for some $\alpha\in F^\times$ and some fractional $\co_K$-ideal $\mathfrak{A}$.  
Of course $\alpha x\overline{y}$ is shorthand for the Hermitian form $(x,y)\mapsto \alpha x\overline{y}$.
For any place $v$ of $F$, let \[\chi_v:F_v^\times \map{}\{\pm 1\}\] be the quadratic character 
associated to the extension $K_v/F_v$.  The \emph{local invariant} of $( L , H )$
at $v$ is $\chi_v(\alpha)$.
If $v$ is archimedean then knowing the local invariant at $v$ is equivalent to knowing the 
signature of $(V,H)$ at $v$.
The collection of local invariants determines the space $(V,H)$
up to isomorphism.    If we choose an $\widehat{\co}_K$-linear
isomorphism $\widehat{L}\iso \widehat{\co}_K$ then 
\[
(\widehat{L},H) \iso (\widehat{\co}_K,\beta x\overline{y})
\]
for some $\beta\in \widehat{F}^\times$ satisfying $\chi_v(\alpha)=\chi_v(\beta)$ for all finite $v$,
and satisfying
\[
\beta\co_F =\alpha\mathfrak{A}\overline{\mathfrak{A}}
\]
Define the \emph{ideal} of $(L,H)$ to be the fractional $\co_F$-ideal $\beta\co_F$.
Given another projective $\co_K$-module of rank one $L'$ and a Hermitian form $H'$,
we say that the pairs $(L,H)$ and $(L',H')$ \emph{belong to the same genus}
if they have the same signature at every archimedean place, and if 
$( \widehat{L} , H ) \iso ( \widehat{L}' , H' ).$
It is not hard to see that the genus of $(L,H)$ is completely determined by
\begin{itemize}
\item
the ideal $\beta\co_F$,
\item
the local invariant at every finite prime of $F$ ramified in $K$,
\item
the signature at every archimedean place.
\end{itemize}

There is a natural group action on the set of all isomorphism classes of pairs $(L,H)$.
 Let $I_K$ be the set of all pairs  $\mathfrak{z}=(\mathfrak{Z},\zeta)$ where $\mathfrak{Z}$ is a fractional $\co_K$-ideal 
 and $\zeta\in F^{\gg 0}$  satisfies $\zeta \mathfrak{Z}\overline{\mathfrak{Z}}=\co_K$.
The set $I_K$ is a group under componentwise multiplication, and has a natural subgroup 
\[
P_K = \{ (z^{-1} \co_K,z\overline{z}) :  z\in K^\times\}.
\]   
Denote by $C_K=I_K/P_K$ the quotient group.    Given a $\mathfrak{z}\in C_K$ and a pair
$(L,H) $ as above, define a new pair
\[
\mathfrak{z} \action (L,H) = ( \mathfrak{Z}L, \zeta H) .
\]
The ideal and signature of $(L,H)$ are obviously unchanged by this action, and the finite group $C_K$ acts
simply transitively on the set of isomorphism classes of pairs with the same ideal and signature
as $(L,H)$.

The action of $C_K$ does not preserve the genus of $(L,H)$, but it has a natural subgroup that does.
Define an algebraic group over $F$
\[
H= \mathrm{ker} ( \mathrm{Nm}: K^\times\to F^\times ),
\] 
a compact open subgroup
\[
U=\ker(\mathrm{Nm}: \widehat{\co}_K^\times \to \widehat{\co}_F^\times  ) \subset H(\widehat{F}), 
\]
and a finite group 
\[
C_K^0 = H(F)\backslash H(\widehat{F}) / U.
\]
The rule $h \mapsto (h \co_K,1)$ defines an injection $C_K^0 \to  C_K$ whose
image  is the \emph{genus subgroup} of $C_K$.  Let
\begin{equation}\label{eta}
\eta : \widehat{\co}_F^\times/ \mathrm{Nm}_{K/F}(\widehat{\co}_K^\times) \to  \{\pm 1\}^{\pi_0(F)}
\end{equation}
be the restriction to $\widehat{\co}_F^\times$ of the  character  (\ref{general character}).
   There is an exact sequence
\begin{equation}\label{genus sequence}
1\to C_K^0\to C_K \map{\mathrm{gen}} \widehat{\co}_F^\times/ \mathrm{Nm}_{K/F}(\widehat{\co}_K^\times)
\map{\eta}\{\pm 1\}^{\pi_0(F)}
\end{equation}
where the middle arrow (the \emph{genus invariant}) is defined 
as follows: given $\mathfrak{z}\in I_K$ choose a finite idele $z\in \widehat{K}^\times$ such that 
$z\co_K = \mathfrak{Z}$ and set  
\[
\mathrm{gen}(\mathfrak{z}) = \zeta z\overline{z}.
\]
A simple calculation shows that
\[
\mathfrak{z} \action ( \widehat{L} ,H) \iso  ( \widehat{L} , \mathrm{gen}(\mathfrak{z})\cdot H),
\]
and it follows easily that $C_K^0$ acts simply transitively on the genus of $(L,H)$.

For us,  the usefulness of the action of $C_K$ on Hermitian spaces is that it is compatible with 
 the twisting construction of Serre.  
 Suppose $S$ is  a connected $\co_\Phi$-scheme and 
 \[
 (A,\kappa,\lambda)\in \mathcal{CM}^\mathfrak{a}_\Phi(S).
 \]
Given $\mathfrak{z}=(\mathfrak{Z},\zeta) \in I_K$,  the abelian scheme 
\[
A^\mathfrak{z}=\mathfrak{Z}\otimes_{\co_K} A
\] 
 carries a natural $\co_K$-action 
$
\kappa^\mathfrak{z}:\co_K\to \End(A^\mathfrak{z})
$ 
defined by $\kappa^\mathfrak{z}(x) = \mathrm{id}\otimes \kappa(x)$,
which  again satisfies the $\Phi$-determinant condition.  There is a  quasi-isogeny
\[
s \in \Hom_{\co_K} (A^\mathfrak{z},A)\otimes_\Z\Q
\]
defined by $s (z\otimes a)=\kappa (z)\cdot a$, and the composition
\[
\lambda^\mathfrak{z}=   s^\vee \circ \lambda\circ  \kappa(\zeta) \circ s
\]
is an $\co_K$-linear  polarization of $A^\mathfrak{z}$ with kernel $A^\mathfrak{z}[\mathfrak{a}]$.
For a proof that $\lambda\circ  \kappa(\zeta)$, and hence $\lambda^\mathfrak{z}$, is a polarization, 
see \cite[Proposition 1.17]{rapoport78}; it is here where
we must assume  $\zeta\gg 0$.  We obtain a new object
\[
(A^\mathfrak{z},\kappa^\mathfrak{z} ,\lambda^\mathfrak{z}) \in  \mathcal{CM}^\mathfrak{a}_\Phi(S),
\]
and in this way the  group $C_K$  acts on  the set of isomorphism classes of objects in 
$\mathcal{CM}^\mathfrak{a}_{\Phi}(S)$.

Now fix a pair 
\[
(A_0,A) \in (\mathcal{M}_{(1,0)} \times \mathcal{CM}_\Phi^\mathfrak{a})(S).
\]
Let  $\langle f_1,f_2\rangle_\CM^\mathfrak{z}$ denote the $\co_K$-Hermitian form on $L(A_0,A^\mathfrak{z})$.
By \cite[Lemma 7.14]{conrad04}, the function  $f\mapsto s \circ f$ defines an isomorphism of $\co_K$-modules
\begin{equation}\label{twist space}
 L(A_0,A^\mathfrak{z}) \iso  \mathfrak{Z} \cdot L(A_0,A)
\end{equation}
identifying
$\langle \cdot   ,  \cdot\rangle_\CM^\mathfrak{z} = \zeta\cdot  \langle \cdot , \cdot \rangle_\CM.$
In other words,
\[
\mathfrak{z}\bullet ( L(A_0,A) , \langle\cdot,\cdot\rangle_\CM ) 
\iso ( L(A_0,A^\mathfrak{z}) , \langle\cdot,\cdot\rangle_\CM^\mathfrak{z} )
\]
(at least  assuming that $L(A_0,A)$ is projective of rank one, the only case in which 
we have defined the action $\mathfrak{z}\bullet$).

Here is the form in which these results will be used.

\begin{Prop}\label{Prop:hermite twist}
Suppose $S$ is  a connected $\co_\Phi$-scheme and 
\[
(A_0,A) \in (\mathcal{M}_{(1,0)} \times \mathcal{CM}_\Phi^\mathfrak{a})(S).
\]
For any  $\mathfrak{z}\in C_K$ there is an isomorphism of $\widehat{\co}_K$-modules
\[
\widehat{L}(A_0,A^\mathfrak{z}) \iso \widehat{L}(A_0,A)
\]
identifying the Hermitian form  $ \langle \cdot  , \cdot  \rangle_\CM^\mathfrak{z}$ on the left with 
the form  $\mathrm{gen}(\mathfrak{z})  \langle \cdot , \cdot  \rangle_\CM$ on the right.
 \end{Prop}
 
\begin{proof}
Fix a representative $(\mathfrak{Z},\zeta)\in I_K$ of $\mathfrak{z}$, and 
a  $z\in \widehat{K}^\times$ satisfying $z\co_K=\mathfrak{Z}$.   Using multiplication by $z$
to identify 
\[
\widehat{L}(A_0,A) \iso \mathfrak{Z}\cdot  \widehat{L}(A_0,A),
\] 
and using (\ref{twist space}), we obtain an isomorphism
\[
 \widehat{L}(A_0,A)  \iso  \widehat{L}(A_0,A^\mathfrak{z}) 
\]
denoted $f\mapsto f^\mathfrak{z}$, where $f^\mathfrak{z} = s^{-1} \circ \kappa(z^{-1}) \circ f$.
This isomorphism satisfies
\[
 \langle f_1^\mathfrak{z} , f_2^\mathfrak{z} \rangle_\CM^\mathfrak{z}
 = \zeta z\overline{z} \cdot \langle f_1,f_2\rangle_\CM,
\]
as desired.
\end{proof}


\subsection{Calculation of $L(A_0,A)$}


We now proceed to compute $L(A_0,A)$ in particular cases, the most important being the case
where $A_0$ and $A$ are  supersingular.

First consider the situation in characteristic $0$. For a pair 
\[
(A_0, A)\in ( \mathcal{M}_{(1,0)} \times \mathcal{CM}_{\Phi}^\mathfrak{a} )  (\C)
\]
 the space $L(A_0,A)$ is rather small.  For example if $F$ is a field, it follows from 
 Proposition \ref{Prop:supersingular} that $L(A_0,A)=0$.  As a substitute for this space, 
we replace  $A_0$ and $A$ by their first homology groups
\[
 H_1(A_0) =H_1(A_0(\C),\Z) \qquad  H_1(A)=H_1(A(\C),\Z)
\]
and define
\begin{equation}\label{betti space}
L_\mathrm{B}(A_0,A)= \Hom_{\co_{K_0}} ( H_1(A_0) ,  H_1(A) ).
\end{equation}
The polarizations of $A_0$ and $A$ induce symplectic forms on $H_1(A_0)$ and $H_1(A)$, which 
we view as $\Z$-module maps from $H_1(A_0)$ and $H_1(A)$ to their $\Z$-duals.
The $\co_K$-module $L_B(A_0,A)$ is then endowed with an $\co_{K_0}$-Hermitian form $\langle\cdot,\cdot\rangle$,
and an $\co_K$-Hermitian form $\langle\cdot,\cdot\rangle_\CM$ defined exactly as for $L(A_0,A)$.
One may think of $L_B(A_0,A)$ as the space of 
$\co_{K_0}$-linear maps of \emph{real} Lie groups $A_0(\C) \to  A(\C)$, and so 
there is an obvious  injection of Hermitian  $\co_K$-modules 
\[
L(A_0,A) \to  L_\mathrm{B}(A_0,A).
\] 
Abbreviate
\[
V_B(A_0,A) = L_B(A_0,A)\otimes_\Z\Q.
\]

The structure of $L_B(A_0,A)$ is quite easy to describe.
Recall that the fractional $\co_F$-ideal $\mathfrak{s}$ was defined
in Proposition \ref{Prop:s ideal}.

\begin{Prop}\label{Prop:betti hermitian}
Suppose
\[
(A_0, A)\in ( \mathcal{M}_{(1,0)} \times \mathcal{CM}_{\Phi}^\mathfrak{a} )  (\C).
\]
 There is a $\beta\in \widehat{F}^\times$ satisfying  
$\beta \co_F=\mathfrak{a}\mathfrak{s}$, and an isomorphism
\[
 \big(  \widehat{L}_B(A_0,A)  ,   \langle \cdot,\cdot\rangle_\CM  \big) \iso  \big(  \widehat{\co}_K , \beta x\overline{y} \big).
 \]
Furthermore, the $\co_K$-Hermitian form $\langle \cdot , \cdot \rangle_\CM$ 
is negative definite at the  place $\infty^{\mathrm{sp}}$ defined
after (\ref{special}), and positive definite at all other archimedean places of $F$.
\end{Prop}

\begin{proof}
This follows from the classical theory of CM abelian varieties over $\C$.
For some  fractional $\co_K$-ideal  $\mathfrak{A}$ there is an isomorphism of $\co_K$-modules
$\mathfrak{A}\iso H_1(A)$, and  the polarization $\lambda$ determines a symplectic pairing on  $\mathfrak{A}$
of the form 
\[
\lambda(x,y)=\mathrm{Tr}_{K/\Q}(\zeta x\overline{y}),
\] 
where $\zeta\in K^\times$ satisfies $\overline{\zeta}=-\zeta$ and
$\zeta\mathfrak{A}\overline{\mathfrak{A}}=\mathfrak{aD}^{-1}$. The real vector space $\mathfrak{A}_\R$
is canonically identified with $\Lie(A)$, and hence comes with a complex structure for which 
the quadratic form $\lambda(ix,x)$  is positive definite.  This last condition is equivalent to  $\varphi(\zeta)\cdot i>0$
for every $\varphi\in \Phi$.  

Similarly, for some fractional $\co_{K_0}$-ideal $\mathfrak{A}_0$ there is an isomorphism
of $\co_{K_0}$-modules $\mathfrak{A}_0\iso H_1(A_0)$, and the symplectic form on $\mathfrak{A}_0$ induced by the 
polarization $\lambda_0$ has the form 
\[
\lambda_0(x,y)=\mathrm{Tr}_{K_0/\Q} (\zeta_0 x\overline{y})
\] 
for some $\zeta_0\in K_0^\times$ satisfying $\overline{\zeta}_0=-\zeta_0$, 
$\zeta_0\mathfrak{A}_0\overline{\mathfrak{A}}_0=\mathfrak{D}_0^{-1}$,
and $\iota(\zeta_0)\cdot i>0$.

There are now  isomorphisms of $\co_K$-modules
\[
 \mathfrak{A}_0^{-1}\mathfrak{A} \iso \Hom_{\co_{K_0}}(\mathfrak{A}_0,\mathfrak{A})  \iso L_B(A_0,A), 
\]
and under these identifications the $\co_K$-Hermitian  form  on $L_B(A_0,A)$ 
is identified with the $\co_K$-Hermitian form $\zeta_0^{-1}\zeta x\overline{y}$ on $\mathfrak{A}_0^{-1}\mathfrak{A}$.
If $\varphi\in \Phi$ with $\varphi\not=\varphi^\mathrm{sp}$ then 
$\varphi(\zeta_0^{-1}\zeta) >0$, while $\varphi^\mathrm{sp}(\zeta_0^{-1}\zeta)<0$.
This shows that $\langle f,f\rangle_\CM$ is negative  definite at $\infty^{\mathrm{sp}}$ and 
positive definite at all other archimedean  places of $F$.  The rest follows by fixing an $\widehat{\co}_K$-linear
isomorphism $\mathfrak{A}_0^{-1}\mathfrak{A}\widehat{\co}_K\iso \widehat{\co}_K$.
\end{proof}

\begin{Rem}\label{Rem:L_B}
Of course Proposition \ref{Prop:betti hermitian} does not determine the Hermitian space $L_B(A_0,A)$
up to isomorphism, nor does it even determine the genus of $L_B(A_0,A)$.  In the 
terminology of Section \ref{ss:twisting}, Proposition \ref{Prop:betti hermitian} tells us the ideal of $L_B(A_0,A)$, and
the signature at every archimedean place, and so only determines the $C_K$-orbit of $L_B(A_0,A)$.
Let $\mathcal{L}_B$ denote the set of isomorphism classes of pairs $(L,H)$ where 
\begin{itemize}
\item
$L$ is a projective $\co_K$-module of rank one, 
\item
$H$ is a $K$-valued $\co_K$-Hermitian form on $L$,
\item
$(L,H)$ has ideal $\mathfrak{as}$, in the terminology of Section \ref{ss:twisting}, 
 \item
$(L,H)$  is negative definite at $\infty^\mathrm{sp}$ and positive definite at all other archimedean places of $F$.
 \end{itemize}
 This is a transitive $C_K$-set, and Proposition \ref{Prop:betti hermitian} tells us that every 
 $L_B(A_0,A)$ lies in $\mathcal{L}_B$.  The discussion preceding 
 Proposition \ref{Prop:hermite twist} applies equally well to $L_B(A_0,A)$,  and shows that 
\[
\mathfrak{z}\bullet ( L_B(A_0,A) , \langle\cdot,\cdot\rangle_\CM ) 
\iso ( L_B(A_0,A^\mathfrak{z}) , \langle\cdot,\cdot\rangle_\CM^\mathfrak{z} )
\]
for any $\mathfrak{z}\in C_K$.  Thus as the pair $(A_0,A)$ varies, we obtain every element of 
$\mathcal{L}_B$.  In this sense, Proposition \ref{Prop:betti hermitian} is as sharp as possible.
\end{Rem}

The remainder of this subsection is devoted to the proof of the following theorem, which 
similarly determines the $C_K$-orbit of  the Hermitian space $(L(A_0,A) , \langle \cdot,\cdot\rangle_\CM)$
at a supersingular point.

\begin{Thm}\label{Thm:global hermitian}
Suppose $\mathfrak{p}$ is a prime of $K_\Phi$ for which $\mathfrak{p}_F$ is nonsplit in $K$, and suppose
\[
(A_0,A ) \in  ( \mathcal{M}_{(1,0)} \times \mathcal{CM}_{\Phi}^\mathfrak{a})  (k^\alg_{\Phi,\mathfrak{p}}).
\]
There is an  isomorphism  
\[
\big( \widehat{L}(A_0,A)  , \langle \cdot , \cdot \rangle_\CM  \big)  \iso
\big(  \widehat{\co}_K , \beta x\overline{y} \big)
\]
 for some  $\beta\in \widehat{F}^\times$ satisfying
 \[
 \beta \co_F=\mathfrak{asp}_F^{\epsilon_p}.
 \]
 Here $p$ is the rational prime below $\mathfrak{p}$, and $\epsilon_p$ is defined by (\ref{epsilon}).
Furthermore, if we view $\beta\in F_\A^\times$ with  trivial archimedean components then $\chi_{K/F}(\beta)=1$. 
\end{Thm}

\begin{proof}
 The  pair  $(A_0,A)$  is necessarily supersingular: $A$ is supersingular by
Proposition \ref{Prop:ST}, and $A_0$ is supersingular as $p$ is nonsplit in $K_0$.
We will determine the structure of $\big( L(A_0,A), \langle\cdot,\cdot\rangle_\CM  \big)$ by
exploiting the fact that the pair  $(A_0,A)$ has a canonical lift, in the sense of Definition \ref{Def:canonical}.
  This will allow us  to reduce most of the 
calculation of $L(A_0,A)$ to a calculation in characteristic $0$, where 
Proposition \ref{Prop:betti hermitian} applies.

By Remark \ref{Rem:extra canonical} there is a unique lift of $(A_0,A)$ to a pair
\[
(A_0^\prime, A^\prime) \in (  \mathcal{M}_{(1,0)} \times \mathcal{CM}_{\Phi}^\mathfrak{a} ) (\C_\mathfrak{p}).
\]
After fixing an isomorphism of $K_\Phi$-algebras $\C_\mathfrak{p}\iso \C$,   
we may view $(A_0^\prime,A^\prime)$  also as a pair
\begin{equation}\label{lift}
(A^\prime_0,A^\prime)  \in   (  \mathcal{M}_{(1,0)} \times \mathcal{CM}_{\Phi}^\mathfrak{a}  )  (\C).
\end{equation}
The comparison between $L(A_0,A)$ and $L_B(A_0^\prime,A^\prime)$ now proceeds by replacing $A_0$, $A$, 
$A_0^\prime$, and $A^\prime$ by their Barsotti-Tate groups.
Suppose  $\mathfrak{q}\subset\co_F$ is a prime lying above a rational
prime $q$ (which may or may not equal $p$).  The  $\co_{K,\mathfrak{q}}$-module
\[
L_{\mathfrak{q}}(A_0,A)= \Hom_{\co_{K_0}} (A_{0}[q^\infty] ,  A[\mathfrak{q}^\infty] )
\]
comes equipped with a $K_\mathfrak{q}$-valued $\co_{K,\mathfrak{q}}$-Hermitian form $\langle f_1,f_2\rangle_\CM$ 
defined exactly as above.  Similarly, define an
$\co_{K,\mathfrak{q}}$-Hermitian space
\[
L_{\mathfrak{q}}(A^\prime_0,A^\prime)= \Hom_{\co_{K_0}} (A^\prime_{0}[q^\infty] ,  A^\prime[\mathfrak{q}^\infty] ).
\]
 There are  isomorphism of Hermitian $\co_{K,\mathfrak{q}}$-modules
\begin{align}\label{betti local}
L_\mathrm{B}(A^\prime_0,A^\prime) \otimes_{\co_K} \co_{K,\mathfrak{q}}  
&  \iso L_\mathfrak{q}(A^\prime_0,A^\prime) \\
L(A_0,A) \otimes_{\co_K} \co_{K,\mathfrak{q}} 
& \iso L_\mathfrak{q}(A_0,A). \nonumber
\end{align}
The first is obvious, as the $q$-divisible groups of $A_0^\prime$ and $A^\prime$ are constant, and 
isomorphic to $H_1(A_0') \otimes_\Z \Q_q/\Z_q$ and $H_1(A') \otimes_\Z \Q_q/\Z_q$, respectively.
The second isomorphism is part of the statement of Proposition \ref{Prop:BT hermite switch}.
These isomorphisms, together with the following lemma, allow us to convert information about
$L_B(A_0',A')$ to information about $L(A_0,A)$.

\begin{Lem}\label{Lem:Hermitian lift}
Suppose $\mathfrak{q}\subset\co_F$ is a prime  with $\mathfrak{q} \not=\mathfrak{p}_F$.  There is 
an $\co_{K}$-linear isomorphism
\begin{equation}\label{reduction iso}
L_\mathfrak{q}(A^\prime_0,A^\prime) \iso L_\mathfrak{q}(A_0,A)
\end{equation}
respecting the Hermitian forms.
\end{Lem}

\begin{proof}
Let $q$ be the rational prime below $\mathfrak{q}$.
If $q\not=p$ then the $q$-adic Tate modules of $A^\prime$ and $A$ are canonically isomorphic,
and similarly for the $q$-adic Tate modules of $A^\prime_0$ and  $A_0$.  Therefore
\[
\Hom_{\co_{K_0}}(A_0^\prime[q^\infty] , A^\prime[q^\infty] ) \iso \Hom_{\co_{K_0}}(A_0[q^\infty] , A[q^\infty] )
\]
and (\ref{reduction iso}) follows by taking $\mathfrak{q}$-parts.

Now suppose $q=p$, so $\mathfrak{q}$ lies above $p$.  Let $\Phi(\mathfrak{q})$ be the set of all $\varphi\in \Phi$ which,
when viewed as a map $K\to \C_\mathfrak{p}$, induce the prime $\mathfrak{q}$.  The hypothesis that 
$\mathfrak{q} \not=\mathfrak{p}_F$ implies that $\varphi^\mathrm{sp}\not\in \Phi(\mathfrak{q})$, 
and so every $\varphi\in \Phi(\mathfrak{q})$ satisfies $\varphi|_{K_0} = \iota$.  In the terminology of
Section \ref{ss:LHII}, $\Phi(\mathfrak{q})$ is a $p$-adic CM type of $K_\mathfrak{q}$ of signature $(m,0)$, where
$m=[F_\mathfrak{q}:\Q_p]$.  Furthermore, the $p$-divisible group $A[\mathfrak{q}^\infty]$, 
with its action of $\co_{K,\mathfrak{q}}$, satisfies the $\Phi(\mathfrak{q})$-determinant condition of 
Section \ref{ss:canonical lifts}.   By Proposition \ref{Prop:simple deformation} the reduction map
\[
\Hom_{\co_{K_0}} ( A^\can_0 [ p^\infty ], A^\can[\mathfrak{q}^\infty] )
\to  
\Hom_{\co_{K_0}} ( A_0 [p^\infty ], A[\mathfrak{q}^\infty] )
\]
is an isomorphism.  Strictly speaking,  Proposition \ref{Prop:simple deformation}
deals with deformations to Artinian quotients of $W_{\Phi,\mathfrak{p}}$, but
one may pass to the limit  by applying 
 \cite[Theorem 3.4]{conrad04} to   truncated $p$-divisible groups. 

The pair   $(A_0^\prime,A^\prime)$ is the image of $(A_0^\can,A^\can)$ under base change through
$W_{\Phi,\mathfrak{p}} \to \C_\mathfrak{p}$, and base change  defines an injection 
 \[
\Hom_{\co_{K_0}} ( A^\can_0 [ p^\infty ], A^\can[\mathfrak{q}^\infty] )
\to  \Hom_{\co_{K_0}} ( A^\prime_0 [p^\infty ], A^\prime[\mathfrak{q}^\infty] )
 \]
whose image is, by Tate's theorem \cite[p.~181]{tate67},  the submodule of invariants
for the action of   $\Aut(\C_{\mathfrak{p}}/ W_{\Phi,\mathfrak{p}})$.  In particular the 
 cokernel is $\Z_p$-torsion free.  We have now constructed an injection
 \[
 L_\mathfrak{q}(A_0,A) \to  L_\mathfrak{q}(A_0^\prime,A^\prime)
 \]
 with $\Z_p$-torsion free cokernel.  But Propositions \ref{Prop:betti hermitian} and \ref{Prop:BT hermite switch},
 together with the isomorphisms (\ref{betti local}), imply that 
 the domain and codomain are free of rank one over $\co_{K,\mathfrak{q}}$, and so 
 this map is an isomorphism.  It is clear from the construction  that it respects 
 the Hermitian forms.
\end{proof}

It only remains to collect the pieces together. 
Let $\mathfrak{q}$ be a prime of $F$.  If $\mathfrak{q}\not=\mathfrak{p}_F$ then (\ref{betti local}),
 and Lemma \ref{Lem:Hermitian lift} tell us that 
\[
L_B(A'_0 , A' ) \otimes_{\co_F}\co_{F,\mathfrak{q}} \iso L(A_0,A) \otimes_{\co_F}\co_{F,\mathfrak{q}},
\]
and so by Proposition \ref{Prop:betti hermitian} there is an isomorphism
\[
L (A_0 , A ) \otimes_{\co_F}\co_{F,\mathfrak{q}}  \iso \co_{K,\mathfrak{q}}
\]
identifying  $\langle \cdot,\cdot \rangle_\CM$ with $\beta_\mathfrak{q}x\overline{y}$ for some
$\beta_\mathfrak{q}\in F^\times_\mathfrak{q}$ satisfying
$\beta_\mathfrak{q}\co_{F,\mathfrak{q}} = \mathfrak{as}\co_{F,\mathfrak{q}}$.

If $\mathfrak{q}=\mathfrak{p}_F$ then, as in the proof of Lemma \ref{Lem:Hermitian lift},
let $\Phi(\mathfrak{q})$ be the set of all $\varphi\in \Phi$ which, when viewed as a map $K\to \C_\mathfrak{p}$,
induce the prime $\mathfrak{q}$.  The assumption that $\mathfrak{q}=\mathfrak{p}_F$ implies that 
$\varphi^\mathrm{sp}\in \Phi(\mathfrak{q})$, and the $p$-adic CM type $\Phi(\mathfrak{q})$ of $K_\mathfrak{q}$
has signature, in the terminology of  Section \ref{ss:LHI},  $(m-1,1)$ where $m=[F_\mathfrak{q}:\Q_p]$.  
The $p$-divisible group $A[\mathfrak{q}^\infty]$, with its action of $\co_{K,\mathfrak{q}}$, satisfies
the $\Phi(\mathfrak{q})$-determinant condition, and so the results of  Section \ref{ss:LHI} apply.
In particular, Proposition \ref{Prop:hermitian local I} and (\ref{betti local}) give  isomorphisms 
\[
 L(A_0,A) \otimes_{\co_F} \co_{F,\mathfrak{q}}\iso   L_\mathfrak{q}(A_0,A)\iso \co_{K,\mathfrak{q}},
\] 
which identify
$\langle f_1,f_2\rangle_\CM$ with $\beta_\mathfrak{q}x\overline{y}$ for some 
$\beta_\mathfrak{q}\in F^\times_\mathfrak{q}$ satisfying
$\beta_\mathfrak{q}\co_{F,\mathfrak{q}} = \mathfrak{asp}_F^{\epsilon_p}\co_{F,\mathfrak{q}}$.

Setting $\beta=\prod_\mathfrak{q}\beta_\mathfrak{q}$, we have now shown that there is an isomorphism
\[
\widehat{L}(A_0,A) \iso \widehat{\co}_K
\]
identifying $\langle\cdot,\cdot\rangle_\CM$ with $\beta x\overline{y}$.
It only remains to show that  $\chi_{K/F}(\beta)=1$.  We know that $V(A_0,A)$ is 
 a free $K$-module of rank one, equipped with a positive definite Hermitian form.
 It follows that  for some $\beta^*\in F^{\gg 0}$ there is an isomorphism 
$V(A_0,A)\iso K$ identifying $\langle \cdot, \cdot \rangle_\CM$ with $\beta^* x\overline{y}$.  
Certainly $\chi_{K/F}(\beta^*)=1$, and $\beta$ and $\beta^*$ differ everywhere locally by a norm from 
$K_\A^\times$.  Therefore also $\chi_{K/F}(\beta)=1$, completing the proof of Theorem \ref{Thm:global hermitian}.
\end{proof}

The following proposition is not needed in the proofs of our main results, but it is 
illuminating, and follows easily from what has been said.

\begin{Prop}\label{Prop:invariant switch}
Let $(A_0,A)$ be as in Theorem \ref{Thm:global hermitian}, and let $(A_0',A')$ be 
as in (\ref{lift}).
 The $K$-Hermitian spaces $V_B(A_0',A')$ and  $V(A_0,A)$ are isomorphic locally at a 
 place $v$ of $F$ if and only if  $v\not\in \{\infty^{\mathrm{sp}} , \mathfrak{p}_F\}$.
\end{Prop}

\begin{proof}
The set of places of $F$ at which the Hermitian spaces in question are not isomorphic is finite of even cardinality.  
As the second is totally positive definite, Proposition \ref{Prop:betti hermitian} implies that 
they are isomorphic at all archimedean places except $\infty^\mathrm{sp}$.  Therefore the 
set of finite places of $F$ at which they are not isomorphic has odd cardinality.  By Lemma \ref{Lem:Hermitian lift}
they are isomorphic at all finite places $\mathfrak{q}\not=\mathfrak{p}_F$, and it follows that $\mathfrak{p}_F$
is the unique finite place at which they are not isomorphic.
\end{proof}

One may interpret Proposition \ref{Prop:invariant switch} as follows.  Recall the collection 
of $\co_K$-Hermitian spaces $\mathcal{L}_B$ of Remark \ref{Rem:L_B}, and 
define a  collection of rank one $K$-Hermitian spaces 
\[
\mathcal{V}_B = \{ (L\otimes_{\co_K}K,H) : (L,H)\in \mathcal{L}_B \}.
\]
This is precisely the collection of Hermitian spaces $V_B(A_0',A')$  that appear as the pair
$(A_0,A)$ varies in Theorem \ref{Thm:global hermitian}.
A rank one Hermitian space is determined by the collection of local invariants at all places of $F$,
and for each space in $\mathcal{V}_B$ one can construct a new Hermitian space
by changing the invariant both at $\infty^\mathrm{sp}$ and  at $\mathfrak{p}_F$.  If we denote by
 $\mathcal{V}_B(\mathfrak{p})$ the set of Hermitian spaces obtained from $\mathcal{V}_B$ in this way,
then as the pair $(A_0,A)$ varies in Theorem \ref{Thm:global hermitian}, the  Hermitian spaces $V(A_0,A)$ vary
over $\mathcal{V}_B(\mathfrak{p})$.


\subsection{The stack $\mathcal{Z}^\mathfrak{a}_{\Phi}(\alpha)$}
\label{ss:zero stack}


If $S$ is an $\co_\Phi$-scheme, then to each $S$-valued point 
\[
(A_0,A) \in  (  \mathcal{M}_{(1,0)} \times \mathcal{CM}_{\Phi}^\mathfrak{a}  ) (S)
\]
we have associated an $\co_K$-module $L(A_0,A)$  equipped with an $\co_K$-Hermitian form
$\langle\cdot,\cdot\rangle_\CM$.

\begin{Def}
For any $\alpha\in F$ let $\mathcal{Z}^\mathfrak{a}_{\Phi}(\alpha)$ be the algebraic 
stack over $\co_\Phi$ classifying triples $(A_0,A,f)$ over $\co_\Phi$-schemes $S$ in which 
\begin{itemize}
\item
$( A_0, A)  \in  (  \mathcal{M}_{(1,0)} \times \mathcal{CM}_{\Phi}^\mathfrak{a})  (S)$,
\item
$f\in L(A_0,A)$ satisfies $\langle f,f\rangle_\CM =\alpha$.
\end{itemize}
If $\alpha=\co_F$ we omit it from the notation.
\end{Def}

The evident forgetful morphism  
\[
\mathcal{Z}^\mathfrak{a}_{\Phi}(\alpha) \to  \mathcal{M}_{(1,0)} \times \mathcal{CM}_{\Phi}^\mathfrak{a}
\] 
 is finite and unramified, by the proof of \cite[Proposition 2.10]{KRunitaryII}.

\begin{Prop}\label{Prop:zero cycle}
Suppose $\alpha\in F^\times$.
\begin{enumerate}
 \item
 The stack  $\mathcal{Z}^\mathfrak{a}_{\Phi}(\alpha)$ 
has dimension zero,  is supported in nonzero characteristic, and every geometric point is  supersingular. 
Furthermore, $\mathcal{Z}_{\Phi}^\mathfrak{a}(\alpha)$ is empty unless $\alpha$ is totally positive.
\item
If  $\mathfrak{p}$ is a prime of $K_\Phi$ for which
$
\mathcal{Z}_{\Phi}^\mathfrak{a}(\alpha)(k_{\Phi,\mathfrak{p}}^\alg)\not=\emptyset,
$
then $\mathfrak{p}_F$ is nonsplit in $K$.
\end{enumerate}
\end{Prop}

\begin{proof}
Suppose  $(A_0,A, f)\in \mathcal{Z}^\mathfrak{a}_{\Phi}(\alpha)(k)$ with $\alpha\in F^\times$
and $k$ an algebraically closed field.  As $\langle f,f\rangle_\CM=\alpha$, Proposition \ref{Prop:supersingular}
shows that $k$ has nonzero characteristic, and that $A_0$ and $A$ are supersingular.
The supersingularity of $A_0$ implies that  $p$ is nonsplit in $K_0$,  and  Proposition \ref{Prop:ST} 
then implies $\mathfrak{p}_F$ is nonsplit in $K$.
Next we show that $\mathcal{Z}^\mathfrak{a}_{\Phi}(\alpha)$ has dimension $0$.  
Suppose $\mathfrak{p}$ is a prime of $\co_\Phi$
and $z\in \mathcal{Z}^\mathfrak{a}_{\Phi}(\alpha)(k^\alg_{\Phi,\mathfrak{p}})$ is a geometric point.
The forgetful morphism 
\[
\mathcal{Z}^\mathfrak{a}_{\Phi}(\alpha) \to   \mathcal{M}_{(1,0)} \times \mathcal{CM}_{\Phi}^\mathfrak{a}
\]
is unramified, and so  induces a surjection on completed strictly Henselian local rings.
Proposition \ref{Prop:reduction II} now implies that $\widehat{\co}_{\mathcal{Z}^\mathfrak{a}_{\Phi}(\alpha),z}$
is a quotient of  $W_{\Phi,\mathfrak{p}}$.  As  $\mathcal{Z}^\mathfrak{a}_{\Phi}(\alpha)$ has no geometric points in 
characteristic $0$, this quotient  has dimension $0$.

The only thing left to prove is that $\mathcal{Z}^\mathfrak{a}_{\Phi}(\alpha)=\emptyset$ unless $\alpha\gg 0$.
This is clear from the fact that $\langle\cdot,\cdot\rangle_\CM$ is totally positive definite.
\end{proof}

The following theorem essentially counts the number of geometric points of 
$\mathcal{Z}^\mathfrak{a}_{\Phi}(\alpha)$.

\begin{Thm}\label{Thm:point count}
Suppose $\alpha\in F^{\gg 0}$ and assume $\mathcal{CM}_{\Phi}^\mathfrak{a}(\C)\not=\emptyset$.
If $\mathfrak{p}$ is a prime of $K_\Phi$ for which $\mathfrak{p}_F$ is nonsplit in $K$, then
\[
 \sum_{  (A_0,A,f) \in \mathcal{Z}^\mathfrak{a}_{\Phi}(\alpha)(k_{\Phi,\mathfrak{p}}^\alg) } 
 \frac{1}{\# \Aut(A_0,A,f)}  
 = \frac{ h(K_0)}{w(K_0)}  \cdot \rho\left( \frac{\alpha\co_F}{\mathfrak{asp}_F^{\epsilon_p}}\right)
\]
where $p$ is the rational prime below $\mathfrak{p}$.  Recall that $\mathfrak{s}$ was defined in 
Proposition \ref{Prop:s ideal}, $\epsilon_p$ was defined by (\ref{epsilon}), $\rho$ was defined by
(\ref{rho}), $h(K_0)$ is the class number of $K_0$, and $w(K_0)$ is the number of roots of unity in $K_0$.
\end{Thm}

\begin{proof}
As an abelian variety over $\C$ with complex multiplication admits a model over a number field
having everywhere good reduction,  the hypothesis $\mathcal{CM}_{\Phi}^\mathfrak{a}(\C)\not=\emptyset$
implies that  $\mathcal{CM}_{\Phi}^\mathfrak{a}(k_{\Phi,\mathfrak{p}}^\alg)\not=\emptyset$.  
As $\mathcal{M}_{(1,0)}(\C)$ has $h(K_0)$ elements, we similarly have 
$\mathcal{M}_{(1,0)}(k_{\Phi,\mathfrak{p}}^\alg)\not=\emptyset$.  Fix a pair
\[
(A_0,A)\in ( \mathcal{M}_{(1,0)} \times   \mathcal{CM}^\mathfrak{a}_{\Phi}  )  (k_{\Phi,\mathfrak{p}}^\alg).
\]  
Using (\ref{twist space})  we compute
\begin{eqnarray*}\lefteqn{
\sum_{\mathfrak{z}\in C_K^0} 
\# \{ f \in L( A_0, A^\mathfrak{z}) :  \langle f,f\rangle_\CM^\mathfrak{z}=\alpha \}  =
\sum_{\mathfrak{z}\in C_K^0}
 \sum_{\substack{ x \in V( A_0,A) \\  \langle x,x \rangle_\CM=\alpha  }  }
\mathbf{1}_{\mathfrak{Z} L(A_0,A)} (x) } \\
& = &
\sum_{ h \in H(F)\backslash H(\widehat{F})/U}
 \sum_{\substack{ x\in V( A_0,A) \\  \langle x,x\rangle_\CM=\alpha  }  }
\mathbf{1}_{\widehat{L}(A_0,A)} (h^{-1} x) \\
& = & 
\#(H(F)\cap U)
\sum_{ h \in  H(\widehat{F})/U}
 \sum_{\substack{ x\in H(F)\backslash V( A_0,A) \\  \langle x,x\rangle_\CM=\alpha  }  }
\mathbf{1}_{\widehat{L}(A_0,A)} (h^{-1} x) .
\end{eqnarray*}
Here and elsewhere, $\mathbf{1}$ means characteristic function.
If $\mu(K)$ denotes the group of roots of unity in $\co_K$, then
$\Aut(z) \iso \mu(K)$ for any $z\in \mathcal{CM}^\mathfrak{a}_{\Phi}(k^\alg_{\Phi,\mathfrak{p}})$,
and so  \[\Aut(A_0,A^\mathfrak{z}) \iso \mu(K_0)\times \mu(K).\]
Furthermore $\mu(K) \iso H(F)\cap U$, and we have now proved
\[
\sum_{\mathfrak{z}\in C_K^0} 
\sum_{ \substack{   f \in L( A_0, A^\mathfrak{z}) 
\\ \langle f,f\rangle_\CM^\mathfrak{z}=\alpha }  }
\frac{ w(K_0)} {  \#\Aut(A_0,A^\mathfrak{z}) } 
=     \sum_{ h \in  H(\widehat{F})/U}
 \sum_{\substack{ x\in H(F)\backslash V( A_0,A) \\  \langle x,x\rangle_\CM=\alpha  }  }
\mathbf{1}_{\widehat{L}(A_0,A)} (h^{-1} x).
\]
If there are no $x\in V(A_0,A)$ satisfying $\langle x,x\rangle_\CM =\alpha$ then of course 
the right hand side is $0$.  If there are such $x$, then they are permuted simply transitively by $H(F)$,
and so
\begin{equation} \label{pre-orbital}
\sum_{\mathfrak{z}\in C_K^0} 
\sum_{ \substack{   f \in L( A_0, A^\mathfrak{z}) 
\\ \langle f,f\rangle_\CM^\mathfrak{z}=\alpha }  }
\frac{ 1} {  \#\Aut(A_0,A^\mathfrak{z}) } 
=    \frac{1}{w(K_0)} \sum_{ h \in  H(\widehat{F})/U}
\mathbf{1}_{\widehat{L}(A_0,A)} (h^{-1} x)
\end{equation}
where on the right we have fixed one $x\in V(A_0,A)$ satisfying $\langle x,x\rangle_\CM =\alpha$.

We interrupt the proof for a definition.

\begin{Def}
For any $\alpha\in \widehat{F}^\times$, define the \emph{orbital integral}
\[
O_\alpha(A_0,A) = \sum_{h\in H(\widehat{F})/U}  \mathbf{1}_{ \widehat{L}(A_0,A)}  (h^{-1} \cdot x)
\]
where $x\in \widehat{V}(A_0,A)$ satisfies $\langle x,x\rangle_\CM =\alpha$. If such  $x$  exist
then  $H(\widehat{F})$ permutes them  simply transitively, so the orbital integral is independent of the choice.  
If no such  $x$ exists then set $O_\alpha(A_0,A) =0$.   
\end{Def}

Using this new notation, (\ref{pre-orbital}) may
be rewritten as
\[
\sum_{\mathfrak{z}\in C_K^0} 
\sum_{ \substack{   f \in L( A_0, A^\mathfrak{z}) 
\\ \langle f,f\rangle_\CM^\mathfrak{z}=\alpha }  }
\frac{1 } {  \#\Aut(A_0,A^\mathfrak{z}) } =   \frac{1}{w(K_0)}\cdot O_\alpha(A_0,A).
\]
It follows from Proposition \ref{Prop:hermite twist} that
\[
O_\alpha(A_0,A^\mathfrak{z}) =  O_{\mathrm{gen}(\mathfrak{z})^{-1}\alpha}(A_0,A)
\]
for any $\mathfrak{z}\in C_K$, and so summing over $\mathfrak{z}\in C_K/C_K^0$ 
and using the exactness of (\ref{genus sequence}) shows that
\begin{equation}\label{reduction to orbital}
 \sum_{\mathfrak{z}\in C_K} 
\sum_{ \substack{   f \in L( A_0,A^\mathfrak{z}) 
\\ \langle f,f\rangle_\CM^\mathfrak{z}=\alpha }  }
\frac{1 }
{  \#\Aut(A_0,A^\mathfrak{z}) }
= \frac{ 1}{w(K_0)} \sum_{ \xi \in\mathrm{ker}(\eta)} O_{\xi\alpha}(A_0,A)
\end{equation}
where the sum is over $\xi$ in the kernel of (\ref{eta}).

 \emph{Assuming} that $\widehat{V}(A_0, A)$ represents $\alpha$, 
Theorem \ref{Thm:global hermitian} reduces the calculation of  $O_\alpha(A_0,A)$ to a pleasant exercise, as in 
\cite[Section 2.5]{howard-yangA}.  We interrupt  our proof yet again to state the result as a lemma.

\begin{Lem}\label{Lem:orbital}
Let $\beta$ be as in the  statement of  Theorem \ref{Thm:global hermitian}.
For any $\alpha \in \widehat{F}^\times$
\[
O_{\alpha}(A_0,A) = \begin{cases}
 \rho (\alpha\beta^{-1}\co_F )
 &\hbox{if $\widehat{V}(A_0,A)$ represents $\alpha$} \\
 0& \hbox{otherwise}.
\end{cases}
\]
\end{Lem}

\begin{proof}
Assume that $\widehat{V}(A_0,A)$ represents $\alpha$, and fix an $x\in \widehat{K}$ such that 
$\alpha=\beta x\overline{x}$.  The orbital integral factors as  product of local integrals $O_{\alpha,v}(A_0,A)$, 
one for each finite place $v$ of $F$,  defined by
\[
O_{\alpha,v}(A_0,A) = \sum_{h\in H(F_v) / U_v} \mathbf{1}_{\co_{K,v}}(h^{-1} x_v).
\]
If $v$ is nonsplit in $K$ then $H(F_v) / U_v =1$ and 
\[
O_{\alpha,v}(A_0,A) = 
\begin{cases}
1 & \hbox{if }\alpha_v\beta_v^{-1}\in \co_{F,v} \\
0 &\hbox{otherwise}.
\end{cases}
\]
If $v$ is split in $K$ then $K_v\iso F_v\times F_v$.  After fixing a uniformizer $\varpi\in F_v$ we find that
$H(F_v)/U_v$ is the cyclic group generated by $ (\varpi,\varpi^{-1}) \in F_v^\times \times F_v^\times$, and 
\[
O_{\alpha,v}(A_0,A) = 
\begin{cases}
1+\ord_v(\alpha_v\beta_v^{-1})  & \hbox{if }\alpha_v\beta_v^{-1}\in \co_{F,v} \\
0 &\hbox{otherwise}.
\end{cases}
\]
In either case $O_{\alpha,v}(A_0,A)$ is the number of  ideals $\mathfrak{C}_v\subset \co_{K,v}$ satisfying
\[
\beta_v\mathfrak{C}_v\overline{\mathfrak{C}}_v = \alpha\co_{F,v},
\] 
and therefore, recalling  the definition (\ref{rho}) of $\rho(\mathfrak{b})$,  we have  proved
\[
O_{\alpha}(A_0,A) =   \rho (\alpha\beta^{-1}\co_F )
\]
completing the proof of the lemma.
\end{proof}

Now go back to our fixed  $\alpha\in F^{\gg 0}$, and assume that
 $ \rho (\alpha\beta^{-1}\co_F )\not=0$.  This implies that 
 $\alpha\co_F=\beta\mathfrak{C}\overline{\mathfrak{C}}$
for some $\co_K$-ideal $\mathfrak{C}$, and it follows that there is a unique
\[
\xi \in \widehat{\co}_F^\times/\mathrm{Nm}_{K/F}\widehat{\co}_K^\times
\]
such that $\xi \alpha$ is represented by the quadratic form 
$\beta x\overline{x}$ on $\widehat{K}$.  Recalling that
$\chi_{K/F}(\beta)=1$ and that $\alpha \gg 0$, this $\xi$ lies in the kernel of (\ref{eta}).   
In other words there is a unique
$\xi \in \mathrm{ker}(\eta)$ such that $\widehat{V}(A_0,A)$ represents $\xi\alpha$.
Using   $\beta\co_F = \mathfrak{asp}_F^{\epsilon_p}$ we now deduce 
\begin{equation}\label{orbit count}
\sum_{\xi\in \mathrm{ker}(\eta) }
O_{\xi \alpha}(A_0,A) =  \rho\left( \frac{\alpha\co_F}{\mathfrak{asp}_F^{\epsilon_p}}\right).
\end{equation}
If, on the other hand, $ \rho (\alpha\beta^{-1}\co_F )=0$, then $\widehat{V}(A_0,A)$ does not represent
$\xi \alpha$ for any $\xi\in \widehat{\co}_F^\times$, and both sides of (\ref{orbit count}) are zero.
Comparing with (\ref{reduction to orbital}) shows that
\begin{equation}
 \sum_{\mathfrak{z}\in C_K} 
\sum_{ \substack{   f \in L( A_0,A^\mathfrak{z}) 
\\ \langle f,f\rangle_\CM^\mathfrak{z}=\alpha }  }
\frac{1 }
{  \#\Aut(A_0,A^\mathfrak{z}) }
= \frac{ 1}{w(K_0)}\cdot \rho\left( \frac{\alpha\co_F}{\mathfrak{asp}_F^{\epsilon_p}}\right). \label{near count}
\end{equation}

The action of $C_K$ on the set of isomorphism classes of
$\mathcal{CM}_{\Phi}^\mathfrak{a}(k_{\Phi,\mathfrak{p}}^\alg)$ is simply transitive.  
For example, one can first prove this in characteristic $0$ using the complex uniformization of CM abelian
varieties, and then use Remark \ref{Rem:extra canonical}  to deduce the result 
over $k_{\Phi,\mathfrak{p}}^\alg$.  The same argument shows that  there are $h(K_0)$ isomorphism classes of 
objects in $\mathcal{M}_{(1,0)}(k_{\Phi,\mathfrak{p}}^\alg)$.  Therefore  (\ref{near count}) implies
\[
\sum_{   \substack{  A_0  \in   \mathcal{M}_{(1,0)}  (k_{\Phi,\mathfrak{p}}^\alg)  \\ 
A  \in   \mathcal{CM}^\mathfrak{a}_{\Phi}(k_{\Phi,\mathfrak{p}}^\alg) } }
\sum_{ \substack{   f \in L( A_0,A) 
\\ \langle f,f\rangle_\CM=\alpha }  } \frac{1 }{  \#\Aut(A_0,A) } 
=\frac{ h(K_0) }{w(K_0)}\cdot \rho\left( \frac{\alpha\co_F}{\mathfrak{asp}_F^{\epsilon_p}}\right),
\]
and Theorem \ref{Thm:point count} follows.
\end{proof}


\subsection{The  degree of $\mathcal{Z}_\Phi^\mathfrak{a}(\alpha)$}


Throughout this subsection we assume that \emph{the discriminants of $K_0/\Q$ and $F/\Q$ are odd and relatively prime.} 
The primary reason for this assumption  is so that we may apply Theorem 
\ref{Thm:crystal deform}, the secondary reason is so that Proposition \ref{Prop:s ideal} applies.

\begin{Def}\label{Def:arakelov} 
For any $\alpha\in F$ for which $\mathcal{Z}^\mathfrak{a}_{\Phi}(\alpha)$ has dimension $0$, 
define the \emph{Arakelov degree}
\[
\widehat{\deg}\, \mathcal{Z}^\mathfrak{a}_{\Phi}(\alpha) = 
 \sum_{ \mathfrak{p} \subset \co_\Phi}   \frac{  \log(\mathrm{N}(\mathfrak{p}))  }{[K_\Phi:\Q]}
\sum_{z\in \mathcal{Z}_{\Phi}( \alpha ) (k^\alg_{\Phi,\mathfrak{p}})  }
\frac{\mathrm{length}(\co^\mathrm{sh}_{\mathcal{Z}_{\Phi}(\alpha),z})}  {\#\Aut(z)}.
\]
\end{Def}

Our goal is to compute the Arakelov degree of $\mathcal{Z}^\mathfrak{a}_{\Phi}(\alpha)$ for $\alpha\gg 0$.
The degree has been normalized in such a way that it is unchanged if the field $K_\Phi$ is enlarged. 
  By the comments at the beginning 
of Section \ref{S:global moduli}, we may therefore make the minimal choice $K_\Phi=\varphi^\mathrm{sp}(K)$.
This will ease comparison with the notation of Section \ref{ss:LHI}.
As in the introduction, let $K^\mathrm{sp}$ be the factor of $K$ on which  $\varphi^\mathrm{sp}:K\to \C$ 
is nonzero.  Let  $F^\mathrm{sp}$ be the maximal totally real subfield of $K^\mathrm{sp}$.
We henceforth use $\varphi^\mathrm{sp}$ to identify 
\[
K^\mathrm{sp} = K_\Phi.
\]
For any prime $\mathfrak{p}$ of $K_\Phi$, we have $\mathfrak{p}_K=\mathfrak{p}$ under this 
identification, and $\mathfrak{p}_F$ is the prime of $F^\mathrm{sp}$ below $\mathfrak{p}$.  
Let  $e_\mathfrak{p}$    be the  ramification degree  of  $K^\mathrm{sp}_{\mathfrak{p}_K}/F^\mathrm{sp}_{\mathfrak{p}_F}$.

\begin{Thm}\label{Thm:local length}
Fix $\alpha\in F$ with $\varphi^\mathrm{sp}(\alpha)\not=0$.
Let $\mathfrak{p}$ be a prime of $K_\Phi$ such that $\mathfrak{p}_F$ is nonsplit in $K$.
The strictly Henselian local ring of $\mathcal{Z}_{\Phi}^\mathfrak{a}(\alpha)$ at any geometric point 
$z\in \mathcal{Z}_{\Phi}^\mathfrak{a}(\alpha)(k^\alg_{\Phi,\mathfrak{p}})$
is Artinian of length
\[
\mathrm{length} ( \co^\mathrm{sh}_{\mathcal{Z}^\mathfrak{a}_{\Phi}(\alpha),z } ) = 
\frac{ 1}{2} \cdot e_\mathfrak{p}
\cdot  \ord_{\mathfrak{p}_F} (\alpha\mathfrak{p}_F\mathfrak{a}^{-1}\mathfrak{d}_F). 
\]
In particular the length does not depend on $z$.
Note that $\varphi^\mathrm{sp}(\alpha)\not=0$ guarantees that $\alpha$ has nonzero projection
to the factor $F^\mathrm{sp}\subset F$, and so $\ord_{\mathfrak{p}_F}(\alpha)<\infty$. 
Thus the right hand side is finite.
\end{Thm}

\begin{proof}
 Let $p$ be the rational 
prime below $\mathfrak{p}$, and recall that $W_{\Phi,\mathfrak{p}}$ 
is the  completed integer ring of the maximal  unramified extension of  $K_{\Phi,\mathfrak{p}}$.
Let $\mathbf{ART}$ be the category of Artinian local $W_{\Phi,\mathfrak{p}}$-algebras with 
residue field $k^\alg_{\Phi,\mathfrak{p}}$.  If 
\[
(A_0,A,f)\in  \mathcal{Z}_{\Phi}^\mathfrak{a}(\alpha)(k^\alg_{\Phi,\mathfrak{p}})
\]
is the triple corresponding to $z$, then the completed 
strictly Henselian local ring $\widehat{\co}^\mathrm{sh}_{\mathcal{Z}^\mathfrak{a}_{\Phi}(\alpha),z }$
pro-represents the functor of deformations of $(A_0,A, f)$ to  objects of $\mathbf{ART}$. 
By the Serre-Tate theorem this is the same as the corresponding deformation functor of $p$-divisible groups
$(A_0[p^\infty], A[p^\infty] , f[p^\infty])$.   

We argue as in the proofs of Lemma \ref{Lem:Hermitian lift} and Theorem \ref{Thm:global hermitian}.
There is a decomposition  $A[p^\infty]\iso \prod_\mathfrak{q}A[\mathfrak{q}^\infty]$ over the primes 
$\mathfrak{q}\subset \co_F$ above $p$, and similarly for any deformation of $(A,\kappa,\lambda)$.  
Fix an isomorphism of $K_\Phi$-algebras $\C_\mathfrak{p}\iso \C$, and let  $\Phi(\mathfrak{q})$ be the 
set of all $\varphi\in \Phi$ whose restriction to $F\to \C_\mathfrak{p}$ induces the prime $\mathfrak{q}$.
The triple $(A[\mathfrak{q}^\infty],\kappa[\mathfrak{q}^\infty], \lambda[\mathfrak{q}^\infty])$
satisfies the $\Phi(\mathfrak{q})$-determinant condition of Section \ref{ss:canonical lifts}.
Set $m=[F_\mathfrak{q}:\Q_p]$.

If $\mathfrak{q}\not=\mathfrak{p}_F$ then $\varphi^\mathrm{sp} \not\in \Phi(\mathfrak{q})$, and 
$\Phi(\mathfrak{q})$ has signature $(m,0)$ in the sense of Section \ref{ss:LHII}.
Theorem \ref{Thm:BT canonical} implies that 
 $(A_0[p^\infty],A[\mathfrak{q}^\infty])$ lifts uniquely to every object of $\mathbf{ART}$, 
 and Proposition \ref{Prop:simple deformation} implies that the homomorphism
  \[
  f[\mathfrak{q}^\infty]  : A_0[p^\infty] \to A[\mathfrak{q}^\infty]
 \]
  lifts uniquely as well.  It follows that the deformation functors of the triples $(A_0,A,f)$ and  
$(A_0[p^\infty], A[\mathfrak{p}_F^\infty] , f [\mathfrak{p}_F^\infty])$ are canonically isomorphic.
As $\varphi^\mathrm{sp} \in \Phi(\mathfrak{p}_F)$, the $p$-adic CM type $\Phi(\mathfrak{p}_F)$ has 
signature $(m-1,1)$ in the sense of Section \ref{ss:LHI}.
By Theorem \ref{Thm:BT canonical}  the deformation functor of the pair 
 $(A_0[p^\infty], A[\mathfrak{p}_F^\infty])$ is pro-represented by $W_{\Phi,\mathfrak{p}}$,
and Theorem \ref{Thm:crystal deform} implies that the deformation functor of 
$(A_0[p^\infty], A[\mathfrak{p}_F^\infty],f[\mathfrak{p}_F^\infty])$ is pro-represented by $W_{\Phi,\mathfrak{p}}/\mathfrak{m}^k$,
where $\mathfrak{m}$ is the maximal ideal of $W_{\Phi,\mathfrak{p}}$ and 
\[
k= \frac{1}{2} \cdot \ord_{\mathfrak{p}_K} (\alpha\mathfrak{p}_F\mathfrak{DD}_0^{-1}\mathfrak{a}^{-1}) .
\]
Therefore the length of $\widehat{\co}^\mathrm{sh}_{\mathcal{Z}^\mathfrak{a}_{\Phi}(\alpha),z }$ is 
\[
k =  \frac{1}{2} \cdot e_\mathfrak{p}\cdot \ord_{\mathfrak{p}_F} (\alpha\mathfrak{p}_F  \mathfrak{s}^{-1}  \mathfrak{a}^{-1})  .
\]
\end{proof}

 Theorem  \ref{Thm:local length} computes the lengths of the  local rings of $\mathcal{Z}^\mathfrak{a}_{\Phi}(\alpha)$,
while Theorem \ref{Thm:point count} counts the number geometric points.  
The calculation of the Arakelov degree is now an easy corollary of these results.
Theorem \ref{Main Theorem I} of the introduction is the  case $\mathfrak{a}=\co_F$ of the 
following theorem.

\begin{Thm}\label{Thm:zero cycle degree}
If  $\alpha\in F^{\gg 0}$ then $\mathcal{Z}^\mathfrak{a}_{\Phi}(\alpha)$ has dimension $0$, and 
\[
\widehat{\deg}\, \mathcal{Z}^\mathfrak{a}_{\Phi}(\alpha)  =   \frac{h(K_0)}{    w(K_0)}
\sum_{\mathfrak{p}}   \frac{ \log(\mathrm{N}(\mathfrak{p}))} { [K^\mathrm{sp} :\Q ] }
 \cdot \ord_{\mathfrak{p}} (\alpha \mathfrak{p}\mathfrak{d}_F \mathfrak{a}^{-1} )
\cdot   \rho ( \alpha\mathfrak{p}^{-\epsilon_p}\mathfrak{d}_F \mathfrak{a}^{-1}  )  
\]
where the sum is over all  primes $\mathfrak{p}$ of $F^\mathrm{sp}$ nonsplit in $K^\mathrm{sp}$, and $p$
is the prime of $\Q$ below $\mathfrak{p}$.
\end{Thm}

\begin{proof}
The first claim is Proposition \ref{Prop:zero cycle}.  If $\mathfrak{p}$ is a prime of $K_\Phi$
for which $\mathfrak{p}_F$ is nonsplit in $K$, then combining Theorem \ref{Thm:point count} and 
Theorem \ref{Thm:local length} shows that
\[
\sum_{z\in \mathcal{Z}^\mathfrak{a}_{\Phi}( \alpha ) (k^\alg_{\Phi,\mathfrak{p}})  }
\frac{\mathrm{length}(\co^\mathrm{sh}_{\mathcal{Z}_{\Phi}(\alpha),z})}  {\#\Aut(z)}  
=
 \frac{ e_\mathfrak{p} h(K_0) }{ 2  w(K_0)} 
\cdot  \ord_{\mathfrak{p}_F} (\alpha \mathfrak{p}_F\mathfrak{d}_F \mathfrak{a}^{-1} )
\cdot   \rho ( \alpha\mathfrak{p}_F^{-\epsilon_p}\mathfrak{d}_F \mathfrak{a}^{-1}  ) .
\]
If $\mathfrak{p}$ is a prime of $K_\Phi$ for which $\mathfrak{p}_F$ is split in $K$
then the left hand side is zero by Proposition \ref{Prop:zero cycle}.
Summing over all  primes $\mathfrak{p}$ of $K_\Phi \iso K^\mathrm{sp}$ yields the result.
\end{proof}


\subsection{Arithmetic divisors on $\mathcal{X}_{\Phi}$}
\label{ss:divisors}


In this subsection we fix an $\alpha\in F^\times$, and restrict to the case $\mathfrak{a}=\co_F$.  
Abbreviate
\[
\mathcal{Z}_{\Phi}(\alpha) = \mathcal{Z}^{\co_F}_{\Phi}(\alpha)
\qquad 
\mathcal{CM}_{\Phi} = \mathcal{CM}^{\co_F}_{\Phi},
\]
and define a regular $1$-dimensional  stack 
\[
\mathcal{X}_\Phi=\mathcal{M}_{(1,0)/\co_\Phi} \times_{\co_\Phi} \mathcal{CM}_{\Phi}
\]
over $\co_\Phi$.
We know from Section \ref{ss:zero stack}
that $\mathcal{Z}_{\Phi}(\alpha)$ is $0$-dimensional, and that the natural map
\[
\mathcal{Z}_{\Phi}(\alpha) \to  \mathcal{X}_{\Phi}
\]
is finite and unramified.  This allows us to view $\mathcal{Z}_{\Phi}(\alpha)$ as a divisor 
on $\mathcal{X}_\Phi$, which we denote by   $\mathtt{Z}_{\Phi}(\alpha)$.  To give the precise definition,
it suffices to describe the pullback of $\mathtt{Z}_{\Phi}(\alpha)$ to an atlas
$ \gamma:  X\map{  } \mathcal{X}_\Phi.$ Let  $Z$ be the cartesian product
\[
\xymatrix{
{ Z } \ar[r]^\phi  \ar[d] &  { X} \ar[d]^\gamma  \\
\mathcal{Z}_{\Phi}(\alpha) \ar[r]   &  \mathcal{X}_\Phi,
}
\]
so that $\phi:Z\to X$ is a finite unramified morphism of \emph{schemes}.  
The divisor $\mathtt{Z}_{\Phi}(\alpha)$ on $\mathcal{X}_\Phi$ is defined as the unique divisor whose pullback to $X$ is
\[
\gamma^* \mathtt{Z}_{\Phi}(\alpha)=  \sum_{z\in Z}  [k(z) : k(\phi(z))] \cdot    \length( \co_{Z,z} ) \cdot \phi(z).
\]
  The reader may consult \cite{gillet84,vistoli} for  the general  theory of divisors and cycles on stacks.
Of course $ \mathtt{Z}_{\Phi}(\alpha) =0$ unless $\alpha \gg 0$.

An \emph{arithmetic divisor} on $\mathcal{X}_\Phi$ is a 
pair $(\mathtt{Z},\green)$ where $\mathtt{Z}$ is a Weil divisor on 
$\ \mathcal{X}_{\Phi}$ with rational coefficients,
and $\green$ is a Green function for $\mathtt{Z}$.  As $\mathtt{Z}$ has no points in characteristic $0$, 
this simply means  that $\green$ is \emph{any} function on the finite set of points
\[
 \bigsqcup_{ \substack{ \sigma:K_\Phi \to \C  \\ \sigma|_{K_0} = \iota }} \mathcal{X}_\Phi^\sigma(\C),
\]
where  $\mathcal{X}_\Phi^\sigma$ is the stack over $\C$ obtained from $\mathcal{X}_\Phi$ by base change.  
To each  rational
function $f$ on $\mathcal{X}_\Phi$ there is an associated \emph{principal arithmetic divisor} $(\mathrm{div}(f), -\log |f|^2)$.
The quotient group of arithmetic divisors modulo principal arithmetic divisors is the 
 \emph{codimension $1$ arithmetic Chow group} $\widehat{\mathrm{CH}}^1 ( \mathcal{X}_\Phi )$
 of Gillet-Soul\'e \cite{BKK, gillet09,gillet-soule90,KRY}.

We will construct a Green function $\green_{\Phi}(\alpha,y,\cdot )$ for the divisor $\mathtt{Z}_{\Phi}(\alpha) $,
depending on an auxiliary parameter $y\in F_\R^{\gg 0}$.   For $t\in \R^{>0}$ define
\begin{equation}\label{beta}
\beta_1(t)=\int_1^\infty e^{-t u}u^{-1} \, du.
\end{equation}
First suppose that $\sigma:K_\Phi\to \C$ is the  inclusion, so that
a  point $z\in \mathcal{X}_\Phi^\sigma(\C)$  corresponds to a pair 
\[
(A_0,A) \in   \mathcal{M}_{(1,0) }(\C) \times \mathcal{CM}_{\Phi}(\C).
\]
To each such pair we attach, exactly as in (\ref{betti space}), an $\co_K$-module
\[
L_B(A_0,A) = \Hom_{\co_{K_0}} (H_1(A_0) , H_1(A))
\]
equipped with a Hermitian form $\langle\cdot, \cdot\rangle_\CM$.  By Proposition \ref{Prop:betti hermitian},
$L_B(A_0,A)$ is a projective $\co_K$-module of rank one, is negative definite at the archimedean place 
$\infty^\mathrm{sp}$ of $F$ determined by $\varphi^\mathrm{sp}:K\to \C$, and is positive definite
at the other archimedean places.  We define
\begin{equation}\label{pre green}
\green_{\Phi}(\alpha,y, z ) = 
\sum_{ \substack{f\in L_B(A_0,A) \\ \langle f,f\rangle_\CM =\alpha } } 
\beta_1( 4\pi |y\alpha|_{\infty^\mathrm{sp} }).
\end{equation}

To complete the definition of $\green_{\Phi}(\alpha,y,\cdot )$ we must generalize this construction to 
an arbitrary  $\Q$-algebra map  $\sigma:K_\Phi\to \C$  whose restriction to $K_0$ is $\iota$.
If we extend $\sigma$ in some way to an automorphism of $\C$, we obtain a new CM type $\Phi^\sigma$,
which does not depend on how $\sigma$ was extended. 
It's not hard to see that  $\mathcal{X}_\Phi^\sigma(\C) \iso \mathcal{X}_{\Phi^\sigma}(\C)$, and so  
points $z\in \mathcal{X}_\Phi^\sigma(\C)$ correspond to a pairs
\[
(A_0,A) \in  \mathcal{M}_{(1,0)}(\C) \times \mathcal{CM}_{K,\Phi^\sigma} (\C).
\]
Define 
\begin{equation}\label{pre green II}
\green_{\Phi}(\alpha,y, z ) = 
\sum_{ \substack{f\in L_B(A_0,A) \\ \langle f,f\rangle_\CM =\alpha } } 
\beta_1( 4\pi |y\alpha|_{\infty^{\mathrm{sp},\sigma} })
\end{equation}
as above, where now $\infty^{\mathrm{sp},\sigma}$ is the archimedean place of $F$ induced by 
the special element  $\sigma\circ \varphi^\mathrm{sp} : K\to \C$ of $\Phi^\sigma$.
As $\langle \cdot , \cdot \rangle_\CM$ is negative definite at $\infty^{\mathrm{sp},\sigma}$, and positive
definite at the remaining archimedean places, the function $\green (\alpha,y,\cdot)$ is identically $0$ 
if  $\alpha\gg 0$.

\begin{Def}
For every $\alpha\in F^\times$ and $y\in F_\R^{\gg 0}$, define an arithmetic divisor
\[
\widehat{\mathtt{Z}}_{\Phi}  (\alpha,y) = \big( \mathtt{Z}_{\Phi}(\alpha)    ,   \green_{\Phi}(\alpha,y,\cdot ) \big)
\in 
\widehat{\mathrm{CH}}^1 ( \mathcal{X}_\Phi ).
\]
Note that if $\alpha \gg 0$ then 
\[
\widehat{\mathtt{Z}}_{\Phi}  (\alpha,y) =  \big( \mathtt{Z}_{\Phi}(\alpha)    ,0 \big),
\]
while if $\alpha\not\gg 0$ then 
\[
\widehat{\mathtt{Z}}_{\Phi}  (\alpha,y) =  \big(   0, \green_{\Phi}(\alpha,y,\cdot )  \big).
\]
\end{Def}

If $\alpha\not\gg 0$ then our definition of $\widehat{\mathtt{Z}}_{\Phi}  (\alpha,y)$ is, at the moment,
rather unmotivated, although the use of the function $\beta_1(t)$ in the definition follows
Kudla \cite{kudla97,KRY}.  The particular choice of Green function will be justified in Section \ref{s:eisenstein},
when we show that, for all $\alpha\in F^\times$, the arithmetic divisor
$\widehat{\mathtt{Z}}_{\Phi}  (\alpha,y)$ is closely related to the Fourier coefficient 
of a Hilbert modular Eisenstein series.

There is a canonical linear functional
\[
\widehat{\deg} : \widehat{\mathrm{CH}}^1 ( \mathcal{X}_\Phi ) \to  \R
\]
defined as the composition 
\[
 \widehat{\mathrm{CH}}^1 ( \mathcal{X}_\Phi ) \to   \widehat{\mathrm{CH}}^1(\Spec(\co_\Phi)) \to  \R
\]
where the first arrow is push-forward by the structure map $\mathcal{X}_\Phi\to  \Spec(\co_\Phi)$
and the second is $[K_\Phi:\Q]^{-1}$ times the degree of \cite[Section 3.4.3]{gillet-soule90}.
If $\mathtt{Z}$ is a prime Weil divisor on $\mathcal{X}_\Phi$ then
\[
\widehat{\deg} \, (\mathtt{Z},0) =  \frac{ 1  }{[K_\Phi:\Q]}
 \sum_{ \mathfrak{p} \subset\co_\Phi}     \log(\mathrm{N}(\mathfrak{p}))  
\sum_{z\in \mathtt{Z} (k^\alg_{\Phi,\mathfrak{p}})  }
\frac{1}  {\#\Aut(z)}.
\]
If $\green$ is a Green function on $\mathcal{X}_\Phi$ then
\begin{equation}\label{arch degree def}
\widehat{\deg} \, (0 , \green) =  \frac{  1  }{ [K_\Phi:\Q]}
 \sum_{ \substack{ \sigma:K_\Phi \to \C  \\ \sigma|_{K_0} = \iota }}
\sum_{z\in \mathcal{X}_\Phi^\sigma(\C) }
\frac{\green(z)}  {\#\Aut(z)}.
\end{equation}

\begin{Thm}\label{Thm:degree formulas}
Suppose the discriminants of $K_0$ and $K$ are odd and relative prime.   Fix 
$\alpha\in F^\times$ and $y\in F_\R^{\gg 0}$.

\begin{enumerate}
\item
If $\alpha \gg 0$ then
\[
\widehat{\deg}\, \widehat{\mathtt{Z}}_{\Phi}(\alpha,y)  =   \frac{h(K_0)}{    w(K_0)}
\sum_{\mathfrak{p}}   \frac{ \log(\mathrm{N}(\mathfrak{p}))} { [K^\mathrm{sp} :\Q ] }
 \cdot \ord_{\mathfrak{p}} (\alpha \mathfrak{p}\mathfrak{d}_F)
\cdot   \rho ( \alpha\mathfrak{p}^{-\epsilon_p}\mathfrak{d}_F   )  
\]
where the sum is over all  primes $\mathfrak{p}$ of $F^\mathrm{sp}$ nonsplit in $K^\mathrm{sp}$, and $p$
is the prime of $\Q$ below $\mathfrak{p}$.

\item
Suppose  $\alpha\not\gg 0$.  If $\alpha$ is negative at exactly one archimedean place $v$ of $F$,
and if the corresponding map $F\to \R$ factors through the summand $F^\mathrm{sp}$ of $F$,  then
\[
\widehat{\deg}\, \widehat{\mathtt{Z}}_{\Phi}(\alpha,y)  =   \frac{h(K_0)}{    w(K_0)}
 \frac{1}{[K^\mathrm{sp} : \Q]}    \cdot 
 \beta_1( 4\pi  |y\alpha|_v ) \cdot \rho(\alpha\mathfrak{d}_F).
\]
If no such $v$ exists then the left hand side is $0$.
\end{enumerate}
\end{Thm}

\begin{proof}
If $\alpha\gg 0$ then 
\begin{equation}\label{degree compare}
\widehat{\deg}\, \widehat{\mathtt{Z}}_{\Phi}  (\alpha,y)  = 
\widehat{\deg}\, \mathcal{Z}_{\Phi}(\alpha),
\end{equation}
where the right hand side is the Arakelov degree of Definition \ref{Def:arakelov}.  Hence the 
first claim is just a restatement of Theorem \ref{Thm:zero cycle degree}, in the special case $\mathfrak{a}=\co_F$.

Now suppose $\alpha\not\gg 0$, and fix a $\sigma:K_\Phi\map{}\C$. 
If $\alpha$ is negative at $\infty^{\mathrm{sp},\sigma}$ 
and positive at all other archimedean places of $F$, then repeating the proof of Theorem \ref{Thm:point count}  shows that
\[
\sum_{ \substack{   A_0\in \mathcal{M}_{(1,0)}(\C) \\ A\in \mathcal{CM}_{K,\Phi^\sigma}(\C)   } }
\sum_{ \substack{  f\in L_B(A_0,A) \\ \langle f,f\rangle_\CM =\alpha   }} 
\frac{1 }{\#\Aut(A_0,A) }
 =  \frac{h(K_0)}{w(K_0)}  \cdot  \rho(\alpha\mathfrak{s}^{-1}).
\]
Indeed, the only difference is in the calculation of the orbital integral (Lemma \ref{Lem:orbital}),
where one replaces the $\beta$ of Theorem \ref{Thm:global hermitian} with the $\beta$ of 
Proposition \ref{Prop:betti hermitian}.    The inner sum on the left is empty if $\alpha$ is positive at $\infty^{\mathrm{sp},\sigma}$
or negative at some other archimedean place.

 As $\sigma:K_\Phi\to \C$ varies over all embeddings 
whose restriction to $K_0$ is $\iota$, $\infty^{\mathrm{sp},\sigma}$ varies over all archimedean places of 
$F^\mathrm{sp}$, each counted with multiplicity $[K_\Phi:\Q]/ [K^\mathrm{sp}:\Q]$.
If  $\alpha$ is negative at exactly one archimedean place $v$ of $F$,
and if this place $v$ lies on $F^\mathrm{sp}$, then we compute
\begin{eqnarray*} \lefteqn{
\widehat{\deg}\, \widehat{\mathtt{Z}}_{\Phi}(\alpha,y)  } \\
& = &
\frac{  1  }{ [K_\Phi:\Q]}
 \sum_{ \substack{ \sigma:K_\Phi \to \C  \\ \sigma|_{K_0} = \iota }}
\sum_{z\in \mathcal{X}_\Phi^\sigma(\C) }
\frac{\green_{\Phi}(\alpha,y,z)}  {\#\Aut(z)} \\
& = &
\frac{  1  }{ [K_\Phi:\Q]}
 \sum_{ \substack{ \sigma:K_\Phi \to \C  \\ \sigma|_{K_0} = \iota }}
 \sum_{ \substack{   A_0\in \mathcal{M}_{(1,0)}(\C) \\ A\in \mathcal{CM}_{K,\Phi^\sigma}(\C)   } }
\sum_{ \substack{f\in L_B(A_0,A) \\ \langle f,f\rangle_\CM =\alpha } } 
\frac{ \beta_1( 4\pi |y\alpha|_{\infty^{\mathrm{sp},\sigma} }) } { \#\Aut(A_0,A) } \\
 & = &
 \frac{1}{[K^\mathrm{sp} : \Q]}   \frac{h(K_0)}{w(K_0)}  \cdot 
 \beta_1( 4\pi  |y\alpha|_v ) \cdot \rho(\alpha\mathfrak{s}^{-1}).
\end{eqnarray*}
If no such $v$ exists then the inner sum on the third line is empty.
To complete the proof, recall from Proposition \ref{Prop:s ideal} that, under our
hypotheses on the discriminants of $K_0$ and $K$, $\mathfrak{s}=\mathfrak{d}_F^{-1}$. 
\end{proof}


\subsection{Arithmetic divisors on $\mathcal{M}$}
\label{ss:divisors II}


In this subsection we study arithmetic intersection theory on the $\co_{K_0}$-stack 
\[
\mathcal{M} = \mathcal{M}_{(1,0)} \times_{\co_{K_0}} \mathcal{M}_{(n-1,1)}
\]
of the introduction.   Recall that $\mathcal{M}$ is smooth of relative dimension $n-1$ over 
$\co_{K_0}[\mathrm{disc}(K_0)^{-1}]$.  
If $S$ is a connected $\co_{K_0}$-scheme, then to every point $(A_0,A)\in \mathcal{M}(S)$
we have attached an $\co_{K_0}$-module
\[
L(A_0,A)= \Hom_{\co_{K_0}}(A_0,A),
\]
and an $\co_{K_0}$-Hermitian form $\langle\cdot,\cdot\rangle$ defined by (\ref{hermitian def}).

\begin{Def}
For any nonzero $m\in \Z$ let $\mathcal{Z}(m)$ be algebraic stack over $\co_{K_0}$ whose 
functor of points assigns to any connected $\co_{K_0}$-scheme $S$ the groupoid of 
triples $(A_0,A,f)$, where $(A_0,A)\in \mathcal{M}(S)$,  and $f\in L(A_0,A)$ satisfies
$\langle f,f\rangle =m$. 
\end{Def}

 We call the stacks $\mathcal{Z}(m)$ the  \emph{Kudla-Rapoport divisors}.
By \cite[Proposition 2.10]{KRunitaryII} the natural map $\mathcal{Z}(m)\to \mathcal{M}$ is finite and unramified.
As in Section \ref{ss:divisors},  we abbreviate $\mathcal{X}_\Phi$ for the $1$-dimensional stack 
\[
\mathcal{X}_\Phi  = \mathcal{M}_{(1,0)/\co_{\Phi}} \times_{\co_\Phi} \mathcal{CM}_\Phi.
\]
The map $\mathcal{CM}_\Phi \map{} \mathcal{M}_{(n-1,1)/\co_\Phi}$ defined by restricting 
the action of $\co_K$ to $\co_{K_0}$ induces a map
\[
\mathcal{X}_\Phi  \to \mathcal{M}_{/\co_\Phi}.
\]
A point in the intersection of $\mathcal{X}_\Phi$ and $\mathcal{Z}(m)$, defined over some $\co_\Phi$-scheme $S$, 
is a triple $(A_0,A,f)$ in which $(A_0,A,f)\in \mathcal{Z}(m)$, and $A$ is endowed with complex multiplication by 
$\co_K$.  We know from Section \ref{ss:hermitian spaces} that the induced $\co_K$-action on $L(A_0,A)$
then endows $L(A_0,A)$ with additional structure: a  totally positive definite
$\co_K$-Hermitian form $\langle\cdot,\cdot\rangle_\CM$
whose trace is the original Hermitian form $\langle\cdot ,\cdot \rangle$.  Thus $\langle f,f\rangle_\CM$ must
satisfy
\[
m= \mathrm{Tr}_{F/\Q} \langle f,f\rangle_\CM.
\]
In this way we see that the stack theoretic intersection 
\[
\mathcal{X}_\Phi \cap \mathcal{Z}(m) = \mathcal{X}_\Phi \times_{\mathcal{M}_{/\co_\Phi}} \mathcal{Z}(m)_{/\co_\Phi}
\]
admits a decomposition
\begin{equation}\label{intersection decomp}
\mathcal{X}_\Phi \cap \mathcal{Z}(m) = 
\bigsqcup_{ \substack{  \alpha\in F \\ \mathrm{Tr}_{F/\Q}(\alpha) =m }} \mathcal{Z}_\Phi(\alpha)
\end{equation}
where $\mathcal{Z}_\Phi(\alpha) = \mathcal{Z}_\Phi^{\co_F}(\alpha)$ is the stack of Section \ref{ss:zero stack}.

\begin{Def}\label{Def:finite intersection}
Define the \emph{intersection multiplicity} 
\[
I ( \mathcal{X}_{\Phi} : \mathcal{Z}(m) )
= 
\sum_{ \mathfrak{p}  \subset\co_\Phi} 
\frac{ \log( \mathrm{N}(\mathfrak{p}))  }{[K_\Phi : \Q]}
\sum_{ z \in  (\mathcal{X}_{\Phi} \cap \mathcal{Z}(m))(k^\alg_{\Phi,\mathfrak{p}} ) }
 \frac{ \mathrm{length}  \big( \co^\mathrm{sh}_{\mathcal{X}_{\Phi} \cap \mathcal{Z}(m),z} \big) } { \#\Aut(z) }.
\]
This is finite if $\mathcal{X}_\Phi\cap \mathcal{Z}(m)$ has dimension $0$.
\end{Def}

\begin{Rem}
From the point of view of arithmetic intersection theory,  Definition \ref{Def:finite intersection} is 
a bit  naive.  The more natural definition is the \emph{Serre intersection multiplicity}
of \cite[Chapter I.2]{soule92} or \cite[Chapter V.3]{serre00}, which takes into account higher $\mathrm{Tor}$
terms of the structure sheaves $\co_{\mathcal{X}_\Phi}$ and $\co_{\mathcal{Z}(m)}$.  
We have not done this, as the stack $\mathcal{M}$
in which the intersection is taking place is neither flat nor regular, and so is itself a rather naive place to be 
doing arithmetic intersection theory. See the comments of Section \ref{ss:speculation}.
 For the reader's benefit, we only point out that \cite[p.~111]{serre00}
shows that under modest hypotheses these higher $\mathrm{Tor}$ terms vanish, and Serre's intersection multiplicity
agrees with the naive intersection multiplicity.
\end{Rem}

\begin{Thm}\label{Thm:fundamental decomp I}
Let $m$ be any nonzero integer. If  $F$ is a field  then (\ref{intersection decomp}) has dimension $0$,
and
\[
I(\mathcal{X}_\Phi : \mathcal{Z}(m))    = 
\sum _{ \substack{ \alpha\in F^\times, \alpha \gg 0 \\ \mathrm{Tr}_{F/\Q} (\alpha) = m} } 
\widehat{\deg}\, \widehat{\mathtt{Z}}_{\Phi}(\alpha,y) 
\]
for any $y\in \R^{>0}$.
 \end{Thm}

\begin{proof}
 The assumption that $F$ is a field implies that every $\alpha\in F$ with
$\mathrm{Tr}_{F/\Q}(\alpha)=m$ must satisfy $\alpha\in F^\times$.  
By Proposition \ref{Prop:zero cycle} the right hand side of (\ref{intersection decomp})  has dimension 
zero, and the only nonempty contribution comes from totally positive $\alpha$.
Therefore
 (\ref{intersection decomp}) implies 
 \[
I(\mathcal{X}_\Phi : \mathcal{Z}(m) ) =
\sum _{ \substack{ \alpha\in F^\times \\ \mathrm{Tr}_{F/\Q} (\alpha) = m} } \widehat{\deg}\, \mathcal{Z}_{\Phi}(\alpha)
\]
where $\widehat{\deg}$ is the Arakelov degree of Definition \ref{Def:arakelov},
and the claim  follows from (\ref{degree compare}).
\end{proof}

In order to construct a Green function for the divisors $\mathcal{Z}(m)$, we first 
describe the complex uniformizations of the algebraic stacks $\mathcal{M}$ and $\mathcal{Z}(m)$, following  \cite{KRunitaryII}.
Recall that $K_0$ comes with a fixed embedding $\iota:K_0\to \C$.   Let  $\delta\in K_0$ 
be the unique  square root of $\mathrm{disc}(K_0)$ for which $\delta=i\cdot |\delta|$.
 If $W$ is any $K_0$-vector space then $W_\R=W\otimes_{\Q} \R$ is a $\C$-vector space.

 \begin{Def}
A \emph{principal Hermitian lattice} of signature $(r,s)$ is a projective 
$\co_{K_0}$-module $\mathfrak{A}$ of rank $r+s$
together with a Hermitian form $H$ of signature $(r,s)$  under which $\mathfrak{A}$ is self-dual.  
\end{Def}

Let $\mathfrak{A}$ and $\mathfrak{A}_0$ be a principal Hermitian lattices of signature $(n-1,1)$ and $(1,0)$, respectively,
with Hermitian forms $H$ and $H_0$.    Define a $\Q$-symplectic form $\lambda$ on $\mathfrak{A}_\Q$ by
\begin{equation}\label{unitary-symplectic}
\delta\cdot \lambda(v,w) = H(v,w) - H(w,v),
\end{equation}
and a $\Q$-symplectic form $\lambda_0$ on $\mathfrak{A}_{0\Q}$ by the same formula, with $H$ replaced by $H_0$.
 The $\co_{K_0}$-module 
\[
L_B(\mathfrak{A}_0,\mathfrak{A})=\Hom_{\co_{K_0}}( \mathfrak{A}_0,\mathfrak{A})
\]
carries a natural Hermitian form $\langle f_1, f_2\rangle= f_2^*\circ f_1$, where for any 
$f\in L_B(\mathfrak{A}_0,\mathfrak{A})$ 
we define $f^* :\mathfrak{A} \to  \mathfrak{A}_0$  by the relation $H(fv,w)  =  H_0(v,f^*w).$
The Hermitian forms $H_0$, $H$, and $\langle \cdot, \cdot \rangle$ are related by
\[
H(f_1v_1, f_2v_2 ) = H_0(v_1,v_2) \cdot \langle f_1,f_2\rangle.
\]

Abbreviate 
\[
V=L_B(\mathfrak{A}_0,\mathfrak{A})\otimes_\Z\Q,
\] 
and  let $\mathcal{D}$ be the set of  negative $\C$-lines in the Hermitian space $V_\R$.
Given a nonzero  isotropic vector  $e\in V_\R$, there is an isotropic 
$e' \in V_\R$ such that $\langle e,e'\rangle=\delta$.  The restriction of 
$\langle\cdot,\cdot\rangle$ to the  orthogonal complement of the $\C$-span of $\{ e,e'\}$ is  positive definite, and so we may 
extend $\{e,e'\}$ to a $\C$-basis  $e,e_1\ldots, e_{n-2}, e'\in V_\R$  in such a way that the Hermitian form 
 is  given by
\[
\langle x,y \rangle = {}^t x \cdot \left(\begin{smallmatrix} 
& &  \delta \\
& A & \\
-\delta 
\end{smallmatrix}\right)  \cdot  \overline{y}
\]
for a diagonal matrix $A\in M_{n-2}(\R)$ with positive diagonal entries.  There is a bijection
\begin{equation}\label{cusp coordinates}
\mathcal{D} \iso \{  (w,u)\in \C\times \C^{n-2} : \mathrm{Tr}(\delta w)+  {}^t u A \overline{u} < 0    \}
\end{equation}
defined by associating $(w,u)$ to the negative $\C$-line spanned by 
\[
\left[ \begin{matrix} w \\ u \\ 1 \end{matrix} \right] \in \C^n \iso V_\R.
\]
We say that the basis $e,e_1,\ldots, e_{n-2}, e'$ and the coordinates $(w,u)$ are
\emph{adapted} to the isotropic vector  $e$, which should be thought of as the limit  as $w\to i\cdot \infty$.  
The coordinates $(w,u)$ make $\mathcal{D}$ into a complex manifold.
If $n=1$ then $V_\R$ has signature $(0,1)$, and so has no nonzero isotropic vector.  In this degenerate case
$\mathcal{D}$ consists of a single point.

Any choice of nonzero vector in $\mathfrak{A}_{0\Q}$ determines an isomorphism (evaluation at the chosen vector)
of $K$-vector spaces $V\to \mathfrak{A}_\Q$,  which identifies $H$ with a 
positive rational multiple of $\langle\cdot,\cdot\rangle$, and  identifies $\mathcal{D}$
with the space of negative lines in $\mathfrak{A}_\R$.  This identification does not depend on the choice of vector
used in its definition.  Any $\mathtt{h}\in \mathcal{D}$, viewed as a negative line in the complex vector space
$\mathfrak{A}_\R$,  determines an endomorphism $J_\mathtt{h}$ of  $\mathfrak{A}_\R$ by 
\[
J_\mathtt{h} v = \begin{cases}
-iv & \mbox{if }v\in \mathtt{h} \\
iv& \mbox{if }v\in \mathtt{h}^\perp,
\end{cases}
\]
where $\mathtt{h}^\perp$ is the orthogonal complement of $\mathtt{h}$ with respect to $H$.
Of course 
\[
J_\mathtt{h} \circ J_\mathtt{h}=-1,
\] 
and it is easy to see that the quadratic form $\lambda(J_\mathtt{h}v,v)$ 
on $\mathfrak{A}_\R$  is  positive definite.   A little linear algebra shows that every $\R$-linear endomorphism
of $\mathfrak{A}_\R$ satisfying these two properties is of the form $J_\mathtt{h}$ for a unique $\mathtt{h}\in\mathcal{D}$.

We now describe the complex uniformization of $\mathcal{M}(\C)$, following \cite{KRunitaryII}.
The complex elliptic curve $A_0(\C)=\mathfrak{A}_{0\R}/\mathfrak{A}_0$, 
with its principal polarization determined by $\lambda_0$, and its natural $\co_{K_0}$-action,
determines a point of $\mathcal{M}_{(1,0)}(\C)$.  To each $\mathtt{h}\in  \mathcal{D}$ 
there is an associated  $(A_\mathtt{h},\kappa_\mathtt{h},\lambda_\mathtt{h})\in \mathcal{M}_{(n-1,1)}(\C)$
in which  
\begin{itemize}
\item
$A_\mathtt{h}(\C) = \mathfrak{A}_\R/\mathfrak{A}$ with the 
complex structure determined by $J_\mathtt{h}$, 
\item
$\kappa_\mathtt{h}:\co_{K_0}\to \End(A_\mathtt{h})$  is  induced by 
the $\co_{K_0}$-module structure on $\mathfrak{A}$,  
\item
$\lambda_\mathtt{h}:A_\mathtt{h} \to A_\mathtt{h}^\vee$ is the polarization
induced by the symplectic form $\lambda$.
\end{itemize}
The rule $\mathtt{h}\mapsto (A_0,A_\mathtt{h})$ defines a morphism of complex orbifolds 
\[
\mathcal{D}\to  \mathcal{M}(\C)
\] 
Let $\Gamma_\mathfrak{A}$ be the automorphism group of $(\mathfrak{A},H)$, and 
let $\Gamma_{\mathfrak{A}_0}$ be the automorphism group of $(\mathfrak{A}_0,H_0)$ 
(so that  $\Gamma_{\mathfrak{A}_0}$ is just the 
group of roots of unity in $K_0$).  The group $\Gamma=\Gamma_{\mathfrak{A}_0}\times\Gamma_\mathfrak{A}$
acts on $L_B(\mathfrak{A}_0,\mathfrak{A})$ through automorphisms preserving the Hermitian form $\langle\cdot,\cdot\rangle$,
and so acts on the space $\mathcal{D}$.  The pair $(A_0,A_\mathtt{h})$ depends only on the $\Gamma$-orbit of $\mathtt{h}$,
and we obtain a morphism of complex orbifolds 
\[
[\Gamma\backslash \mathcal{D}] \to \mathcal{M}(\C)
\] 
identifying $[\Gamma\backslash \mathcal{D}]$ with a connected component of 
$\mathcal{M}(\C)$.  The other connected components 
are obtained by repeating this construction for each of the finitely many 
isomorphism classes of pairs $(\mathfrak{A}_0,\mathfrak{A})$.

Given a nonisotropic  $f\in L_B(\mathfrak{A}_0,\mathfrak{A})$, define
\[
\mathcal{D}(f) = \{ \mathtt{h} \in \mathcal{D} : f\perp \mathtt{h} \}.
\]
Following \cite{kudla97} or \cite{bruinier99} there is a standard way to construct a smooth  function on 
$\mathcal{D} \smallsetminus \mathcal{D}(f)$  with a logarithmic  singularity along  $\mathcal{D}(f)$.   
Let $f_\mathtt{h}$ be the orthogonal projection of $f$ to  $\mathtt{h}$, and set 
\[
R (f,\mathtt{h})  = - \langle f_\mathtt{h}, f_\mathtt{h} \rangle,
\]
 a nonnegative real analytic function on $\mathcal{D}$ whose zero set is $\mathcal{D}(f)$.  If we write 
\[
f = ae + b_1e_1+\cdots + b_{n-2} e_{n-2} + ce'
\]
in terms of a basis adapted to an isotropic $e\in \mathfrak{A}_\R$, then this function is given by the 
explicit formula
\begin{equation}\label{majorant}
R (f,\mathtt{h})  =
\frac{ |\delta (\overline{c} w - \overline{a})  + {}^t\overline{b}  A u|^2  }
{ |\delta (w-  \overline{w}) +{}^t uA\overline{u} |}
\end{equation}
in the coordinates (\ref{cusp coordinates}), where ${}^t b = [b_1\cdots b_{n-2}]$.  
This calculation shows that $\mathcal{D}(f)$ is a complex analytic divisor on $\mathcal{D}$, defined by the 
equation 
\[
\delta (\overline{c} w - \overline{a})  + {}^t\overline{b}  A u=0.
\]
If $\langle f,f\rangle<0$ then
$\mathcal{D}(f)=\emptyset$, and $R(f,\mathtt{h})$ is a positive function on $\mathcal{D}$.
In the degenerate case $n=1$ the set $\mathcal{D}(f)$ is empty, and 
 $R(f,\mathtt{h})=-\langle f,f\rangle$.  For any $\mathtt{h}\in\mathcal{D}$, each  $f\in L_B(\mathfrak{A}_0,\mathfrak{A})$ 
induces a homomorphism of \emph{real} Lie groups 
\[
f:A_0(\C)\to A_\mathtt{h}(\C),
\]
and linear algebra shows that this map  is complex analytic if and only if $\mathtt{h}\in \mathcal{D}(f)$.  
In this way we obtain a morphism of orbifolds
\[
 \Big[\Gamma\backslash
 \bigsqcup_{ \substack{  f\in L_B(\mathfrak{A}_0,\mathfrak{A}) \\ \langle f,f\rangle =m   }} 
 \mathcal{D}(f)  \Big] \to  \mathcal{Z}(m)(\C)
\]
defined by sending $\mathtt{h}\in\mathcal{D}(f)$ to the triple $(A_0,A_\mathtt{h}, f)$.  The image is 
an open and closed suborbifold of $\mathcal{Z}(m)(\C)$, and taking the disjoint union over all isomorphism
classes of pairs $(\mathfrak{A}_0,\mathfrak{A})$ gives a complex uniformization of $\mathcal{Z}(m)(\C)$.

The function $\beta_1(x)$ of (\ref{beta}) has a logarithmic singularity at $x=0$, in the sense that
$\beta_1(x) + \log (x)$  can be extended smoothly to  $\R$.  Furthermore, $\beta_1(x)$ decays 
exponentially as $x\to\infty$. Given  a  positive parameter $y\in \R$, define a smooth function 
\[
\green(f, y, \mathtt{h}) = \beta_1\big( 4\pi y R(f,\mathtt{h}) \big)
\]
on $\mathcal{D} \smallsetminus \mathcal{D}(f)$. If  $g(\mathtt{h}) =0$ is any equation for the divisor $\mathcal{D}(f)$ on some open
subset $U$ of  $\mathcal{D}$, then (\ref{majorant}) shows that
$ \green(f, y, \mathtt{h} )  +\log |g(\mathtt{h})|^2$ extends smoothly to all of $U$. For nonzero $m\in \Z$  the sum
\[
\green(m,y,\mathtt{h}) 
=  \sum_{ \substack{  f\in L_B(\mathfrak{A}_0 , \mathfrak{A}) \\ \langle f,f\rangle =m   }}  \green(f,y,\mathtt{h})
\]
 defines a Green function, in the sense of \cite{gillet-soule90, soule92}, for the orbifold divisor 
\[
 \Big[\Gamma\backslash
 \bigsqcup_{ \substack{  f\in L_B(\mathfrak{A}_0,\mathfrak{A}) \\ \langle f,f\rangle =m   }} 
 \mathcal{D}(f)  \Big] \to   [ \Gamma \backslash \mathcal{D}  ].
 \]
  Using the complex uniformizations  of 
$\mathcal{Z}(m)(\C)$ and $\mathcal{M}(\C)$ described above, 
the function $\green(m,y,\cdot)$, constructed 
now for every isomorphism class of pairs $(\mathfrak{A}_0,\mathfrak{A})$, 
defines a Green function for  the divisor  $\mathcal{Z}(m)$ on $\mathcal{M}$.  
If $m<0$ then $\green(m,y,\cdot )$ is a smooth function on $\mathcal{M}(\C)$.

Using the forgetful map $\mathcal{X}_\Phi \to  \mathcal{M}_{/\co_\Phi}$, it makes sense to evaluate 
$\green(m,y,\cdot)$ on the finite set of points of the complex fiber of $\mathcal{X}_\Phi$.  More
precisely, we define:
\[
\green(m,y,  \mathcal{X}_\Phi) = \frac{1}{[K_\Phi:\Q]} 
\sum_{ \substack{\sigma : K_\Phi\to \C \\ \sigma|_{K_0} = \iota } }
\sum_{z\in \mathcal{X}_\Phi^\sigma(\C) } \frac{ \green(m,y, z) }{\#\Aut(z)} .
\]
Here $\mathcal{X}_\Phi^\sigma$ is the $\C$-scheme obtained from $\mathcal{X}_\Phi$ by base change
through $\sigma$.  There is a slight abuse of notation on the right hand side, as we are confusing 
$z\in \mathcal{X}_\Phi^\sigma(\C)$ with its image in 
$(\mathcal{M}_{/\co_\Phi})^\sigma(\C)  \iso \mathcal{M} ( \C )$.
The right hand side is only defined if the images of $\mathcal{X}_\Phi^\sigma$ and $\mathcal{Z}(m)$
have no common points in the complex fiber of $\mathcal{M}$.
This is equivalent to (\ref{intersection decomp}) being $0$-dimensional.

\begin{Thm}\label{Thm:fundamental decomp II}
Let $m$ be any nonzero integer. If  $F$ is a field  then (\ref{intersection decomp}) has dimension $0$,
and
\[
 \green(m,y,  \mathcal{X}_\Phi)  = 
\sum _{ \substack{ \alpha\in F^\times, \alpha \not \gg 0 \\ \mathrm{Tr}_{F/\Q} (\alpha) = m} } 
\widehat{\deg}\, \widehat{\mathtt{Z}}_{\Phi}(\alpha,y) 
\]
for any $y\in \R^{>0}$.
\end{Thm}

\begin{proof}
We already saw in  Theorem \ref{Thm:fundamental decomp I}  that (\ref{intersection decomp}) has dimension $0$.
Suppose we have a point $z \in \mathcal{X}_{\Phi}(\C)$ representing a pair $(A_0,A)$.  Set
$\mathfrak{A}_0 = H_1(A_0(\C),\Z)$ and $\mathfrak{A} = H_1(A(\C) , \Z)$, 
viewed as principal Hermitian lattices using the Hermitian 
forms $H_0$ and $H$ determined, using (\ref{unitary-symplectic}), by the polarizations $\lambda_0$ and $\lambda$.
The canonical isomorphism $\mathfrak{A}_\R \iso \Lie(A)$
determines a complex structure on $\mathfrak{A}_\R$, and under this complex structure multiplication by $i$ 
has the form  $J_\mathtt{h}$ for a unique $\mathtt{h}\in\mathcal{D}$.
If  $\epsilon^\mathrm{sp}\in F_\R$ is the idempotent corresponding to the 
place $\infty^\mathrm{sp}$ determined by the restriction of $\varphi^\mathrm{sp}:K\to \C$ to $F$,
 then the negative  line $\mathtt{h}$ is none other than $\mathtt{h}= \epsilon^\mathrm{sp} \cdot \mathfrak{A}_\R.$

It follows easily that for any $f\in L_B( A_0, A)$ we have
\[
R(f,\mathtt{h} ) = -\langle f_\mathtt{h}, f_\mathtt{h}\rangle 
=  -\langle \epsilon^\mathrm{sp} f ,  \epsilon^\mathrm{sp}f   \rangle 
=  | \langle f,f\rangle_\CM |_{\infty^\mathrm{sp}}.
\]
Therefore $\green ( f,y, \mathtt{h} ) 
= \beta_1\big(   4\pi y |   \langle f,f\rangle_\CM  |_{\infty^\mathrm{sp}}\big)$
and 
\begin{eqnarray*} 
\sum_{z\in \mathcal{X}_{\Phi}(\C)} \frac{  \green(m,y, z)}{\#\Aut(z)}   
& = &
\sum_{(A_0,A) \in   \mathcal{X}_{\Phi}(\C)} 
\sum_{ \substack{  f\in L_B(A_0,A) \\ \langle f,f\rangle =m   }}  
\frac{\beta_1\big(   4\pi y |   \langle f,f\rangle_\CM  |_{\infty^\mathrm{sp}}\big) }{\#\Aut(A_0,A) }  \\
& = &
\sum_{ \substack{ \alpha\in F \\ \mathrm{Tr}_{F/\Q}(\alpha) =m  }  }
\sum_{(A_0,A) \in   \mathcal{X}_\Phi(\C)} 
\sum_{ \substack{  f\in L_B(A_0,A) \\ \langle f,f\rangle_\CM =\alpha   }}  
\frac{  \beta_1\big( 4\pi y |\alpha|_{\infty^\mathrm{sp}} \big)}{\#\Aut(A_0,A) } \\
& = &
\sum_{ \substack{ \alpha\in F \\ \mathrm{Tr}_{F/\Q}(\alpha) =m  }  }
\sum_{ z \in   \mathcal{X}_\Phi(\C)}  \frac{  \green_\Phi(\alpha,y,z )}{\#\Aut(z) }
\end{eqnarray*}
where the final equality is by the definition (\ref{pre green}) of $ \green_\Phi(\alpha,y,z )$
at a point $z\in \mathcal{X}_\Phi(\C)$.  As $F$ is a field and $m\not=0$, we may restrict to $\alpha\in F^\times$
in the final sum.  As the Hermitian form $\langle\cdot,\cdot\rangle_\CM$ is negative definite at $\infty^\mathrm{sp}$
by Proposition \ref{Prop:betti hermitian}, we may further restrict to $\alpha\not\gg 0$.

Now suppose $z\in  \mathcal{X}_\Phi^\sigma(\C)$.  As in the discussion preceding (\ref{pre green II}),
we may identify $\mathcal{X}_\Phi^\sigma(\C)  \iso \mathcal{X}_{\Phi^\sigma}(\C)$.  
Repeating the argument above with $\Phi$ replaced by $\Phi^\sigma$ and 
$\infty^\mathrm{sp}$ replaced by $\infty^{\mathrm{sp},\sigma}$,
and using (\ref{pre green II}) instead of (\ref{pre green}), shows that
\[
 \green(m,y,  \mathcal{X}_\Phi)  = \frac{1}{[K_\Phi:\Q]}
 \sum_{ \substack{ \alpha\in F^\times, \alpha\not\gg 0 \\ \mathrm{Tr}_{F/\Q}(\alpha) =m  }  }
 \sum_{ \substack{\sigma : K_\Phi\to \C \\ \sigma|_{K_0} = \iota } }
 \sum_{ z \in   \mathcal{X}_\Phi^\sigma(\C)}  \frac{  \green_\Phi(\alpha,y,z )}{\#\Aut(z) }.
\]
Comparing with (\ref{arch degree def}) completes the proof.
\end{proof}


\section{Eisenstein series}
\label{s:eisenstein}


Keep $K_0$, $F$, $K$,  $\Phi$, and $K_\Phi$ as in Section \ref{S:global moduli}.
In this section we construct a Hilbert modular Eisenstein series $\mathcal{E}_\Phi(\tau,s)$ 
on $\mathcal{H}_F$. This Eisenstein series is \emph{incoherent} in the sense of  
Kudla \cite{kudla97}, and so vanishes at $s=0$.
We use formulas of Yang \cite{yang05} to compute the Fourier coefficients of the derivative at $s=0$, 
and show that  these coefficients agree with the arithmetic degrees appearing in 
Theorem \ref{Thm:degree formulas}.


\subsection{A Hilbert modular Eisenstein series}


In this subsection we attach to every $\mathbf{c}\in F_\A^\times$ a Hilbert modular Eisenstein series
$\mathcal{E}(\tau, s;\mathbf{c},\psi_F)$ of the type considered in \cite{yang05}. 

First we quickly recall some of the local theory of \cite{kudla97,kudla-rallis,kudla99b,yang05}.
Fix a place $v$ of $F$ and a $\mathbf{c} \in F_v^\times$, and let $\chi_v$ be the character of $F_v^\times$ associated
to the quadratic extension $K_v/F_v$.    Let $\psi$ be an additive character $F_v\to \C^\times$.
Associated to the $F_v$-quadratic space $(K_v,\mathbf{c} x\overline{x})$ and the character 
$\psi$ is a Weil representation
$\omega_{\mathbf{c},\psi}$ of $\mathrm{SL}_2(F_v)$ on the space of Schwartz functions 
$\mathfrak{S}(K_v)$ on $K_v$; see \cite[Chapter II.4]{kudla96}.
For $s\in \C$ let $I(\chi_v,s)$ be the space of the induced representation of the character $\chi_v(x)\cdot |x|_v^s$.
There is an $\mathrm{SL}_2(F_v)$-intertwining operator 
\[
\lambda_{\mathbf{c},\psi}: \mathfrak{S}(K_v) \to  I(\chi_v,0)
\]
defined by 
\[
\lambda_{\mathbf{c},\psi}(\varphi) (g) =( \omega_{\mathbf{c},\psi}(g)\varphi)(0).
\] 
For any $\varphi\in \mathfrak{S}(K_v)$ there is an associated section $\Phi(g,s)\in I(\chi_v,s)$ characterized by the 
properties
\begin{itemize}
\item $\Phi( \cdot ,0) = \lambda_{\mathbf{c},\psi}(\varphi)$,
\item $\Phi(g,s)$ is \emph{standard} in the sense that $\Phi(k,s)$ is independent of $s$ for all $k$ in the 
usual maximal compact subgroup of $\mathrm{SL}_2(F_v)$.  
\end{itemize}
It will always be clear from context whether $\Phi$ refers to a section of $I(\chi_v,s)$, or to the fixed CM type of $K$.

If  $v$ is a finite place of $F$,   let $\mathbf{1}_{\co_{K,v}}\in \mathfrak{S}(K_v)$ be the 
characteristic function of $\co_{K,v}$, and let 
\[
\Phi_{\mathbf{c},\psi}(g,s)\in I(\chi_v,s)
\] 
be the standard section 
satisfying $\Phi_{\mathbf{c},\psi}(\cdot ,0)= \lambda_{\mathbf{c},\psi}(\mathbf{1}_{\co_K,v}).$
If $v$ is archimedean,   let 
$
\varphi(x)=\mathrm{exp}(-2\pi |\mathbf{c}x\overline{x}|_v)
$
be the Gaussian, and let $\Phi_{\mathbf{c},\psi}$ be the corresponding standard
section satisfying $\Phi_{\mathbf{c},\psi} (\cdot,0)= \lambda_{\mathbf{c},\psi}(\varphi)$.  If 
\[
\mathrm{sign}(\mathbf{c})=\mathbf{c}/|\mathbf{c}|_v
\] 
denotes the sign of $\mathbf{c}$,  then
$\Phi_{\mathbf{c},\psi}$ is the  \emph{normalized standard section of weight 
$\mathrm{sign}(\mathbf{c})$}, characterized by the property 
\[
\Phi_{ \mathbf{c},\psi }\left( \left(\begin{matrix} \cos\theta& \sin \theta \\ -\sin\theta & \cos\theta  \end{matrix}\right),s \right) 
= e^{  \mathrm{sign}(\mathbf{c})\cdot i \theta} 
\]
for every $\theta\in \R$; see for example \cite[(4.29)]{kudla-rallis}.

For each $\alpha\in F_v^\times$ and  $\Phi\in I(\chi_v,s)$  define the local Whittaker function
\[
W_\alpha(g,s; \Phi,\psi) = 
\int_{F_v} \Phi\left( \left(\begin{smallmatrix}  0 & -1 \\ 1 & 0 \end{smallmatrix}\right)
 \left(\begin{smallmatrix} 1 & x\\ 0& 1\end{smallmatrix}\right) g, s \right) \psi_v (-\alpha x)\, dx,
\]
where $g\in \mathrm{SL}_2(F_v)$, and the Haar measure on $F_v$ is self-dual with respect to $\psi$.  
When $\Phi=\Phi_{\mathbf{c},\psi}$ as above we abbreviate
\begin{equation}\label{local whitt}
W_\alpha(g,s; \mathbf{c},\psi) = W_\alpha(g,s;\Phi_{\mathbf{c},\psi}, \psi).
\end{equation}
If we fix a $\delta \in F_v^\times$ and set $(\delta \psi)(x) =\psi(\delta x)$ then 
$\Phi_{\mathbf{c}, \delta \psi } =  \Phi_{\delta \mathbf{c},\psi}$
and
\begin{equation}\label{whitt shift}
W_{\alpha}( g,s; \mathbf{c} , \delta \psi ) 
= |\delta|_v^{1/2} \cdot W_{\delta \alpha}( g,s; \delta \mathbf{c} ,\psi)
\end{equation}
for all $\alpha\in F_v$.  Indeed, the first equality is clear from explicit formulas for the Weil representation, as in 
\cite[Chapter II.4]{kudla96}, and the second is clear from the first.

Now we switch to the global setting.  Let $\psi_\Q: \Q\backslash \Q_\A \to \C^\times$ be the 
usual additive character, whose  archimedean component  satisfies $\psi_\Q(x)=e^{2\pi i x}$ for all  $x\in \R$,
and whose nonarchimedean components are unramified.  Set  
\[
\psi_F(x) = \psi_\Q(\mathrm{Tr}_{F/\Q}(x)).
\]
Let 
\[
\chi:F_\A^\times\to \{\pm 1\}
\] 
be the composition of (\ref{general character}) with the product 
map $\{\pm 1\}^{\pi_0(F)} \to \{\pm 1\},$ so that $\chi=\prod_v\chi_v$, and let
$I(\chi,s) = \otimes_v I(\chi_v,s)$
be the representation of $\mathrm{SL}_2(F_\A)$ induced by the character $\chi$.
Given any $\mathbf{c}\in F_\A^\times$  we define a section of $I(\chi,s)$ by 
$
\Phi_{\mathbf{c},\psi_F} = \otimes_{v} \Phi_{\mathbf{c}_v,\psi_{F,v}}
$
and an Eisenstein series
\[
E(g,s;\mathbf{c},\psi_F) = \sum_{\gamma\in B(F)\backslash \mathrm{SL}_2(F)} \Phi_{\mathbf{c},\psi_F}(\gamma g,s)
\]
on $\mathrm{SL}_2(F_\A)$, where $B\subset\mathrm{SL}_2$ is the subgroup of upper triangular matrices.

Let
\[
\mathcal{H}_F = \{ x+iy\in F_\C : x,y\in F_\R, y\gg 0\}
\]
be the $F$-upper half plane.   A choice of isomorphism $F_\R\iso \R^n$, which we do not make, 
identifies $\mathcal{H}_F$ with a product of $n$ complex upper half planes.   
For $\tau=x+iy\in \mathcal{H}_F$ set
\[
g_\tau =\left(\begin{matrix} 1& x \\ & 1 \end{matrix}\right)
\left(\begin{matrix} y^{1/2} & \\ & y^{-1/2} \end{matrix}\right) \in \mathrm{SL}_2(F_\R),
\]
viewed as an element of $\mathrm{SL}_2(F_\A)$ with trivial nonarchimedean components, and set
(recall that $\mathfrak{d}_F$ is the different of $F/\Q$)
\[
\mathcal{E}(\tau,s;\mathbf{c},\psi_F) =\mathrm{N}(\mathfrak{d}_F)^{\frac{s+1}{2}} \cdot 
  \frac{ L(s+1,\chi)}{  \mathrm{Norm}_{F/\Q}(y)^{1/2} }
 \cdot E(g_\tau,s;\mathbf{c},\psi_F),
\]
where $L(s,\chi)=\prod_v L(s,\chi_v)$ is the Dirichlet $L$-function of $\chi$, including the $\Gamma$-factor 
\[
L(s,\chi_v)=\pi^{-(s+1)/2}\cdot  \Gamma\left(\frac{s+1}{2}\right)
\]
for archimedean $v$. Thus $\mathcal{E}(\tau,s;\mathbf{c},\psi_F)$ is a Hilbert modular form,  and admits a Fourier expansion
\[
\mathcal{E}(\tau,s;\mathbf{c},\psi_F) = \sum_{\alpha\in F} \mathcal{E}_\alpha(\tau,s;\mathbf{c},\psi_F)
\]
in which 
\[
\mathcal{E}_\alpha(\tau,s;\mathbf{c},\psi_F) =\mathrm{Norm}_{F/\Q}(y)^{-1/2}   \int_{F\backslash F_\A} 
E\left( \left(\begin{matrix} 1 & b \\ & 1\end{matrix}\right) g_\tau,s ; \mathbf{c}, \psi_F(-b\alpha)\right)\, db.
\]

Assume that $\mathbf{c}\in F_\A^\times$ satisfies $\mathbf{c}\co_F=\mathfrak{d}_F^{-1}$, 
$\mathbf{c}_v=1$ for every archimedean $v$, and $\chi(\mathbf{c})=-1$.
The second condition implies that $\mathcal{E}(\tau,s ;\mathbf{c} , \psi_F)$ has parallel weight one.  The
third condition implies that the $F_\A$-quadratic space $(K_\A,\mathbf{c} x\overline{x})$ is not the adelization
of any $F$-quadratic space, and so the Eisenstein series $\mathcal{E}(\tau,s;\mathbf{c},\psi_F)$ is \emph{incoherent}
in the sense of \cite{kudla97}.  In particular 
\[
\mathcal{E}(\tau,0;\mathbf{c},\psi_F)  =0
\]
 by \cite[Theorem 2.2]{kudla97}.  Strictly speaking, the notion of an incoherent Eisenstein 
 series  only makes  sense if $F$ is a field.  In general  we write 
 \[
 F= \prod_j F_j
 \]
 as a product  of totally real fields.  There are corresponding factorizations $K=\prod_j K_j$
 where each $K_j$ is a quadratic totally imaginary extension of $F_j$, and 
 \[
 \mathcal{H}_F = \prod_j \mathcal{H}_{F_j}
 \]
 where $\mathcal{H}_{F_j}$ is the $F_j$-upper half plane.  The element  $\mathbf{c}$ factors as
 $\mathbf{c}=\prod_j\mathbf{c}_j$, where each $\mathbf{c}_j\in F_{j\A}^\times \subset F_\A^\times$ 
 has trivial components away from the factor $F_{j\A}^\times$.  Similarly $\psi_F=\prod_j \psi_{F_j}$,
 and  there is a factorization of Eisenstein series
 \begin{equation}\label{eisenstein components}
 \mathcal{E}(\tau,s;\mathbf{c},\psi_F) = \prod_j \mathcal{E}_j(\tau_j, s; \mathbf{c}_j,\psi_{F_j}).
 \end{equation}
Each Eisenstein series in the factorization is then either  \emph{coherent} or \emph{incoherent}, 
depending on  whether $\chi(\mathbf{c}_j)=1$ or $-1$.  All incoherent factors vanish at $s=0$,
and as $\prod_j \chi(\mathbf{c}_j) = \chi(\mathbf{c})=-1$ there is at least one incoherent factor.

Recall that the fixed  CM type $\Phi$ has a distinguished element $\varphi^\mathrm{sp}:K\to \C$, which
determines a direct factor  $K^\mathrm{sp}$ of $K$ with maximal totally real subfield  $F^\mathrm{sp}$.  
Recall also that the restriction of $\varphi^\mathrm{sp}$ to $F$ determines an archimedean place denoted
$\infty^\mathrm{sp}$.
We will be restricting our attention to Eisenstein series 
 $\mathcal{E}(\tau,s;\mathbf{c},\psi_F)$ with $\mathbf{c}$ chosen so that all factors
 on the right hand side of (\ref{eisenstein components}) are coherent, except for the incoherent factor
 corresponding to  $F^\mathrm{sp}$.

\begin{Def}\label{Def:eisenstein}
Define  a Hilbert modular Eisenstein series of weight one
\[
\mathcal{E}_{\Phi}(\tau, s) = \sum_{\mathbf{c}\in \Xi} \mathcal{E}(\tau,s;\mathbf{c},\psi_F)
\]
where the sum is over the finite set $\Xi$ of $\mathrm{Nm}_{K/F}(\widehat{\co}_K^\times)$-orbits of 
$\mathbf{c}\in \A_F^\times$ satisfying 
\begin{itemize}
\item
$\mathbf{c}\co_F=\mathfrak{d}_F^{-1},$
\item
$\mathbf{c}_v=1$ for every archimedean $v$,
\item
for every factor $F_j$ of $F$
\[
\chi(\mathbf{c}_j) = \begin{cases}
1 & \mbox{if }F_j\not=F^\mathrm{sp} \\
-1 &\mbox{if }F_j=F^\mathrm{sp}.
\end{cases}
\]
\end{itemize}
\end{Def}

To understand the motivation behind the particular set $\Xi$, reconsider the collection
of Hermitian spaces $\mathcal{L}_B$ of Remark \ref{Rem:L_B}.  Thus $\mathcal{L}_B$
consists of all isomorphism classes of  Hermitian spaces $( L_B(A_0,A), \langle\cdot,\cdot\rangle_\CM )$
as   
\[
(A_0,A) \in (\mathcal{M}_{(1,0)} \times \mathcal{CM}_\Phi^\mathfrak{a})(\C)
\]
 varies.   Take $\mathfrak{a}=\co_F$,
and assume that $\mathfrak{s}=\mathfrak{d}_F^{-1}$ (which is the case if $K_0$ and $K$
have relatively prime discriminants, by Proposition \ref{Prop:s ideal}).  The elements of
$\mathcal{L}_B$ are alternately characterized as the  isomorphism classes of pairs $(L,H)$ in which 
$L$ is a projective $\co_K$-module of rank $1$,
$H$ is a $K$-valued $\co_K$-Hermitian form on $L$,  the ideal of $(L,H)$ is $\mathfrak{d}_F^{-1}$,
and $(L,H)$ is negative definite at $\infty^\mathrm{sp}$ and positive definite at all other archimedean places.
For any such $(L,H)$ there is a  $\beta\in F_\A^\times$  and  an isomorphism of $K_\A$-Hermitian spaces
\[
( L\otimes_{\co_K} K_\A ,H) \iso ( K_\A, \beta x\overline{y})
\] 
identifying $L\otimes_{\co_K} \widehat{\co}_K \iso \widehat{\co}_K$.   This $\beta$ must 
satisfy $\beta\co_F=\mathfrak{d}_F^{-1}$, be negative at $\infty^\mathrm{sp}$ and  positive at
all other archimedean places, and satisfy $\chi(\beta_j)=1$ for each factor $F_j$ of $F$.   The finite part of $\beta$
is well-defined up to multiplication by a norm from $\widehat{\co}_K^\times$, and each archimedean
component is well-defined  up to sign.   This makes clear the connection between 
$\Xi$ and $\mathcal{L}_B$: the elements of $\Xi$ arise by taking the $\beta$'s corresponding to 
elements of $\mathcal{L}_B$, and replacing the negative component at $\infty^\mathrm{sp}$ by
a positive component.  The corresponding $K_\A$-Hermitian spaces $(K_\A,\mathbf{c} x\overline{y})$
from which the Eisenstein series $\mathcal{E}(\tau, s;\mathbf{c},\psi_F)$ are constructed
are therefore incoherent at the factor $F^\mathrm{sp}$, and coherent at all other factors.
That is to say, the $K_{j\A}$-Hermitian space $(K_{j\A}, \mathbf{c}_j x\overline{y})$ arises
as the adelization of a $K_j$-Hermitian space if any only if $F_j\not= F^\mathrm{sp}$.


\subsection{Fourier coefficients}


Of course $\mathcal{E}_{\Phi}(\tau,0)=0$, and so we study the derivative at $s=0$, which 
 has a Fourier expansion
\[
\frac{d}{ds} \mathcal{E}_{\Phi}(\tau,s) \big|_{s=0}  =\sum_{\alpha\in F} b_{\Phi}(\alpha,y)\cdot  q^\alpha
\]
in which 
\[
q^\alpha = \mathrm{exp}(2\pi i \mathrm{Tr}_{F/\Q}(\alpha\tau) ).
\]
We will give an explicit formula for the coefficients, at least when $\alpha\in F^\times$, and compare them with the formulas
of Theorem \ref{Thm:degree formulas}.

For $\alpha\in F^\times$ and $\mathbf{c}\in \Xi$, define a finite set of places of $F$
\[
\mathrm{Diff} (\alpha,\mathbf{c}) = \{ v: \chi_v(\alpha\mathbf{c}) =-1\}.
\]
Note that every $v\in \mathrm{Diff} (\alpha,\mathbf{c})$ is nonsplit in $K$, and that there
 is a  disjoint union 
\[
\mathrm{Diff} (\alpha,\mathbf{c})  = \bigsqcup_j \, \{ \mbox{places }v \mbox{ of }F_j : \chi_v(\alpha\mathbf{c}) =-1\}.
\]
Our hypotheses on $\mathbf{c}$ imply that every set in the disjoint union has even cardinality, except for 
\[
\mathrm{Diff}^\mathrm{sp}(\alpha,\mathbf{c}) =  \{ \mbox{places }v \mbox{ of }F^\mathrm{sp} : \chi_v(\alpha\mathbf{c}) =-1\},
\]
which has odd cardinality. In particular   $\mathrm{Diff}(\alpha,\mathbf{c})$ has odd cardinality, and if it 
contains a unique place of $F$,  that place must lie on the factor $F^\mathrm{sp}$.

If $v$ is a finite place of $F$ and $\mathfrak{b}$ is a fractional $\co_{F,v}$-ideal, let 
\[
\rho_v(\mathfrak{b})
= \# \{ \mathfrak{B} \subset  \co_{K,v} : \mathfrak{B}\overline{\mathfrak{B}} = \mathfrak{b}\co_{K,v}\} .
\]
If $\mathfrak{b}$ is a fractional $\co_F$-ideal, set $\rho(\mathfrak{b})=\prod_v\rho_v(\mathfrak{b}_v)$, as in the 
introduction.  The following proposition follows from calculations of Yang \cite{yang05}.

\begin{Prop}\label{Prop:whitt calculation}
Suppose $\alpha\in F^\times$, let $d_{K/F}$ be the relative discriminant of $K/F$, 
and let $r$ denote the number of places  of $F$ ramified in $K$ (including the archimedean places).  
Suppose $\mathbf{c}\in\Xi$.
\begin{enumerate}
\item
If $\#\mathrm{Diff}(\alpha,\mathbf{c}) >1$ then $\ord_{s=0}\, \mathcal{E}_{\alpha}(\tau,s; \mathbf{c}, \psi_F) > 1$.
\item
If $\mathrm{Diff}(\alpha,\mathbf{c}) =\{\mathfrak{p}\}$ with $\mathfrak{p}$ finite prime of $F$ then 
\[
\frac{d}{ds} \mathcal{E}_{\alpha}(\tau,s;  \mathbf{c}, \psi_F) \big|_{s=0}   
 =  \frac{ - 2^{r-1 } }    {  \mathrm{N}(d_{K/F})^{1/2}} 
 \cdot  \rho(\alpha\mathfrak{d}_F\mathfrak{p}^{-\epsilon_\mathfrak{p}})  \cdot 
  \ord_{\mathfrak{p}}(\alpha\mathfrak{d}_F\mathfrak{p}) \cdot    \log(\mathrm{N}(\mathfrak{p})) \cdot q^\alpha 
\]
where $\epsilon_\mathfrak{p}=0$ if $\mathfrak{p}$  ramifies in $K$, and $\epsilon_\mathfrak{p}=1$ if 
$\mathfrak{p}$ is unramified in $K$.
\item
If $\mathrm{Diff}(\alpha,\mathbf{c}) =\{v\}$ with $v$ an archimedean place of $F$ then 
\[
\frac{d}{ds} \mathcal{E}_{\alpha}(\tau,s;  \mathbf{c}, \psi_F ) \big|_{s=0}  
 =   \frac{ - 2^{r-1}   }{\mathrm{N}(d_{K/F})^{1/2}} \cdot   \rho(\alpha\mathfrak{d}_F)  \cdot \beta_1(4\pi |y\alpha|_v) 
 \cdot q^\alpha.
\]
Recall that  $\beta_1(t)$ was defined by (\ref{beta}).
\end{enumerate}
\end{Prop}

\begin{proof}
Returning briefly to the local setting of (\ref{local whitt}), define the normalized local Whittaker function 
\[
W_{\alpha_v}^*(g_v ,s; \mathbf{c}_v,\psi_v) =  L(s+1,\chi_v) \cdot W_{\alpha_v} ( g_v ,s; \mathbf{c}_v ,\psi_v).
\]
Here $\psi_v$ is any local additive character.   The  Fourier coefficient  factors as a product 
\[
\mathcal{E}_\alpha(\tau ,s;\mathbf{c},\psi_F) 
= \mathrm{N}(\mathfrak{d}_F)^{(s+1)/2} \mathrm{Norm}_{F/\Q}(y)^{-1/2} \prod_v W^*_{ \alpha_v}(g_{\tau,v},s; \mathbf{c}_v,\psi_{F,v}).
\]
The character $\psi^\mathrm{unr}(x) = \psi_F(\mathbf{c} x)$
is  an unramified character of $F_\A^\times$, and  (\ref{whitt shift}) shows that
\begin{equation}\label{Fourier factor}
\mathcal{E}_\alpha(\tau,s;\mathbf{c},\psi_F) = \mathrm{N}(\mathfrak{d}_F)^{s/2}  \mathrm{Norm}_{F/\Q}(y)^{-1/2}
 \prod_v  W^*_{\delta_v\alpha_v}(g_{\tau,v},s; 1,  \psi_v^\mathrm{unr}).
\end{equation}

Let $v$ be a nonarchimedean place of $F$, fix a uniformizing parameter $\varpi\in F_v$,   let $f_v=\ord_v(d_{K/F})$,
and let $q_v=\#\co_{F,v}/(\varpi)$.    We now invoke \cite[Proposition 2.1]{yang05} and \cite[Proposition 2.3]{yang05}.
If  $\chi_v(\alpha \mathbf{c})=1$  then
 \begin{eqnarray*}\lefteqn{
 W^*_{\delta_v\alpha_v}(g_{\tau,v}, 0; 1,\psi_v^\mathrm{unr})  } \\
& = &    \chi_v(-1) \epsilon(1/2,\chi_v,\psi_v^\mathrm{unr}) \rho_v(\alpha\mathfrak{d}_F)
\cdot \begin{cases}
2 q_v^{-f_v/2} & \mbox{ if $v$ is ramified in $K$} \\
1&\mbox{ if $v$ is unramified in $K$.}
\end{cases}
 \end{eqnarray*}
 If  instead  $\chi_v(\alpha \mathbf{c})=-1$ then 
 $ W^*_{\delta_v\alpha_v}(g_{\tau,v}, s; 1,\psi_v^\mathrm{unr})$ vanishes at $s=0$, and 
\begin{align*}
\frac{d}{ds}W^*_{\delta_v\alpha_v} (g_{\tau,v} ,s; 1 ,\psi_v^\mathrm{unr}) \big|_{s=0}   
&=
   \chi_v(-1) \epsilon(1/2,\chi_v,\psi_v^\mathrm{unr}) \log (q_v )  \cdot 
 \frac{\ord_v(\alpha\mathfrak{d}_F)+1 }{2}  \\
&   \quad  \times \begin{cases}
 2q_v^{-f_v/2}  \cdot \rho_v(\alpha\mathfrak{d}_F)   & \mbox{ if $v$ is ramified in $K$} \\
\rho_v(\alpha\mathfrak{d}_F\mathfrak{p}_v^{-1}) &\mbox{ if $v$ is unramified in $K$}
\end{cases}
\end{align*}
where $\mathfrak{p}_v$ is the prime ideal associated to $v$.

Now suppose  $v$ is an archimedean place of $F$. In this case we cite \cite[Proposition 2.4]{yang05}. 
 If $\chi_v(\alpha\mathbf{c})=1$ then
\[
W^*_{\delta_v\alpha_v}(g_{\tau,v},0;  1 ,\psi_v^\mathrm{unr}) = 
2\chi_v(-1) \epsilon(1/2,\chi_v,\psi_v^\mathrm{unr}) \cdot y_v^{1/2} e^{2\pi i \alpha_v \tau_v}  .
\]
If $\chi_v(\alpha\mathbf{c})=-1$ then $W^*_{\delta_v\alpha_v}(g_{\tau,v},0; 1 ,\psi_v^\mathrm{unr})=0$ and 
\[
\frac{d}{ds}W^*_{\delta_v\alpha_v} (g_{\tau,v} ,s; 1 ,\psi_v^\mathrm{unr}) \big|_{s=0} 
=   \chi_v(-1) \epsilon(1/2,\chi_v,\psi_v^\mathrm{unr}) \cdot  y_v^{1/2}  e^{2\pi i \alpha_v\tau_v} \beta_1(4\pi |y\alpha|_v ).
\]

Everything now follows easily.  The above formulas show that when $v\in \mathrm{Diff}(\alpha,\mathbf{c})$
the $v$ factor on the right hand side of (\ref{Fourier factor}) vanishes at $s=0$, and so the order of vanishing of 
$\mathcal{E}_\alpha(\tau,s;\mathbf{c},\psi_F)$ is at least  $\#\mathrm{Diff}(\alpha,\mathbf{c})$.
If $\mathrm{Diff}(\alpha,\mathbf{c})=\{w\}$ then  differentiating (\ref{Fourier factor}) at $s=0$ shows that
\begin{eqnarray*}
\frac{d}{ds} \mathcal{E}_{\alpha}(\tau,s;  \mathbf{c}, \psi_F) \big|_{s=0}   
&=&   \mathrm{Norm}_{F/\Q}(y)^{-1/2} \cdot 
\frac{d}{ds}  W^*_{\delta_w\alpha_w}(g_{\tau,w},s; 1,  \psi_w^\mathrm{unr}) \big|_{s=0} \\
& &   \times 
 \prod_{v\not=w}  W^*_{\delta_v\alpha_v}(g_{\tau,v},0; 1,  \psi_v^\mathrm{unr}),
\end{eqnarray*}
and the claim follows from the formulas above and the root number calculation 
\[
\prod_v\epsilon(1/2,\chi_v,\psi_v^\mathrm{unr}) = \chi(\mathbf{c}) \cdot \prod_v\epsilon(1/2,\chi_v,\psi_{F,v}) = -1
\]
(the first equality follows from \cite[(3.29)]{kudla-tate},  the second follows from the functional equation of $L(s,\chi)$,
which shows that $\epsilon(1/2,\chi)=1$).
\end{proof}

By the first claim of the proposition, for any $\alpha\in F^\times$ we have
\begin{align*}
b_\Phi (\alpha,y) \cdot q^\alpha   &   =  
 \sum_{\mathbf{c} \in \Xi }
\mathcal{E}'_{\alpha}(\tau,0;  \mathbf{c}, \psi_F)  \\
& = \sum_{v}  \sum_{ \substack{  \mathbf{c} \in \Xi   \\   \mathrm{Diff}(\alpha,\mathbf{c}) =\{ v \}    } }
\mathcal{E}'_{\alpha}(\tau,0;  \mathbf{c}, \psi_F) 
\end{align*}
where the outer sum is over all places $v$ of $F$.  This sum is unchanged if we restrict further to 
places $v$ of $F^\mathrm{sp}$ which are nonsplit in $K^\mathrm{sp}$, as these are the only places 
for which the relation $\mathrm{Diff}(\alpha,\mathbf{c}) =\{ v \} $ can ever hold.

\begin{Cor}\label{Cor:final fourier}
Suppose $\alpha\in F^\times$ and $y\in F_\R^{\gg 0}$.
\begin{enumerate}
\item
If $\alpha$ is totally positive then 
\begin{eqnarray*}
b_{\Phi}(\alpha,y)  
 =     \frac{-2^{r -1 } }{  \mathrm{N}( d_{K/F})^{1/2}} \cdot 
\sum_\mathfrak{p}  
  \ord_{\mathfrak{p}}(\alpha\mathfrak{d}_F\mathfrak{p})  \cdot
    \rho(\alpha\mathfrak{d}_F\mathfrak{p}^{-\epsilon_\mathfrak{p}}) \cdot  \log(\mathrm{N}(\mathfrak{p}))  
\end{eqnarray*}
where the sum is over all primes $\mathfrak{p}$ of $F^\mathrm{sp}$ nonsplit in $K^\mathrm{sp}$.
In particular $b_{\Phi}(\alpha,y)$ is independent of $y$.
\item
If $\alpha$ is negative at exactly one archimedean place, $v$, of $F$, and if this $v$ lies on the  factor $F^\mathrm{sp}$, 
then 
\[
b_{\Phi}(\alpha,y)  =   \frac{-2^{r - 1 } }{  \mathrm{N}( d_{K/F})^{1/2}} \cdot 
  \rho(\alpha\mathfrak{d}_F)  \cdot  \beta_1(4\pi |y\alpha|_v).
\]
\item
In all other cases $b_{\Phi}(\alpha,y)=0$.
\end{enumerate}
\end{Cor}

\begin{proof}
Suppose first that $\alpha$ is totally positive, so that $\mathrm{Diff}(\alpha,\mathbf{c})$ contains only finite 
places of $F$.  Proposition \ref{Prop:whitt calculation} implies
\[
b_{\Phi}(\alpha,y)  = \frac{ - 2^{r-1 } }    {  \mathrm{N}(d_{K/F})^{1/2}}  
\sum_\mathfrak{p} \sum_{ \substack {\mathbf{c} \in \Xi \\  \mathrm{Diff}(\alpha,\mathbf{c}) = \{\mathfrak{p}\} }  }
 \rho(\alpha\mathfrak{d}_F\mathfrak{p}^{-\epsilon_\mathfrak{p}})  \cdot 
  \ord_{\mathfrak{p}}(\alpha\mathfrak{d}_F\mathfrak{p}) \cdot    \log(\mathrm{N}(\mathfrak{p}))
\]
where the first sum is over  all primes  of $F^\mathrm{sp}$ that are nonsplit in $K^\mathrm{sp}$.  Obviously,
we may further restrict to those $\mathfrak{p}$  for which 
 $\rho(\alpha\mathfrak{d}_F\mathfrak{p}^{-\epsilon_\mathfrak{p}})\not=0$, and for each $\mathfrak{p}$ 
there is a unique choice of $\mathbf{c}\in \Xi$ for which $\mathrm{Diff}(\alpha,\mathbf{c})=\{\mathfrak{p}\}$. 
This proves the first claim, and the proofs of the remaining claims are similar.
\end{proof}

Comparing Theorem \ref{Thm:degree formulas} and Corollary \ref{Cor:final fourier} proves the following result.

\begin{Thm}\label{Thm:degree-fourier}
Assume the discriminants of $K_0/\Q$ and $F/\Q$ are odd and relatively prime. 
If $\alpha\in F^\times$ and $y\in F_\R^{\gg 0}$ then
\[
\widehat{\mathrm{deg}}\, \widehat{\mathtt{Z}}_{\Phi}(\alpha,y)
 =       -   \frac{ h(K_0) } {w(K_0) } \cdot   \frac{  \sqrt{\mathrm{N}( d_{K/F}) }  } { 2^{r -1 }    [K^\mathrm{sp}:\Q]}
 \cdot  b_{\Phi}(\alpha,y) .
\]
\end{Thm}

Let  $i_F:\mathcal{H}\to \mathcal{H}_F$ be the diagonal embedding of the usual complex upper half plane into 
$\mathcal{H}_F$.  The restriction $\mathcal{E}_{\Phi}(i_F(\tau),s)$ 
of $\mathcal{E}_{\Phi} (\tau,s)$
to $\mathcal{H}$ vanishes at $s=0$, and the derivative has a Fourier expansion
\[
\frac{d}{ds} \mathcal{E}_{\Phi}(i_F(\tau),s) \big|_{s=0} = \sum_{m\in \Z} c_{\Phi}(m,y)  \cdot  q^m
\]
in which
\[
c_{\Phi}(m,y) = \sum_{ \substack{ \alpha\in F \\ \mathrm{Tr}_{F/\Q}(\alpha) =m  }  } b_{\Phi}(\alpha,y).
\]
Here $\tau=x+iy\in\mathcal{H}$ and $q=\exp(2\pi i \tau)$, as usual.

\begin{Cor}
Assume the discriminants of $K_0/\Q$ and $F/\Q$ are odd and relatively prime. 
If $F$ is a field and $m$ is  nonzero  then
\[
I(\mathcal{X}_\Phi : \mathcal{Z}(m) )  +  \green(m,y, \mathcal{X}_\Phi ) =
 -   \frac{ h(K_0) }{w(K_0) } \cdot 
 \frac{  \sqrt{\mathrm{N}( d_{K/F}) }  } { 2^{r -1 } [K:\Q]} \cdot  c_{\Phi}(m,y) 
\]
for all $y\in \R^{>0}$.
\end{Cor}

\begin{proof}
Theorems \ref{Thm:fundamental decomp I} and \ref{Thm:fundamental decomp II}  imply
\[
I(\mathcal{X}_\Phi : \mathcal{Z}(m) )  +  \green(m,y, \mathcal{X}_\Phi ) 
= \sum_{ \substack{ \alpha\in F \\ \mathrm{Tr}_{F/\Q}(\alpha) =m  }  } \widehat{\deg}\, \widehat{\mathtt{Z}}_\Phi(\alpha,y),
\]
and so the claim is clear from Theorem \ref{Thm:degree-fourier}.
\end{proof}

\bibliographystyle{plain}

\end{document}